\documentclass[a4paper,10pt,leqno]{amsart}
\title[Nielsen realization and manifold models]{On  Nielsen realization  and manifold models for classifying spaces}
              \author{James F. Davis}
              \address{Department of Mathematics\\
              Indiana University\\
              Rawles Hall\\
              831 East 3rd St\\
              Bloomington, IN 47405\\
              U.S.A.}
                 \email{jfdavis@indiana.edu}
              \urladdr{https://jfdmath.sitehost.iu.edu}
              \author{Wolfgang L\"uck}
        \address{Mathematical Institute of the University of  Bonn\\ 
                Endenicher Allee 60\\
                53115 Bonn, Germany}
         \email{wolfgang.lueck@him.uni-bonn.de}
         \urladdr{http://www.him.uni-bonn.de/lueck}
         \date{January 2024}
    \keywords{classifying spaces for proper actions, existence and uniqueness of manifold models}
    \makeatletter
\@namedef{subjclassname@2020}{%
  \textup{2020} Mathematics Subject Classification}
\makeatother
\subjclass[2020]{57N99,55R35,18F25}

\usepackage{hyperref} \usepackage{color} \usepackage{pdfsync} \usepackage{calc}
\usepackage{enumerate,amssymb} \usepackage{graphicx}
\usepackage[arrow,curve,matrix,tips,2cell]{xy} \usepackage{amssymb}
\SelectTips{eu}{10}\UseTips\UseAllTwocells \DeclareMathAlphabet{\matheurm}{U}{eur}{m}{n}
\newcommand{\Groupoids}{\matheurm{Groupoids}}

\newcommand{\Spectra}{\matheurm{Spectra}}

\newcommand{\tp}{$\,{}^{\prime}$}

\DeclareMathOperator{\aut}{aut}

\DeclareMathOperator{\colim}{colim} 
\DeclareMathOperator{\cone}{cone}

\DeclareMathOperator{\id}{id} \DeclareMathOperator{\im}{im}

\DeclareMathOperator{\ind}{ind} \DeclareMathOperator{\Idem}{Idem}
\DeclareMathOperator{\Isom}{Isom}

 \DeclareMathOperator{\Out}{Out}
 \DeclareMathOperator{\per}{per}
\DeclareMathOperator{\pr}{pr}

 \DeclareMathOperator{\UNil}{UNil}
 
\DeclareMathOperator{\Wh}{Wh}

  \newcommand{\IQ}{\mathbb{Q}}
\newcommand{\IR}{\mathbb{R}}  
  
  \newcommand{\IZ}{\mathbb{Z}}
\newcommand{\Z}{\mathbb{Z}}

\newcommand{\cala}{\mathcal{A}} \newcommand{\calb}{\mathcal{B}}

 \newcommand{\calf}{\mathcal{F}}
 \newcommand{\calg}{\mathcal{G}}
\newcommand{\calh}{\mathcal{H}} 
 
 \newcommand{\calm}{\mathcal{M}}

\newcommand{\calp}{\mathcal{P}} 
 \newcommand{\cals}{\mathcal{S}}

\newcommand{\calv}{\mathcal{V}} \newcommand{\calvI}{\mathcal{V}_{\operatorname{I}}}
\newcommand{\calvII}{\mathcal{V}_{\operatorname{II}}}

 \newcommand{\bfE}{\mathbf{E}}

  \newcommand{\bfK}{\mathbf{K}}
\newcommand{\bfL}{\mathbf{L}}

 \newcommand{\NK}{N\!K}

\newcommand{\pt}{\{\bullet\}}

\newcommand{\EGF}[2]{E_{#2}(#1)} 
\newcommand{\bub}[1]{\underline{B}#1} \newcommand{\eub}[1]{\underline{E}#1}
\newcommand{\edub}[1]{\underline{\underline{E}}#1}

   \newcommand{\calfin}{{\mathcal F}{\mathcal I}{\mathcal N}}
   \newcommand{\calvcyc}{{\mathcal V}{\mathcal C}{\mathcal Y}}

   \newcommand{\wh}[1]{{\widehat{#1}}} \newcommand{\wt}[1]{{\widetilde{#1}}}

   \newcounter{commentcounter}

   \theoremstyle{plain} \newtheorem{theorem}{Theorem}[section]
    \newtheorem{lemma}[theorem]{Lemma}
   \newtheorem{corollary}[theorem]{Corollary}
   \newtheorem{proposition}[theorem]{Proposition}

   \newtheorem*{theorem*}{Theorem} \newtheorem*{theoremA*}{Theorem A}
   \newtheorem*{theoremB*}{Theorem B} \newtheorem*{NRQ}{Neilsen Realization Question}
    \newtheorem*{MMQ}{Manifold Model
     Question}

   \theoremstyle{definition} \newtheorem{definition}[theorem]{Definition}

    \newtheorem{remark}[theorem]{Remark}
   \newtheorem{notation}[theorem]{Notation} 
    \newtheorem*{definition*}{Definition}

   \theoremstyle{remark}

   \makeatletter\let\c@equation=\c@theorem\makeatother

    \newcommand{\CAT}{\operatorname{CAT}}
   \newcommand{\CATzero}{\operatorname{CAT(0)}}

   \newcommand{\version}[1] %marks the date of last editing and compilation
   {\begin{center} last edited on #1\\
       last compiled on \today\\
       name of tex-file: \jobname
     \end{center}}

\begin{document}

     \begin{abstract}
       We consider the problem of whether, for a given virtually torsionfree discrete
       group $\Gamma$, there exists a cocompact proper topological $\Gamma$-manifold,
       which is equivariantly homotopy equivalent to the classifying space for proper
       actions.  This problem is related to Nielsen Realization.  We will make the
       assumption that the expected manifold model has a zero-dimensional singular
       set. Then we solve the problem in the case, for instance, that $\Gamma$ contains a
       normal torsionfree subgroup $\pi$ such that $\pi$ is hyperbolic and $\pi$ is the
       fundamental group of an aspherical closed manifold of dimension greater or equal to
       five and $\Gamma/\pi$ is a finite cyclic group of odd order.
     \end{abstract}

     \maketitle

     \newlength{\origlabelwidth}\setlength\origlabelwidth\labelwidth

%%%%%%%%%%%%%%%%%%%%%%%%%%%%%%%%%%%%%%%%%%%%%%%%%%%%%%%%%%%%%%%%%%%%
     %%%%%%%%%%%%%%%%%%%%%%%%%% Introduction %%%%%%%%%%%%%%%%%%%%%%%%%%%%%%%%
%%%%%%%%%%%%%%%%%%%%%%%%%%%%%%%%%%%%%%%%%%%%%%%%%%%%%%%%%%%%%%%%%%%%

     \typeout{------------------- Introduction -----------------}

     \section{Introduction}%
     \label{sec:introduction}

     If a group $G$ acts effectively on a manifold $X$ with fundamental group $\pi$, then
     there is a short exact sequence
     \begin{equation} \label{group_extension_intro} 1 \to \pi \xrightarrow{i} \Gamma
       \xrightarrow{p} G \to 1
     \end{equation}
     and a group action of $\Gamma$ on the universal cover $\widetilde X$ so that the
     action of $\Gamma/\pi$ on $\widetilde X/\pi$ recovers the $G$-action on $X$.  (Here
     $\Gamma$ is the subgroup of the homeomorphism group of $\widetilde X$ given by lifts
     of the elements of $G$.)

     This paper makes progress on the two following interrelated questions.  We will
     discuss these questions and then state our results.

     \begin{NRQ} If $X$ is a closed aspherical manifold with fundamental group $\pi$, can
       any group monomorphism $\phi : G \to \Out(\pi)$ from a finite group to the outer
       automorphism group of $\pi$ be realized by a $G$-action on $X$?
     \end{NRQ}

\begin{MMQ}
  Given a closed aspherical manifold $X$ with fundamental group $\pi$ and dimension $d$,
  and a short exact sequence
  \[
    1 \to \pi \to \Gamma \to G \to 1
  \]
  with $G$ finite, does there exist a $d$-dimensional manifold model for $\eub{\Gamma}$,
  the classifying space for proper $\Gamma$-actions?
\end{MMQ}

The Nielsen Realization Question was raised by Nielsen for 2-manifolds, and was answered
affirmatively by Kerckhoff~\cite{Kerckhoff(1983)}.  The answer to the Nielsen Realization
Question is also yes for closed Riemannian manifolds with constant negative sectional
curvature (see Subsection~\ref{subsec:hyperbolic_NRQ_and_MMQ}).

In considering the Nielsen Realization Question, the first step is to see if, given
$\phi$, there is an extension~\eqref{group_extension_intro} realizing $\phi$.  There is a
cohomological obstruction in $H^3(G;Z(\pi))$ to the existence of the
extension~\cite[Theorem IV.8.7]{MacLane(1963)} and, if an extension realizing $\phi$
exists, $H^2(G;Z(\pi))$ classifies the extensions~\cite[Theorem IV.8.8]{MacLane(1963)}.
Here $Z(\pi)$ is the center of the group $\pi$.  Raymond and
Scott~\cite{Raymond-Scott(1977)} gave a negative answer to the Nielsen Realization
Question, by giving examples of $(X,\phi)$ where the group extension does not exist.
Block and Weinberger~\cite{Block-Weinberger(2008)} gave negative answers to the Nielsen
Realization Question where the center of $\pi$ is trivial.  However, Nielsen's original
question concerned surfaces of genus $>1$, so it is worth noting that there are no
counterexamples known when $X$ is negatively curved, or more generally when $\pi$ is a
hyperbolic group (in the sense of Gromov).  An affirmative answer was given
in~\cite[Remark 1.21]{Lueck(2022_Poincare_models)} when $\dim X \geq 5, \pi= \pi_1X$ is a
hyperbolic group, and the extension $\Gamma$ of $\pi$ by $G$ realizing $\phi$ is
torsionfree.  This generalized the analogous result of Farrell and
Jones~\cite{Farrell-Jones(1998)} in the case where $X$ is a Riemannian manifold of
negative curvature with $\dim X \geq 5$.

The answer to the Nielsen Realization Question is yes for closed Riemannian manifolds with
constant negative sectional curvature (see Subsection~\ref{subsec:hyperbolic_NRQ_and_MMQ}).

Recall that for a discrete group $\Gamma$, \emph{a model for $\eub{\Gamma}$} is a
$\Gamma$-space $M$ which is a $\Gamma$-CW-complex so that for every finite subgroup $H$,
the fixed point set $M^H$ is contractible and for every infinite subgroup $H$, the fixed
point set $M^H$ is empty.  $\eub{\Gamma}$ is the classifying space for proper actions in
the sense that if $Y$ is a proper $\Gamma$-CW-complex (i.e.~$\Gamma$-CW-complex with
finite isotropy), there is a $\Gamma$-map $Y \to \eub{\Gamma}$, unique up to
$\Gamma$-homotopy.  For a survey on $\eub{\Gamma}$ we refer to~\cite{Lueck(2005s)}.

A \emph{manifold model for $\eub{\Gamma}$} is simply a model $M$ for $\eub{\Gamma}$, so
that $M$, ignoring the group action, is a topological manifold.  (One could also include
the hypothesis that $M^H$ is a submanifold for non-trivial finite subgroups $H$, but we
are interested in the case where the singular set is discrete, so this distinction is not
relevant for us).  A model $M$ for $\eub{\Gamma}$ is \emph{cocompact} if $M/\Gamma$ is
compact.  In the statement of the Manifold Model Question we could have replaced the words
``$d$-dimensional'' with ``cocompact'' and had an equivalent question
(see~\ref{subsec:Cocompact_and_d-dimensional_models}).

Counterexamples to the Manifold Model Question have been given by Davis and Leary~\cite{Davis-Leary(2003asph)}
and by Block and Weinberger~\cite[Theorem~1.5]{Block-Weinberger(2008)}.  However, to the best of our knowledge, there are no
counterexamples known in three cases of interest: (1) if the normalizer of each nontrivial
finite subgroups of $\Gamma$ is finite, or (2) if $\eub \Gamma$ has a model which is a
finite $\Gamma$-CW-complex, or (3) if $\pi$ is a hyperbolic group.

The answer to the Manifold Model Question is yes for closed Riemannian manifolds with
constant negative sectional curvature (see Subsection~\ref{subsec:hyperbolic_NRQ_and_MMQ}).

The Borel Conjecture for a closed aspherical manifold $X$ states that any homotopy
equivalence $N \to X$ where $N$ is a closed manifold is homotopic to a homeomorphism.  It
has been proven for many manifolds $X$, but is open in general
(see~\cite{Weinberger(2023)} and~\cite{Lueck(2022book)} for a discussion).

An affirmative answer to the Manifold Model Question implies an affirmative answer to the
Nielsen Realization Question in the following sense.  Suppose $X$ is a closed aspherical
manifold with fundamental group $\pi$ and $\phi: G \to \Out(\pi)$ is a group monomorphism
with $G$ finite.  Suppose, in addition, that the Borel Conjecture holds for $X$ and that
$\phi$ is realized by a group extension
\[
  1 \to \pi \to \Gamma \to G \to 1.
\]
If there is a cocompact manifold model $M$ for $\eub \Gamma$, then $M/\pi$ is a closed
manifold with a $G$-action realizing $\phi$, and the Borel Conjecture gives a
homeomorphism between $X$ and $M/\pi$ inducing the identity on the fundamental group.
 
Under very special circumstances an affirmative answer to the Nielsen Realization Question
implies an affirmative answer to the Manifold Model Question, see
Section~\ref{subsec:NRQ_and_MMQ}.

For the rest of the introduction we focus on the Manifold Model Question.  We note that
the Manifold Model Question is an existence question. The corresponding uniqueness
question is: are two $d$-dimensional manifold models for $\eub{\Gamma}$ equivariantly
homeomorphic?

The simplest case of the existence and uniqueness questions is when $\Gamma$ is
torsionfree, equivalently when $\Gamma$ acts freely on $M$.  The uniqueness question is
the famous Borel Conjecture. The existence question was solved when $\dim X \geq 5$ and
$X$ is negatively curved in~\cite{Farrell-Jones(1998)} and extended to the case where
$\dim X \geq 5$ and $\pi$ is hyperbolic in~\cite{Lueck(2022_Poincare_models)}.

The next level in complexity (compared to free actions) is the pseudo-free case.  A
$\Gamma$-space $M$ is \emph{pseudo-free} if the singular set
$M^{> 1} = \{x \in M \mid \Gamma_x \not= 1\}$ is discrete, or, equivalently the
$\Gamma$-space $M^{>1}$ is the disjoint union of its $\Gamma$-orbits. If $M$ is
pseudo-free model for $\eub{\Gamma}$ and $H$ is a non-trivial finite subgroup, then $M^H$
is a point, fixed by its normalizer $N_\Gamma H$, hence the normalizer is finite.
Conversely, Proposition 2.3 of~\cite{Connolly-Davis-Khan(2015)} asserts that if $\Gamma$
is a virtually torsionfree group where the normalizers of non-trivial finite subgroups are
finite, and if $\Gamma$ acts properly and cocompactly on a contractible manifold, then the
action is pseudo-free.  In summary, a cocompact manifold model for $\eub{\Gamma}$ is
pseudo-free if and only if the normalizer of each non-trivial finite subgroup is finite.
Thus a geometric condition is equivalent to an algebraic condition.

This is our basic assumption in this paper.  The uniqueness question in this case was
studied extensively in~\cite{Connolly-Davis-Khan(2014H1)}
and~\cite{Connolly-Davis-Khan(2015)}.  We improve some of the techniques from these papers
and extend their uniqueness results.

% The Manifold Model Question was answer affirmatively in the pseudo-free case when $G$
% has order 2, the dimension of $X$ is greater than 4 and the Farrell-Jones Conjecture
% holds for $\Gamma$ prior to our work by the first author, but was never published.

A question related to the manifold model question was posed by Brown~\cite[page
32]{Brown(1979)}.  It asks that, given extension~\eqref{group_extension_intro} with $G$
finite, if $\eub{\pi}$ has a $d$-dimensional CW-model, then does $\eub{\Gamma}$ have a
$d$-dimensional CW-model? The paper~\cite{Lueck(2022_Poincare_models)} studied this
question in the pseudo-free case and the results of that paper are a key input for our
paper.

A recent book that discussed topics connected to the themes of this paper
is~\cite{Weinberger(2023)}.

%%%%%%%%%%%%%%%%%%%%%%%%%%%%%%%%%%%%%%%%%%%%%%%%%%%%%%%%%%%%%%%%%%%%%%%%%%%%%%%%%%%

\subsection{A special case}%
\label{sec:A_special_case}

As an illustration we state a special case of our main theorem. Recall that $E\Gamma$ is a
free $\Gamma$-$CW$-complex, which is contractible after forgetting the $\Gamma$-action,
or, equivalently, $E\Gamma \to B\Gamma := E\Gamma/\Gamma$ is the universal principal
$\Gamma$-bundle. Recall that $\eub{\Gamma}$ is a $\Gamma$-$CW$-complex such that
$\eub{\Gamma}^H$ is contractible for every finite subgroup $H \subseteq \Gamma$ and all
its isotropy groups are finite, or, equivalently $\eub{\Gamma}$ is the classifying space
for proper $\Gamma$-actions.  Two models for $E\Gamma$ or for $\eub{\Gamma}$ are
$\Gamma$-homotopy equivalent.

\begin{notation}\label{not_delta-spaces_intro}
  Let $\calm$ be a complete system of representatives of the conjugacy classes of maximal
  finite subgroups of $\Gamma$.  Put
  \begin{eqnarray*}
    \partial E\Gamma & := & \coprod_{F \in \calm} \Gamma \times_F EF;
    \\
    \partial \eub{\Gamma} & := & \coprod_{F \in \calm} \Gamma /F;
    \\
    \partial B\Gamma & := & \coprod_{F \in \calm} BF;
    \\
    \bub{\Gamma} & := &\eub{\Gamma/\Gamma}.
  \end{eqnarray*}            
\end{notation}

Recall that a virtually cyclic group is finite, surjects onto the infinite cyclic group
with finite kernel (type I), or surjects onto the infinite dihedral group with finite
kernel (type II).

We may impose some of the following conditions on a group $\Gamma$.
  
\begin{definition}[Conditions on $\Gamma$]\label{def:conditions_on_Gamma_intro}\
  \begin{itemize}

  \item[(M)] Every non-trivial finite subgroup of $\Gamma$ is contained in a unique
    maximal finite subgroup;

  \item[(NM)] If $F$ is a non-trivial maximal finite subgroup, then its normalizer
    satisfies $N_{\Gamma}F = F$;

  \item[(OH)] The composite
     
    \[
      H_d^{\Gamma}(E\Gamma,\partial E\Gamma) \xrightarrow{\partial}
      H_{d-1}^{\Gamma}(\partial E\Gamma) \xrightarrow{\cong} \bigoplus_{F \in \calm}
      H^F_{d-1}(EF) \to H^F_{d-1}(EF)
    \]
    of the boundary map, the inverse of the obvious isomorphism and the projection to the
    summand of $F \in \calm$ is surjective for all $F \in \calm$;

  \item[(F)] If $H \subseteq \Gamma$ is finite and non-trivial, then $N_{\Gamma}H$ is
    finite;

  \item[(V)] Every infinite virtually cyclic subgroup lies in a unique maximal infinite
    virtually cyclic subgroup;

  \item[(NV)] Every maximal infinite virtually cyclic subgroup $V$ satisfies
    $N_{\Gamma}V = V$;

  \item[(V$_{\operatorname{II}}$)] Every virtually cyclic subgroup of type {II} lies in a
    unique maximal virtually cyclic subgroup of type {II};

  \item[(NV$_{\operatorname{II}}$)] Every maximal virtually cyclic subgroup $V$ of type
    {II} satisfies lies $N_{\Gamma}V = V$.
  \end{itemize}
\end{definition}

\begin{theorem}[Oriented manifold models]\label{the:special_case_of_ultiimate_theorem_intro}
  Suppose there is a short exact sequence of groups
  \[
    1 \to \pi \xrightarrow{i} \Gamma \xrightarrow{p} G \to 1
  \]
  with $G$ finite.
 
  Suppose that the following conditions are satisfied:

  \begin{itemize}
    
  \item There exists a closed $d$-dimensional oriented manifold, which is homotopy
    equivalent to $B\pi$;

  \item The natural number $d$ satisfies $d \ge 5$;
    
  \item The group $\pi$ is hyperbolic;
        
  \item Every non-trivial finite subgroup of $\Gamma$ is odd order cyclic;

  \item The group $\Gamma$ satisfies conditions (M), (NM), and (OH), see
    Definition~\ref{def:conditions_on_Gamma_intro}.

  \end{itemize}

  Then:

  \begin{enumerate}

  \item\label{the:special_case_of_ultiimate_theorem_intro:existence} There exists a proper
    cocompact oriented $d$-dimensional topological manifold $M$, which is a model for
    $\eub{\Gamma}$;

  \item\label{the:special_case_of_ultiimate_theorem_intro:pseuso-free} Any
    $\Gamma$-manifold appearing in
    assertion~\eqref{the:special_case_of_ultiimate_theorem_intro:existence} is
    pseudo-free;

  \item\label{the:special_case_of_ultiimate_theorem_intro:uniqueness} Any two
    $\Gamma$-manifolds appearing in
    assertion~\eqref{the:special_case_of_ultiimate_theorem_intro:existence} are
    $\Gamma$-homeomorphic.

  \end{enumerate}
\end{theorem}

If we require that our manifold model for $\eub{\Gamma}$ is pseudo-free, then the
conditions (M), (NM), and (OH) are automatically satisfied as explained
in~\cite[Lemma~1.9]{Lueck(2022_Poincare_models)}.  Hence these conditions have to appear
in Theorem~\ref{the:special_case_of_ultiimate_theorem_intro}.

%%%%%%%%%%%%%%%%%%%%%%%%%%%%%%%%%%%%%%%%%%%%%%%%%%%%%%%%%%%%%%%%%%%%%%%%%%%%%%

\subsection{Slice manifold systems and slice manifold models}%
\label{subsec:Slice_manifold_systems_and_slice_manifold_models_intro}

\begin{definition}
  Let $F$ be a nontrivial finite group.  A \emph{Swan complex of type $(F,d-1)$} is a
  $(d-1)$-dimensional free $F$-$CW$-complex $S_F$ such that $S_F$, after forgetting the
  $F$-action, is homotopy equivalent to the sphere $S^{d-1}$.  The Swan complex is
  \emph{oriented}, if we have chosen a generator $[S_F]$ for the infinite cyclic group
  $H_{d-1}(S_F)$.
\end{definition}

Let $\calm$ be a complete system of representatives of the conjugacy classes of maxi\-mal
finite subgroups of $\Gamma$.  The following definition is taken
from~\cite[Definition~3.1]{Lueck(2022_Poincare_models)}.

\begin{definition}\label{def:free_d-dimensional_slice_system}
  A \emph{$d$-dimensional free slice system} $\cals = \{S_F \mid F\in \calm\}$, or just
  \emph{slice system}, consists of a Swan complex $S_F$ of type $(F,d-1)$ for every
  $F \in \calm$.  We call $\cals$ \emph{oriented}, if each Swan complex is oriented.
\end{definition}

We need the following manifold version of it.

 \begin{definition}\label{def:manifold_slice_system}
   A \emph{$d$-dimensional free slice mani\-fold system} or just \emph{slice mani\-fold
     system} $\cals = \{S_F \mid F\in \calm\}$ is a $d$-dimensional free slice system
   $\cals = \{S_F \mid F\in \calm\}$ so that each $S_F$, after forgetting the $F$-action,
   is homeomorphic to $S^{d-1}$.
 \end{definition}

\begin{remark}\label{rem:Swan_complexes}
  Swan complexes were introduced in~\cite{Swan(1960b)}.  For a Swan complex of type
  $(F,d-1)$, the Lefschetz Fixed Point Theorem implies that if $d-1$ is even, that $F$ is
  cyclic of order 2 and acts reversing orientation, and that if $d-1$ is odd, then $F$
  acts preserving orientation.  Any two Swan complexes of type $(C_2,2k)$ are
  equivariantly homotopy equivalent.
 
  Now assume $d-1$ is odd.  There is a Swan complex of type $(F,d-1)$ if and only if $F$
  has periodic cohomology with period $d$ (see~\cite[Proposition 4.1]{Swan(1960b)}), which
  is equivalent to $H_{d-1}(BF)$ being cyclic of order $|F|$ (see~\cite[Proposition
  11.1]{Cartan-Eilenberg(1956)}).  There exists a Swan complex $(F, d-1)$ for some $d$ if
  and only if the sylow $2$-subgroup is cyclic or generalized quaternionic and for $p$ odd
  the sylow $p$-subgroups are cyclic (see~\cite[Theorem 11.6]{Cartan-Eilenberg(1956)}).
 
  Let $S_F$ be an oriented Swan complex of type $(F,d-1)$.  Let
  $[S_F/F] \in H_{d-1}(S_F/F)$ be chosen so that the covering map $S_F \to S_F/F$ sends
  $[S_F]$ to $|F|\cdot [S_F/F]$.  Let $c_F \colon S_F \to EF$ and
  $\bar c_F \colon S_F/F \to BF$ be classifying maps.  Define the \emph{$k$-invariant}
  \[
   \kappa(S_F/F) = H_{d-1}(\bar c_F) [S_F/F] \in H_{d-1}(BF).
  \]
  It is a generator of this cyclic group.

  Two oriented Swan complexes $S_F$ and $S'_F$ of type $(F,d-1)$ are \emph{oriented  homotopy equivalent}
  if there is an orientation preserving equivariant homotopy
  equivalence.  This occurs if and only if their $k$-invariants are
  equal (see~\cite[Proposition 2.21]{Davis-Milgram(1985)}). Furthermore, any additive generator of
  $H_{d-1}(BF)$ is realized as the $k$-invariant of an oriented Swan complex (see~\cite[Lemma 2.22]{Davis-Milgram(1985)}).
\end{remark}

\begin{remark}\label{rem:slice-systems_and_Gamma_intro}
  Thus for $d$ odd, a $d$-dimensional slice system exists if and only if all $F \in \calm$
  have order 2, and all slice systems are homotopy equivalent.  For $d$ even, the
  existence of a $d$-dimensional slice system is equivalent  to every
  $F\in \calm$ having periodic cohomology of period $d+1$, and the homotopy type of a
  slice system is determined by the $k$-invariants.  Thus if $G := \Gamma/\pi$ has
  periodic cohomology of period $d+1$, then a $d$-dimensional slice system exists.  The
  existence of $d$-dimensional manifold slice system is equivalent to $S^{d-1}$ admitting
  a free $F$-action for every element $F \in \calm$.  This occurs if $G$ acts freely on
  $S^{d-1}$, for example if $G$ is cyclic and $d$ is even.
\end{remark}

Let $\cals = \{S_F \mid F\in \calm\}$ be a slice manifold system.  We denote by $D_F$ the
cone over $S_F$.  So we get a compact $d$-dimensional topological manifold $D_F$ with
boundary $\partial D_F = S_F$ together with a topological $F$-action such that the
$F$-action is free outside one point $0_F$ in the interior of $D_F$, whose isotropy group
is $F$, the pair $(D_F,S_F)$ is a finite $F$-$CW$-pair, and $(D_F,S_F)$ is homeomorphic to
$(D^d,S^{d-1})$.

In dimension $d \ge 6$ the desired $F$-$CW$-complex structure on $S_F$ comes for free in
Definition~\ref{def:manifold_slice_system}.  Namely, the closed topological manifold
$S_F/F$ has a handlebody structure and hence a $CW$-structure, if
$\dim(S_F/F) =d-1 \ge 5$, see~\cite[Section~9.2]{Freedman-Quinn(1990)}
and~\cite[III.2]{Kirby-Siebenmann(1977)}, and therefore $S_F$ is a free
$F$-$CW$-complex. Note that it is an open question, whether every closed $4$-manifold
carries a $CW$-structure.  There are examples of closed $4$-manifolds, which admit no
triangulation, see~\cite{Manolescu(2014ICM)}.

    \begin{notation}\label{not_C(Z)_intro}
      Given a space $Z$, with path components $\pi_0(Z)$, let $C(Z)$ be its \emph{path
        componentwise cone}, i.e, $C(Z) := \coprod_{C \in \pi_0(Z) } \cone(C)$.
    \end{notation}

    One may describe $C(Z)$ also by the pushout
    \[
      \xymatrix{Z \ar[r]^p \ar[d]_{i_0} & \pi_0(Z) \ar[d]
        \\
        Z \times [0,1] \ar[r] & C(Z) }
    \]
    where $i_0 \colon Z \to Z \times[0,1]$ sends $z$ to $(z,0)$, $p \colon Z \to \pi_0(Z)$
    is the projection, and $\pi_0(Z)$ is equipped with the discrete topology.  If $Z$ is a
    $\Gamma$-$CW$-complex, then $C(Z)$ inherits a $\Gamma$-$CW$-structure. If
    $\cals = \{S_F \mid F \in \calm\}$ is a free $d$-dimensional slice manifold system, we
    get an identification of $\Gamma$-manifolds
    \[
      C(\coprod_{F \in \calm} \Gamma \times_F S_F) = \coprod_{F \in \calm} \Gamma \times_F
      D_F.
    \]
  
\begin{definition}[Slice manifold model]\label{def:slice-manifold_model}
  We call a proper $\Gamma$-manifold $M$ without boundary a \emph{slice mani\-fold model
    for $\underline{E}\Gamma$}, or just \emph{slice mani\-fold model}, with respect to the
  slice manifold system $\cals = \{S_F \mid F \in \calm\}$, if there exists a proper
  cocompact free $d$-dimensional $\Gamma$-manifold $N$ with boundary
  $\partial N = \coprod_{F \in \calm} \Gamma \times_F S_F$ and a $\Gamma$-pushout
  \[\xymatrix{\partial N = \coprod_{F \in \calm} \Gamma \times_F S_F \ar[r]\ar[d] & N
      \ar[d]
      \\
      C(\partial N) = \coprod_{F \in \calm} \Gamma \times_F D_F \ar[r] & M, }
  \]
  where the left vertical arrow is the obvious inclusion, such that $M$ is
  $\Gamma$-homotopy equivalent to $\eub{\Gamma}$.

  We call the pair $(N,\partial N)$ a \emph{slice manifold complement}.
\end{definition}

Note that for a slice manifold model $M$ we have specified an open $\Gamma$-neighborhood
of the singular set $M^{>1}$. Such an open neighborhood exists automatically in the smooth
category. We will discuss this assumption in the topological category in
Section~\ref{sec:Comparisons}.

We will frequently use that a slice manifold model comes with a $\Gamma$-pushout
\begin{equation}
  \xymatrix{\partial N \ar[r] \ar[d]
    &
    N
    \ar[d]
    \\
    \partial \eub{\Gamma} \ar[r]
    &
    \eub{\Gamma},
  }
  \label{canonical_homotopy_Gamma_pushout_for_a_slice_manifold_model_intro}
\end{equation}
where the left vertical arrow is the disjoint union over $F \in \calm$ of the canonical
projections $\Gamma \times_F S_F \to \Gamma/F$.

%%%%%%%%%%%%%%%%%%%%%%%%%%%%%%%%%%%%%%%%%%%%%%%%%%%%%%%%%%%%%%%%%%%%%%%%%%%%%%%%%%%

\subsection{Main theorems}%
\label{subsec:Main_theorems}

We introduce some notation and then formulate our main theorems.  Let
\[
  1 \to \pi \xrightarrow{i} \Gamma \to G \to 1
\]
be a group extension with $G$ finite as in~\eqref{group_extension_intro}.  Assume that
$B\pi$ is a Poincar\'e complex of dimension $d >0$ (e.g.~$B\pi$ is a closed
$d$-manifold.).  Then by Poincar\'e duality, $H^d_{\pi}(E\pi; \Z\pi)$ is infinite cyclic
as an abelian group, hence is isomorphic to $\Z^v$ as a $\Z\pi$-module for a unique
homomorphism
\[
  v : \pi \to \{\pm 1\}.
\]
Here $\Z^v$ is the $\Z\pi$-module which is infinite cyclic as an abelian group, but where
$\gamma x = v(\gamma)x$ for $\gamma \in \pi$ and $x \in \Z^v$.  It follows that
$H_d^\pi(E\pi; \Z^v)$ is infinite cyclic as an abelian group.  Choose a generator $[B\pi]$
(a ``fundamental class").

Shapiro Lemma~\cite[Proposition~III.6.2]{Brown(1982)} gives an isomorphism of $\Z\pi$-modules
$H_\Gamma^d(E \Gamma; \Z \Gamma) \cong H_\pi^d(E\pi; \Z\pi)$.  Hence the $\Z\Gamma$-module
$H^d_{\Gamma}(E\Gamma; \Z\Gamma) $ is infinite cyclic as an abelian group and thus
determines a homomorphism
\begin{equation} \label{w_colon_Gamma_to_(pm_1)} w : \Gamma \to \{\pm 1\}
\end{equation}
which restricts to $v$. It has already been defined in~\cite[Notation~6.7]{Lueck(2022_Poincare_models)}.

Assume $\Gamma$ satisfies conditions (M) and (NM), in other words, assume that every
non-trivial finite subgroup is contained in a unique maximal finite subgroup $F$ and that
$N_\Gamma F = F$. Assume also that every such $F$ has cohomology of period $d+1$
(equivalently, there is a $d$-dimensional slice system) Then each of the two maps below
are rational isomorphisms
\[
  H_d^{\pi}(E\pi; \Z^v) \to H_d^{\Gamma}(E\Gamma; \Z^w) \to H_d^{\Gamma}(E\Gamma, \partial
  E\Gamma; \Z^w)
\]
and each of the three groups are infinite cyclic (see~\cite[Lemma 6.21 and diagram
(6.4)]{Lueck(2022_Poincare_models)}). The specified generator $[B\pi]$ for the first group
specifies generators $[B\Gamma]$ and $[B\Gamma,\partial B\Gamma]$ for the second and third
groups by requiring that they are a positive multiple of the image of $[B\pi]$.

\begin{definition}[Condition (H)]\label{definition:(H)}
  The composite
  \[
    H_d^{\Gamma}(E\Gamma,\partial E\Gamma;\Z^w) \xrightarrow{\partial}
    H_{d-1}^{\Gamma}(\partial E\Gamma;\Z^w) \xrightarrow{\cong} \bigoplus_{F \in \calm}
    H^F_{d-1}(EF;\Z^{w|F}) \to H^F_{d-1}(EF;\Z^{w|F})
  \]
  of the boundary map, the inverse of the obvious isomorphism and the projection to the
  summand of $F \in \calm$ is surjective for all $F \in \calm$.
\end{definition}

We now review the condition (S) on an oriented slice system $\cals$.  This condition was
introduced in Section 7 of~\cite{Lueck(2022_Poincare_models)}, where further details and
explanations are given.

\begin{definition}[Condition (S)]\label{definition:(S)} Let
  $\cals = \{S_F,[S_F] \mid F\in \calm\}$ be an oriented $d$-dimensional free slice system
  with $d$ even and suppose $w : \Gamma \to \{\pm 1\}$ satisfies condition (H).  Let
  $\kappa_F \in H_{d-1}(BF)$ be the image of $[B\Gamma,\partial B\Gamma]$ under the
  composite
  \[
    H_d^{\Gamma}(E\Gamma,\partial E\Gamma;\IZ^w) \xrightarrow{\partial}
    H^{\Gamma}_{d-1}(\partial E\Gamma;\IZ^w ) \xrightarrow{\cong} \bigoplus_{F \in \calm}
    H_{d-1}(BF) \to H_{d-1}(BF).
  \]
  \emph{Condition (S) on the oriented slice system $\cals$} says that
  \[
    \kappa[S_F/F] = \kappa_F \in H_{d-1}(BF).
  \]
\end{definition}

Note that if assumption (H) holds and $d$ is even, there is always an oriented slice
system satisfying condition (S), unique up to oriented homotopy equivalence (see
Remark~\ref{rem:slice-systems_and_Gamma_intro}).

So regardless what slice manifold model $M$ we get out of
Theorem~\ref{the:existence_intro}, its underlying slice manifold system $\cals'$ has the
property that the $F$-homotopy type of $S_F'$ is uniquely determined by the group $\Gamma$
itself and is independent of $M$.

Let $\UNil_d(\IZ;\IZ^{\pm 1},\IZ^{\pm 1})$ the UNil-groups defined by Cappell, see
Remark~\ref{rem:Identification_with_UNil-groups}.

\begin{theorem}[Existence]\label{the:existence_intro}
  Suppose there is a short exact sequence of groups
  \[
    1 \to \pi \xrightarrow{i} \Gamma \xrightarrow{p} G \to 1
  \]
  with $G$ finite.  Suppose that the following conditions are satisfied:

  \begin{enumerate}

  \item\label{the:existence_intro_Bpi} There is a closed manifold of dimension $d$, which
    is homotopy equivalent to $B\pi$. Fix a generator $[B\pi]$ of the infinite cyclic
    group $H_d^{\pi}(E\pi;\IZ^{v})$;
       
  \item\label{the:existence_intro:d} The natural number $d$ satisfies $d \ge 5$;

  \item\label{the:orientation} For every $F \in \calm$ the restriction of the homomorphism
    $w : \Gamma \to \{\pm1\}$ to $F$ is trivial, if $d$ is even, and is non-trivial, if
    $d$ is odd;

  \item\label{the:existence_intro:(M),(NM),(H)} The group $\Gamma$ satisfies conditions
    (M), (NM), and (H), see Definitions~\ref{def:conditions_on_Gamma_intro}
    and~\ref{definition:(H)};

  \item\label{the:existence_intro:eub(gamma)} One of the following assertions holds:

    \begin{enumerate}

    \item\label{the:existence_intro:eub(gamma):model} There exists a finite
      $\Gamma$-$CW$-model for $\eub{\Gamma}$, the group $\pi$ satisfies the Full
      Farrell-Jones Conjecture, see
      Subsection~\ref{subsec:The_Full_Farrell_Jones_Conjecture}, and $\Gamma$ satisfies
      condition (V$_{\operatorname{II}})$, see
      Definition~\ref{def:conditions_on_Gamma_intro};

    \item\label{the:existence_intro:eub(gamma):hyperbolic} The group $\pi$ is hyperbolic;

    \item\label{the:existence_intro:eub(gamma):CAT(0)} The group $\Gamma$ acts
      cocompactly, properly, and isometrically on a proper $\CATzero$-space;
    \end{enumerate}

  \item\label{the:existence_intro:cals} There exists an oriented free $d$-dimensional
    slice system $\cals$ in the sense of
    Definition~\ref{def:free_d-dimensional_slice_system}, which satisfies condition (S).
    Fix such a choice;

  \item\label{the:existence_intro:UNIL} One of the following conditions is satisfied:
    \begin{enumerate}
    \item\label{the:existence_intro:UNIL:UNIL_vanishes} The groups
      $\UNil_d(\IZ;\IZ^{(-1)^d},\IZ^{(-1)^d})$ and
      $\UNil_{d+1}(\IZ;\IZ^{(-1)^d},\IZ^{(-1)^d})$ vanish;
      
    \item\label{the:existence_intro:UNIL:w_trivial_d_divislble_by_4} We have
      $d \equiv 0 \mod (4)$;
    \item\label{the:existence_intro:UNIL:no_D_infty} The group $\Gamma$ contains no
      subgroup isomorphic to $D_{\infty}$;
    \item\label{the:existence_intro:UNIL:F_odd_order} Every element $F \in \calm$ has odd
      order;
    \item\label{the:existence_intro:UNIL:G_odd_order} The order of $G$ is odd.
    \end{enumerate}
    
  \end{enumerate}

  Then there exists a $d$-dimensional free slice mani\-fold system
  $\cals' = \{S_F' \mid F\in \calm\}$ in the sense of
  Definition~\ref{def:manifold_slice_system} and a slice mani\-fold model for
  $\eub{\Gamma}$ with respect to the slice system $\cals'$ in the sense of
  Definition~\ref{def:slice-manifold_model}.  Moreover, for any such pair $(M,\cals')$,
  the $F$-$CW$-complexes $S_F$ and $S_F'$ are $F$-homotopy equivalent.
\end{theorem}

Theorem~\ref{the:existence_intro} is a direct consequence of
Remark~\ref{rem:implications_of_subconditions},
Lemma~\ref{lem:consequences_of_the_conditions_(M),(NM)_(F)},
Remark~\ref{rem:Identification_with_UNil-groups},
Theorem~\ref{the:computing_the_periodic_structure_group}~\ref{the:computing_the_periodic_structure_group:vanishing},
and Theorem~\ref{the:existence_of_manifold_model}.

\begin{remark}[The role of $w$]\label{rem:role_of_w}
  Consider the situation of Theorem~\ref{the:existence_intro}.  If $M$ is any slice
  mani\-fold model for $\eub{\Gamma}$ and $N$ is any slice manifold complement in the
  sense of Definition~\ref{def:slice-manifold_model}, then $N$ is simply connected and the
  homomorphism $w$ of~\eqref{w_colon_Gamma_to_(pm_1)} is automatically the first
  Stiefel-Whitney class of $N/\Gamma$ under the obvious identification
  $\Gamma = \pi_1(N/\Gamma)$. Moreover, the restriction of $w$ to $\pi$ is the first
  Stiefel-Whitney class of any closed manifold model for $B\pi$.
\end{remark}

\begin{remark}[Orientation preserving]~\label{rem:Orientation_preserving} One can improve
  Theorem~\ref{the:existence_intro} by taking the orientations of the slice systems into
  account.  For simplicity we consider only the case, where $d$ is even.  Namely, the
  choice of a generator $[B\pi] \in H_d^{\pi}(E\pi;\IZ^w)$ yields a choice of a generator
  $[N,\partial N]$ of the infinite cyclic group $H_d^{\pi}(N,\partial N;\IZ^w)$,
  see~\cite[Notation~6.22]{Lueck(2022_Poincare_models)}.  Its image under the obvious
  composite
  \[
    H_d^{\pi}(N,\partial N;\IZ^w) \to H_{d-1}^{\pi}(\partial N;\IZ^w) \xrightarrow{\cong}
    \bigoplus_{F \in \calm} H_{d-1}(S'_F/F)
  \]
  yields orientations on $S_F'$ for every $F \in \calm$.  Then one can show that $S_F$ and
  $S_F'$ are oriented $F$-homotopy equivalent for every $F \in \calm$.
\end{remark}

\begin{remark}[Some redundance]\label{rem:implications_of_subconditions}
  Consider condition~\eqref{the:existence_intro:eub(gamma)} in
  Theorem~\ref{the:existence_intro}.
  Conditions~\eqref{the:existence_intro:eub(gamma):hyperbolic}
  or~\eqref{the:existence_intro:eub(gamma):CAT(0)} imply
  condition~\eqref{the:existence_intro:eub(gamma):model}, provided that conditions (M) and
  (NM) are satisfied and there is a finite $CW$-model for $B\pi$,
  see~\cite[Theorem~B]{Bartels-Lueck(2012annals)},~\cite[Theorem~1.12]{Lueck(2022_Poincare_models)}
  and Lemma~\ref{lem:consequences_of_the_conditions_(M),(NM)_(F)}.  Hence it suffices to
  treat conditions~\eqref{the:existence_intro:eub(gamma):model} when dealing with
  condition~\eqref{the:existence_intro:eub(gamma)}.

  Similarly, consider condition~\eqref{the:existence_intro:UNIL} appearing in
  Theorem~\ref{the:existence_intro}.  Obviously the
  implications~\eqref{the:existence_intro:UNIL:G_odd_order}~$\implies$~\eqref{the:existence_intro:UNIL:F_odd_order}~%
  $\implies$~\eqref{the:existence_intro:UNIL:no_D_infty} hold. The
  implication~\eqref{the:existence_intro:UNIL:w_trivial_d_divislble_by_4}~%
  $\implies$~\eqref{the:existence_intro:UNIL:UNIL_vanishes} has been proved
  in~\cite{Connolly-Davis(2004),Banagl-Ranicki(2006),Connolly-Ranicki(2005),Connolly-Kozniewski(1995)}.
  Hence it suffices to treat condition~\eqref{the:existence_intro:UNIL:UNIL_vanishes}
  and~\eqref{the:existence_intro:UNIL:no_D_infty}, when dealing with
  condition~\eqref{the:existence_intro:UNIL}.
\end{remark}

\begin{remark}[Some conditions are necessary]\label{rem:some_conditions_are_necessary}
  If we want to find a slice manifold model for $\eub{\Gamma}$ in the sense of
  Definition~\ref{def:slice-manifold_model}, then
  conditions~\eqref{the:orientation},~\eqref{the:existence_intro:(M),(NM),(H)},~\eqref{the:existence_intro_Bpi},
  and~\eqref{the:existence_intro:cals} appearing in Theorem~\ref{the:existence_intro} are
  necessary and there must be a finite $\Gamma$-$CW$-model for $\eub{\Gamma}$,
  see~\cite[Lemma~1.9, Lemma~3.3, and Lemma~7.10]{Lueck(2022_Poincare_models)}.

  Condition~\eqref{the:existence_intro:d} stems from the well-known problem that the
  Whitney trick and hence surgery theory works without further assumptions only in high
  dimensions.

  The Farrell-Jones Conjecture, condition (V$_{\operatorname{II}})$, and
  condition~\eqref{the:existence_intro:UNIL:UNIL_vanishes} will enter in the proof that
  certain periodic structure sets are trivial.
\end{remark}

\begin{theorem}[Uniqueness]\label{the:uniqueness_intro}
  Suppose there is a short exact sequence of groups
  \[
    1 \to \pi \xrightarrow{i} \Gamma \xrightarrow{p} G \to 1
  \]
  with $G$ finite.  Let $d$ be a natural number. Consider two $d$-dimen\-sio\-nal free
  slice mani\-fold systems $\cals = \{S_F \mid F\in \calm\}$ and
  $\cals' = \{S_F' \mid F\in \calm\}$.  Let $M$ and $M'$ be two slice manifold models for
  $\eub{\Gamma}$ with respect to $\cals$ and $\cals'$.  Suppose that the following
  conditions are satisfied:

  \begin{itemize}

  \item The natural number $d$ satisfies $d \ge 5$,
 
  \item\label{the:uniqueness_intro:eub(gamma):FJC} One of the following assertions holds:

    \begin{itemize}

    \item\label{the:uniqueness_intro:eub(gamma):FJC:FJV_and_(V_II)} The group $\pi$ is a
      Farrell-Jones group, see Subsection~\ref{subsec:The_Full_Farrell_Jones_Conjecture},
      and $\Gamma$ satisfies condition (V$_{\operatorname{II}})$, see
      Definition~\ref{def:conditions_on_Gamma_intro};

    \item\label{the:uniqueness_intro:eub(gamma):FJC:hyperbolic} The group $\pi$ is
      hyperbolic;

    \item\label{the:uniqueness_intro:eub(gamma):FJC:CAT(0)} The group $\Gamma$ acts
      cocompactly, properly, and isometrically on a proper $\CATzero$-space;
    \end{itemize}

  \item\label{the:uniqueness_intro:UNIL} One of the following conditions is satisfied:
    \begin{itemize}
    \item\label{the:uniqueness_intro:UNIL:UNIL_vanishes} The group
      $\UNil_{d+1}(\IZ;\IZ^{(-1)^d},\IZ^{(-1)^d})$ vanishes;
    \item\label{the:uniqueness_intro:UNIL:w_trivial_d_divislble_by_4} We have
      $d \equiv 0 \mod (4)$ or $d \equiv 1 \mod (4)$;
    \item\label{the:uniqueness_intro:UNIL:no_D_infty} The group $\Gamma$ contains no
      subgroup isomorphic to $D_{\infty}$;
    \item\label{the:uniqueness_intro:UNIL:F_odd_order} Every element $F \in \calm$ has odd
      order;
    \item\label{the:uniqueness_intro:UNIL:G_odd_order} The order of $G$ is odd.
    \end{itemize}
  \end{itemize}

  Then:

  \begin{enumerate}
  \item\label{the:uniqueness_intro:Gamma-homeo_M_and_M'} There exists a
    $\Gamma$-homeomorphism $M \xrightarrow{\cong} M'$;

  \item\label{the:uniqueness_intro:h-cobordant} For every $F$ there exists an
    $F$-$h$-cobordism between $S_F$ and $S_F'$;

  \item\label{the:uniqueness_intro:slice_complements} Suppose additionally that for every
    $F \in \calm$ the $2$-Sylow subgroup of $F$ is cyclic.  Let $(N,\partial N)$ and
    $(N',\partial N')$ be slice manifold complements of $M$ and $M'$. Suppose that $S_F$
    and $S_F'$ are simple $F$-homotopy equivalent for every $F \in \calm$.

    Then there exists a $\Gamma$-homeomorphism
    $(N,\partial N) \xrightarrow{\cong} (N',\partial N')$. In particular $S_F$ and $S_F'$
    are $F$-homeomorphic for every $F \in \calm$.
  \end{enumerate}

\end{theorem}

Theorem~\ref{the:uniqueness_intro} is now a direct consequence of
Remark~\ref{rem:implications_of_subconditions},
Remark~\ref{rem:Identification_with_UNil-groups}, and
Theorem~\ref{the:uniqueness_of_manifold_models}.  This theorem is similar to the main
results of~\cite{Connolly-Davis-Khan(2014H1)},~\cite{Khan(2013)},
and~\cite{Connolly-Davis-Khan(2015)}; some discussion of the variant statements is given
in~\ref{subsec:Assumptions}.

%%%%%%%%%%%%%%%%%%%%%%%%%%%%%%%%%%%%%%%%%%%%%%%%%%%%%%%%%%%%%%%%%%%%%%%%%%%%%%%

\subsection{Acknowledgments}\label{subsec:Acknowledgements}

The first author was supported by NSF grant DMS 1615056, Simons Foundation Collaboration
Grant 713226, and the Max Planck Institute for Mathematics.  He thanks Dylan Thurston for
a useful conversation.  The second author is funded by the Deutsche Forschungsgemeinschaft
(DFG, German Research Foundation) under Germany's Excellence Strategy \--- GZ 2047/1,
Projekt-ID 390685813, Hausdorff Center for Mathematics at Bonn.

We also thank the referee for an extraordinary careful reading of this paper, and for the
helpful expository suggestions. We also thank Dominik Kirstein, Christian Kremer, and Markus Land for helpful discussions.
  
%%%%%%%%%%%%%%%%%%%%%%%%%%%%%%%%%%%%%%%%%%%%%%%%%%%%%%%%%%%%%%%%%%%%%%%%%%%%%%%

The paper is organized as follows:

\tableofcontents

%%%%%%%%%%%%%%%%%%%%%%%%%%%%%%%%%%%%%%%%%%%%%%%%%%%%%%%%%%%%%%%%%%%%%%%%%%%%%%%%%%%
%%%%%%%%%%%%%%%%%%%%%%%%%%%%%%%%%%%%%%%%%%%%%%%%%%%%%%%%%%%%%%%%%%%%%%%%%%%%%%%%%%%
%%%%%%%%%%%%%%%%%%%%%%%%%%%%%%%%%%%%%%%%%%%%%%%%%%%%%%%%%%%%%%%%%%%%%%%%%%%%%%%%%%%

\typeout{------ Section: Relating some conditions on $\Gamma$ -----------------------}

\section{Relating some conditions on $\Gamma$}%
\label{sec:Relating_some_conditions_on_Gamma}

\begin{lemma}\label{lem:consequences_of_the_conditions_(M),(NM)_(F)}
  Let $\Gamma$ be a group.
  \begin{enumerate}

  \item\label{lem:consequences_of_the_conditions_(M),(NM)_(F):(M)_and_(NM)_imply_(F)}
    Suppose that $\Gamma$ satisfies (M). Then $\Gamma$ satisfies (F), if and only if it
    satisfies (NM);

  \item\label{lem:consequences_of_the_conditions_(M),(NM)_(F):type_I} Suppose that
    $\Gamma$ satisfies condition (F). Then every virtually cyclic subgroup of type I is
    infinite cyclic;

  \item\label{lem:consequences_of_the_conditions_(M),(NM)_(F):type_II} Suppose that
    $\Gamma$ satisfies condition (F).  Let $V$ be a virtually cyclic subgroup of type
    {II}.  Then $V$ and $N_{\Gamma}V$ are both isomorphic to the infinite dihedral group
    $D_{\infty}$;
  
  \item\label{lem:consequences_of_pseudo-free} Suppose $\Gamma$ acts properly,
    cocompactly, and effectively on a contractible manifold $M$.  The following are
    equivalent:
    \begin{enumerate}
    \item The $\Gamma$ action on $M$ is pseudo-free;
    \item $\Gamma$ satisfies $(F)$;
    \item $\Gamma$ satisfies $(M)$ and $(NM)$;
    \end{enumerate}

  \item\label{lem:consequences_of_the_conditions_(M),(NM)_(F):hyperbolic} Suppose that
    $\Gamma$ is hyperbolic or, more generally, that any infinite subgroup, which is not
    virtually cyclic, contains a copy of $\IZ\ast\IZ$ as subgroup.  Then $\Gamma$
    satisfies conditions (V) and (NV);

  \item\label{lem:consequences_of_the_conditions_(M),(NM)_(F):(F)_and_(V_II)_imply_(NV_II)}
    Suppose that $\Gamma$ satisfies condition (F) and (V$_{\operatorname{II}}$).  Then
    $\Gamma$ satisfies condition (NV$_{\operatorname{II}}$);

  \item\label{lem:consequences_of_the_conditions_(M),(NM)_(F):(F)_and_manifold} Suppose
    that $\Gamma$ acts cocompactly, properly, and isometrically on a proper
    $\CATzero$-space $X$ and that $\Gamma$ satisfies (F).

    Then $\Gamma$-satisfies conditions (V$_{\operatorname{II}}$) and
    (NV$_{\operatorname{II}}$).
    
  \end{enumerate}
\end{lemma}
\begin{proof}~\eqref{lem:consequences_of_the_conditions_(M),(NM)_(F):(M)_and_(NM)_imply_(F)}
  Suppose that (M) and (NM) hold. Choose a maximal finite subgroup $M$ of $\Gamma$
  satisfying $H \subseteq M$.  Consider $\gamma \in N_{\Gamma}H$. We get
  $\{1\} \not= H = H \cap \gamma H \gamma^{-1}\subseteq M \cap \gamma M \gamma^{-1}$.  We
  conclude $M = \gamma M \gamma^{-1}$ from condition (M). This implies
  $\gamma \in N_{\Gamma} M$.  Therefore $N_{\Gamma}H$ is contained in the finite subgroup
  $M$ by condition (NM). Hence (F) holds.

  Suppose that (M) and (F) hold. Consider a non-trivial maximal finite subgroup
  $F \subseteq \Gamma$.  Since $N_{\Gamma} F$ is finite by assumption and $F$ is maximal,
  we get $N_{\Gamma} F= F$. Hence (NM) holds.
  \\[1mm]~\eqref{lem:consequences_of_the_conditions_(M),(NM)_(F):type_I} We can find a
  normal finite subgroup $K$ of $H$ such that $V/K$ is infinite cyclic. Since
  $V \subseteq N_{\Gamma} K$ holds and $V$ is infinite, $K$ must be trivial by condition
  (F).  Hence $V$ is infinite cyclic.
  \\[1mm]~\eqref{lem:consequences_of_the_conditions_(M),(NM)_(F):type_II} Choose an
  epimorphism $p \colon V \to D_{\infty}$ with finite kernel $K$.  Since
  $V \subseteq N_{\Gamma} K$ holds and $V$ is infinite, $K$ must be trivial by condition
  (F). Hence $V$ is isomorphic to $D_{\infty}$.

  Let $V$ be a virtually cyclic subgroup of type {II}. We have already proved that $V$
  contains two cyclic subgroups $C_1$ and $C_2$ of order two such that $C_1 \ast C_2 = V$.
  We have the short exact sequence $1 \to C_{\Gamma}V \to N_{\Gamma} V \to \aut(V)$ where
  the first term is the centralizer of $V$.  Obviously
  $C_{\Gamma}V = C_{\Gamma} C_1 \cap C_{\Gamma}C_2 \subseteq N_{\Gamma}C_1 \cap
  N_{\Gamma}C_2$.  The condition (F) implies that $N_{\Gamma}C_1 \cap N_{\Gamma}C_2$ is
  finite. Since $\aut(V)$ is the semi-direct product $D_{\infty} \rtimes \IZ/2$, the group
  $N_{\Gamma}V$ is virtually cyclic.  Since it does not have a central element of infinite
  order, it is a virtually cyclic subgroup of type {II}. We have already shown that any
  such group is isomorphic to $D_{\infty}$.
  \\[1mm]~\eqref{lem:consequences_of_pseudo-free} Proposition 2.3
  of~\cite{Connolly-Davis-Khan(2015)} shows that (a) and (b) are equivalent, and,
  furthermore if either of these are satisfied, then the fixed-point set of a finite
  non-trivial subgroup is a point.  It follows that if (a) and (b) are satisfied, then for
  non-trivial finite subgroups $F \subseteq H$, $M^H = M^F$ and both are contained in the
  isotropy group of this point, thus conditions (M) and (NM) are satisfied.  But
  conditions (M) and (NM) imply condition (F) by
  assertion~\eqref{lem:consequences_of_the_conditions_(M),(NM)_(F):(M)_and_(NM)_imply_(F)}.
  \\[1mm]~\eqref{lem:consequences_of_the_conditions_(M),(NM)_(F):hyperbolic} This follows
  from~\cite[Theorem~3.1, Lemma~3.4, Example~3.6]{Lueck-Weiermann(2012)}.
  \\[1mm]~\eqref{lem:consequences_of_the_conditions_(M),(NM)_(F):(F)_and_(V_II)_imply_(NV_II)}
  Let $V$ be a maximal virtually cyclic subgroup of type {II}. Since $V$ is maximal and
  $N_{\Gamma} V$ is a virtually cyclic subgroup of type {II} by
  assertion~\eqref{lem:consequences_of_the_conditions_(M),(NM)_(F):type_II} we must have
  $N_{\Gamma} V = V$.
  \\[1mm]~\eqref{lem:consequences_of_the_conditions_(M),(NM)_(F):(F)_and_manifold}
  By~\eqref{lem:consequences_of_the_conditions_(M),(NM)_(F):(F)_and_(V_II)_imply_(NV_II)},
  we only need prove that $\Gamma$ satisfies condition (V$_{\operatorname{II}}$).

  Note that any virtually cyclic subgroup of type II is isomorphic to $D_{\infty}$ by
  assertion~\eqref{lem:consequences_of_the_conditions_(M),(NM)_(F):type_II}. Let
  $V \subseteq \Gamma$ be a subgroup isomorphic to
  $D_{\infty} = \langle t,s \mid t^2 = 1, tst= s^{-1}\rangle$. First we show that there is
  a geodesic $c \colon \IR \to X$ such that $\im(c)$ is $V$-invariant.  Consider the
  isometry $l_s \colon X \to X$ given by multiplication with $s$. It is semi-simple,
  see~\cite[Proposition~6.10~(2) in~II.6 on page~233]{Bridson-Haefliger(1999)}.  Obviously
  it has no fixed point. Therefore it has to be a hyperbolic isometry in the sense
  of~\cite[Definition~6.3 in~II.6 on page~229]{Bridson-Haefliger(1999)}.  Hence there is
  an axis for $l_s$, i.e., a geodesic $c \colon \IR \to X$ such that
  $s\cdot c(\tau) = c(\tau + |l_s|)$ holds for all $\tau \in \IR$, where $|l_s| > 0$ is
  the translation length of $l_s$, see~\cite[Theorem~6.8~(1) in~II.6 on
  page~231]{Bridson-Haefliger(1999)}.  Let $Y$ be the subspace of $X$ appearing
  in~\cite[Theorem~6.8~(4) in~II.6 on page~231]{Bridson-Haefliger(1999)}.  Since $Y$ is
  closed and convex in $X$, it is itself a $\CATzero$-space.  Since we have for
  $\tau \in \IR$
  \[s\cdot (t\cdot c)(\tau) = t \cdot s^{-1}\cdot c(\tau) = t \cdot c(\tau - |l_s|),
  \]
  we get by $\tau \mapsto (t\cdot c)(-\tau)$ another axis for $l_s$. This implies that
  $l_t \colon X \to X$ leaves the subspace $\operatorname{Min}(l_s)$ invariant,
  see~\cite[Theorem~6.8~(3) in~II.6 on page~231]{Bridson-Haefliger(1999)}.  Hence there is
  a point $(y_0,\tau_0) \in Y \times \IR$ such that $t \cdot (y_0,\tau_0) = (y_0,\tau_0)$
  holds, see~\cite[Corollary II.2.8 on page 179]{Bridson-Haefliger(1999)}. Now the
  geodesic $c \colon \IR \to X$ given by $y_0$ is an axis for $l_s$ and satisfies
  $tc(\tau_0) = c(\tau_0)$.  We get for all $m \in \IZ$
  \[
    t ^\cdot c(\tau_0+ m \cdot |l_s|) = t\cdot s^m \cdot c(\tau_0) = s^{-m} \cdot t \cdot
    c(\tau_0) = s^{-m} \cdot c(\tau_0) = c(\tau_0 - m \cdot |l_s|) \in \im(c).
  \]
  This implies that $l_t(\im(c)) = \im(c)$. Since $l_s(\im(c)) = \im(c)$ holds, we
  conclude that $c$ is $V$-invariant.

  Next we show for two subgroups $V \subseteq V' \subseteq \Gamma$ and a geodesic
  $c \colon \IR \to X$ such that $V$ and $V'$ are isomorphic to $D_{\infty}$ and $\im(c)$
  is $V$-invariant, that $\im(c)$ is $V'$-invariant.  We can find a presentation of
  $V' =\langle t,s \mid t^2 = 1, tst= s^{-1}\rangle $ and a natural number $m \ge 1$ such
  that $V$ is generated by $t$ and $s^m$. Hence it suffices to show that an axis for $s^m$
  is automatically an axis for $s$. This follows by inspecting the proof
  of~\cite[Theorem~6.8~(2) in~II.6 on page~231]{Bridson-Haefliger(1999)}.

  Now consider any virtually cyclic subgroup $V$ of $\Gamma$ of type {II}.  We know
  already that $V$ has to be isomorphic to $D_{\infty}$. Choose a geodesic
  $c \colon \IR \to X$ such that $\im(c)$ is $V$-invariant. Let $\Gamma_c$ be the subgroup
  of $\Gamma$ consisting of elements, for which $l_{\gamma}(\im(c)) = \im(c)$ holds.  Note
  that $\Gamma_c$ is non-abelian, since it contains $V$.  Since the non-abelian group
  $\Gamma_c$ acts properly and cocompactly on $\IR$, it must be isomorphic to
  $D_{\infty}$, hence is virtually cyclic of type {II}.  Clearly $\Gamma_c$ is a maximal
  virtually cyclic subgroup of type {II} leaving $\im(c)$ invariant, and, in fact, it is
  the unique such maximal subgroup.  The preceding paragraph implies that $\Gamma_c$ is
  the unique maximal virtually cyclic subgroup of type {II} containing $V$.
   
  Hence $\Gamma$ satisfies condition (V$_{\operatorname{II}}$).
\end{proof}

Lemma~\ref{lem:consequences_of_the_conditions_(M),(NM)_(F)}~%
\eqref{lem:consequences_of_the_conditions_(M),(NM)_(F):(F)_and_manifold} was proved
in~\cite[Remark~1.2]{Connolly-Davis-Khan(2015)} in the case where $\Gamma$ acts
cocompactly, properly, and isometrically on a contractible Riemannian manifold $X$ of
non-positive sectional curvature and in the $\CATzero$-case in the
announcement~\cite{Khan(2013)}.

%%%%%%%%%%%%%%%%%%%%%%%%%%%%%%%%%%%%%%%%%%%%%%%%%%%%%%%%%%%%%%%%%%%%%%%%%%%%%%%%%%%
%%%%%%%%%%%%%%%%%%%%%%%%%%%%%%%%%%%%%%%%%%%%%%%%%%%%%%%%%%%%%%%%%%%%%%%%%%%%%%%%%%%
%%%%%%%%%%%%%%%%%%%%%%%%%%%%%%%%%%%%%%%%%%%%%%%%%%%%%%%%%%%%%%%%%%%%%%%%%%%%%%%%%%%

\typeout{-------------------------- Section: Computing the $K$-theory
  --------------------------}

\section{Equivariant homology theories, spectra over groupoids and the Full Farrell-Jones
  Conjecture}%
\label{sec:Equivariant_homology_theories_spectra_over_groupoids_and_the_Full_Farrell-Jones_Conjecture}

%%%%%%%%%%%%%%%%%%%%%%%%%%%%%%%%%%%%%%%%%%%%%%%%%%%%%%%%%%%%%%%%%%%%%%%%%%%%%%%%%%%

\subsection{Equivariant homology theories}\label{subsec:equivariant_homology_theories}

For the definition of a \emph{$G$-homology theory} $\calh_*^G$ for $G$-$CW$-pairs we refer
for instance to~\cite[Chapter~1]{Lueck(2002b)}. This is extended to the notion of an
\emph{equivariant homology theory} $\calh^?_*$ in~\cite[Chapter~1]{Lueck(2002b)}. Roughly
speaking, an equivariant homology theory assigns to every group $G$ a $G$-homology theory
$\calh_*^G$ and comes with a so-called induction structure, i.e., for any group
homomorphism $\alpha \colon G \to G'$ and $G$-$CW$-pair $(X,A)$, there is a natural
homomorphism $\ind_{\alpha} \colon \calh_*^G(X,A) \to \calh^{G'}_*(\alpha_*(X,A))$ of
$\IZ$-graded abelian groups, where $\alpha_*(X,A)$ is the induced $G'$-$CW$-pair.  It
satisfies certain naturality conditions and is compatible with the long exact sequence of
pairs and disjoint unions. If the kernel of $\alpha$ acts freely on $(X,A)$, the map
$\ind_{\alpha} \colon \calh_*^G(X,A) \to \calh^{G'}(\alpha_*(X,A))$ is an isomorphism.  In
particular we get for every group $G$ and subgroup $H$ a natural isomorphism
$\calh^G_*(G/H) \cong \calh^H_*(\pt)$ and for every free $G$-$CW$-pair $(X,A)$ a natural
isomorphism $\calh^G_*(X,A) \cong \calh^{\{1\}}_*(X/G,A/G)$, using the induction structure.

Given a map $f \colon X \to Y$ of $G$-$CW$-complexes, one can define a $\IZ$-graded
abelian group $\calh^G_*(f)$, which fits into a long exact sequence
\begin{multline*}
  \cdots \to \calh_n^G(X) \xrightarrow{f_n} \calh_n^G(Y) \to \calh^G_n(f)
  \xrightarrow{\partial_n} \calh_{n-1}^G(X)
  \\
  \xrightarrow{f_{n-1}} \calh_{n-1}^G(Y) \to \calh^G_{n-1} (f)
  \xrightarrow{\partial_{n-1}} \cdots.
\end{multline*}
  
Given a commutative square of $G$-$CW$-complexes
\[
  \Phi = \quad \raisebox{8mm}{\xymatrix{X_0 \ar[r]^{f_1} \ar[d]_{f_2} & X_1 \ar[d]^{g_1}
      \\
      X_2 \ar[r]_{g_2} & X }}
\]
one obtains a $\IZ$-graded abelian group $\calh^G_*(\Phi)$, which fits into an exact
sequence
\begin{multline*}
  \cdots \to \calh_n^G(f_2) \to \calh_n^G(g_1) \to \calh^G_n(\Phi)
  \xrightarrow{\partial_n} \calh_{n-1}^G(f_2)
  \\
  \to \calh_{n-1}^G(g_1) \to \calh^G_{n-1}(\Phi) \xrightarrow{\partial_{n-1}} \cdots.
\end{multline*}
If $\Phi$ is a $G$-homotopy pushout, e.g., $\Phi$ is a $G$-pushout and $f_1$ or $f_2$ is
an inclusion of $G$-$CW$-complexes, then $\calh^G_n(\Phi) = 0$ for all $n \in \IZ$.

%%%%%%%%%%%%%%%%%%%%%%%%%%%%%%%%%%%%%%%%%%%%%%%%%%%%%%%%%%%%%%%%%%%%%%%%%%%%%%%%%%%

\subsection{Equivariant homology theories from spectra over groupoids}%
\label{subsec:Equivariant_homology_theories_from_spectra_over_groupoids}

Let $\Groupoids$ be the category of small connected groupoids.  A
\emph{$\Groupoids$-spectrum} $\bfE$ is a functor from $\Groupoids$ to the category of
spectra $\Spectra$. Any $\Groupoids$-spectrum $\bfE$ gives rise to an equivariant homology
theory $H^?_*(-;\bfE)$ in the sense of~\cite[Chapter~1]{Lueck(2002b)}, see for
instance~\cite[Proposition~157 on page~796]{Lueck-Reich(2005)} which is based
on~\cite{Davis-Lueck(1998)}. Thus for any group $G$, we get a $G$-homology theory
$H^G_*(-;\bfE)$ such that for any subgroup $H$ we have
\[H_n^G(G/H;\bfE) \cong H_n^H(\pt,\bfE) \cong \pi_n(\bfE(\widehat{H})),
\]
where $\widehat{H}$ is the groupoid with one object and $H$ as its automorphism group.

Let $\bfE \colon \Groupoids\to \Spectra$ be a $\Groupoids$-spectrum.  Denote by
$\bfE\langle 1 \rangle$ the $\Groupoids$-spectrum obtained from $\bfE$ by passing to the
\mbox{$1$-connected} covering. There is a morphism $\bfE\langle 1 \rangle \to \bfE$ of
$\Groupoids$-spectra such that for every groupoid $\calg$ the map
$\pi_q(\bfE\langle 1 \rangle(\calg)) \to \pi_q(\bfE(\calg))$ is an isomorphism for
$q \ge 1$ and $\pi_q(\bfE\langle 1 \rangle(\calg)) = 0$ for $q \le 0$.  Define a sequence
of $\Groupoids$-spectra $\bfE\langle 1 \rangle \to \bfE \to \overline{\bfE}$ such that its
evaluation at any groupoid is a cofibration sequence of spectra.  For any groupoid,
$\pi_q(\overline{\bfE}(\calg)) = 0$ for $q \ge 1$ and
$\pi_q(\bfE(\calg)) \xrightarrow{\cong} \pi_q(\overline{\bfE}(\calg))$ is an isomorphism
for $q \le 0$. For any square $\Phi$ of $\Gamma$-$CW$-complexes, there is a long exact
sequence, natural in $\Phi$,
\begin{multline}
  \cdots \to H^{\Gamma}_{n+1}(\Phi;\bfE\langle 1 \rangle ) \to H^{\Gamma}_{n+1}(\Phi;\bfE)
  \to H^{\Gamma}_{n+1}(\Phi;\overline{\bfE})
  \\
  \to H^{\Gamma}_{n}(\Phi;\bfE\langle 1 \rangle ) \to H^{\Gamma}_{n}(\Phi;\bfE) \to
  H^{\Gamma}_{n}(\phi;\overline{\bfE}) \to \cdots.
  \label{long_exact_homology_sequence_for_E_to_E_langle_1_rangle_to_overline(E)}
\end{multline}

\subsection{Some basics about $K$-and $L$-theory of groups rings}%
\label{subsec:Some_basics_about_K-and_L-theory_of_groups_rings}

Let $K_n(RG)$ denote the \emph{$n$-th algebraic $K$-group} of the group ring $RG$ in the
sense of Quillen for $n \ge 0$ and in the sense of Bass for $n \le -1$.  Let $\NK_n(R)$
denote the \emph{Bass-Nil-groups} of $R$, which are defined as the cokernel of the map
$K_n(R) \to K_n(R[x])$.  Recall that the Bass-Heller-Swan decomposition says
\begin{eqnarray}
  K_n(R\IZ) & \cong & K_n(R) \oplus K_{n-1}(R) \oplus \NK_n(R) \oplus \NK_n(R).
                      \label{Bass-Heller-Swan}
\end{eqnarray}
If $R$ is a regular ring, then $\NK_n(R) = 0$ for every $n \in \IZ$, see for
instance~\cite[Theorems~3.3.3 and 5.3.30]{Rosenberg(1994)}.

For a ring with involution $R$, for a integer $n$, and for
$ j \in \{1,0,-1,-2, \ldots \} \amalg \{ -\infty\}$, one defines the Wall-Ranicki
algebraic $L$-group $L_n^{\langle j \rangle}(R)$, combine~\cite[Section~13]{Ranicki(1992)}
with~\cite[Section~17]{Ranicki(1992a)}.
These groups are 4-periodic in $n$.  The index $j$ is called the decoration.  The group
$L_n^{\langle - \infty \rangle}(R)$ is called the {\em ultimate lower quadratic
  $L$-group}.  The $L$-groups are given as the homotopy groups of a 4-periodic spectrum
$\bfL^{\langle j \rangle}(R)$ (see~\cite[Section 13]{Ranicki(1992)}).  When $R = \Z$, the
$L$-groups and $L$-spectra are constant in $j$, that is, they are independent of the
decoration.  Often $L^{\langle j \rangle}$ and $L^{\langle 1 \rangle}$ are denote by $L^p$
and $L^h$ respectively.  For groups, $RG$, one also includes $j = 2$, which is also
denoted by $L^s(RG)$ (see~\cite[page 105]{Ranicki(1981)}).

For $ j \in \{2,1,0,-1,\ldots \} \amalg \{ -\infty\}$, there are  $\Groupoids$-spectra, see~\cite[Section~2]{Davis-Lueck(1998)},
\begin{eqnarray*}
\bfK_R \colon \Groupoids & \to &\Spectra;
\\
\bfL^{\langle j \rangle}_R \colon \Groupoids & \to &\Spectra,
\end{eqnarray*}
coming with  natural identifications
\begin{eqnarray*}
\pi_n(\bfK_R(\widehat{G})) & = & K_n(RG);
\\
\pi_n(\bfL^{\langle j \rangle}_R(\widehat{G})) & = & L^{\langle j \rangle}_n(RG).
\end{eqnarray*}

If $G$ is a group with an orientation character, then some modifications to the above
theory must be made, see~\cite{Bartels-Lueck(2009coeff)}.  An \emph{orientation character}
is a homomorphism $w \colon G \to \{\pm 1\}$.  This determines the \emph{$w$-twisted
  involution on $R G$} sending $\sum_{g \in G} \lambda_g \cdot g$ to
$\sum_{g \in G} \overline \lambda_g \cdot w(g) \cdot g^{-1}$.  The corresponding
$L$-groups of this ring with involution are denoted $L_{n}^{\langle j\rangle}(RG,w)$.  To
deal with non-trivial orientation characters, one needs functors
\[
  \bfL^{\langle j \rangle}_{R,w} \colon \Groupoids \downarrow \{\pm 1\} \to \Spectra
\]
coming with a natural identification
\[
  \pi_n(\bfL^{\langle j \rangle}_{R,w}(\widehat{G})) = L^{\langle j \rangle}_n(RG,w).
\]

A group with orientation character $(G,w)$ determines a $G$-homology theory denoted by
$H^G_*(-; \bfL^{\langle j \rangle}_{R,w})$.  There is an isomorphism
$H_n^G(G/H;\bfL^{\langle j \rangle}_{R,w}) \cong L_n^{\langle j \rangle}(RH;w|_H)$.

 For a group with orientation character $(G,w)$ and for a free $G$-CW-complex $X$, 
define the \emph{periodic $n$-th structure group with decoration $\langle j \rangle$} to be
\begin{eqnarray*}
\cals_n^{\per,\langle j \rangle}(X/G) 
& := &
H_n^G(X\to \pt;\bfL_{\Z,w}^{\langle j \rangle}).
%\label{def_of_xals}
\end{eqnarray*}
It the orientation character is trivial, these groups fit into the periodic version of the algebraic 
surgery exact sequence with decoration $\langle j \rangle$,
\begin{multline*}
\cdots \to H_n(X/G;\bfL(\Z)) \to
L_n^{\langle j \rangle}(\Z G) \to \cals_n^{\per,\langle j 
 \rangle}(X/G)
\\
\to H_{n-1}(X/G;\bfL(\Z)) \to
L_{n-1}^{\langle j\rangle}(\Z G) \to \cdots.
\end{multline*}
Here we have identified $H^G_*(X;\bfL_{\Z}^{\langle j \rangle})$ with $H_*(X/G;\bfL(\Z))$
using homotopy invariance and the induction structure.  This periodic surgery sequence
appears in the classification of ANR-homology manifolds in
Bryant-Ferry-Mio-Weinberger~\cite[Main Theorem]{Bryant-Ferry-Mio-Weinberger(1996)}.  It is
related to the algebraic surgery exact sequence and thus to the classical surgery
sequence, see Ranicki~\cite[Section~18]{Ranicki(1992)}.

For $n \in \IZ$, the abelian group
$H_n^{D_{\infty}}(\eub{D_{\infty}} \to \pt;\bfL_R^{\langle - \infty \rangle})$ can be
identified with the $\UNil(R;R,R)$-groups of Cappell~\cite{Cappell(1974c)} see
Remark~\ref{rem:Identification_with_UNil-groups}.  If
$w : D_\infty = \langle a,b \mid a^2=b^2=1 \rangle \to \{\pm1\}$ is given by
$w(a) = w(b) = (-1)^n$, then
$H_n^{D_{\infty}}(\eub{D_{\infty}} \to \pt;\bfL_{\Z,w}^{\langle - \infty \rangle})$ agrees
with the group $\UNil_n(\IZ;\IZ^{(-1)^n},\IZ^{(-1)^n})$ appearing in
Theorem~\ref{the:existence_intro} and Theorem~\ref{the:uniqueness_intro}.

%%%%%%%%%%%%%%%%%%%%%%%%%%%%%%%%%%%%%%%%%%%%%%%%%%%%%%%%%%%%%%%%%%
  
\subsection{The Full Farrell-Jones Conjecture}%
\label{subsec:The_Full_Farrell_Jones_Conjecture}

For the precise formulation of the Full-Farrell-Jones Conjecture and its current status we
refer to~\cite[Sections~13.6 and~16.2]{Lueck(2022book)}.  It is the most general version of the Farrell-Jones
Conjecture. We call a group $G$ a \emph{Farrell-Jones group}, if it satisfies the Full
Farrell-Jones Conjecture. The class of Farrell-Jones groups contains   $\CAT(0)$-groups,
lattices in locally compact second countable
Hausdorff groups, solvable groups, and fundamental group of manifolds of dimension
$\le 3$.  It is closed under taking subgroups, passing to overgroups of finite index, and
colimits over directed systems of groups with not necessarily injective structure
maps. For our purposes it suffices to know that the Full-Farrell-Jones Conjecture implies
that the projection $\edub{G} \to \pt$ induces for every $n \in \IZ$ and every ring $R$
(with involution) isomorphisms
\begin{eqnarray*}
  H_n^G(\edub{G};\bfK_R)
  & \xrightarrow{\cong} &
 H_n^G(\pt;\bfK_R) = K_n(RG);
  \\
  H_n^G(\edub{G};\bfL_{R,w}^{\langle -\infty \rangle})
  & \xrightarrow{\cong} &
H_n^G(\pt;\bfL_{R,w}^{\langle -\infty \rangle}) = L_n^{\langle -\infty \rangle}(RG,w),
\end{eqnarray*}                         
where $\edub{G}$ is the classifying space for the family of virtually cyclic subgroups of $G$.

%%%%%%%%%%%%%%%%%%%%%%%%%%%%%%%%%%%%%%%%%%%%%%%%%%%%%%%%%%%%%%%%%%%%%%%%%%%%%%%%%%%
%%%%%%%%%%%%%%%%%%%%%%%%%%%%%%%%%%%%%%%%%%%%%%%%%%%%%%%%%%%%%%%%%%%%%%%%%%%%%%%%%%%
%%%%%%%%%%%%%%%%%%%%%%%%%%%%%%%%%%%%%%%%%%%%%%%%%%%%%%%%%%%%%%%%%%%%%%%%%%%%%%%%%%%

\typeout{-------------------------- Section: Computing the $K$-theory   --------------------------}

\section{Computing the $K$-theory}%
\label{sec:Computing_the_K_theory}

%%%%%%%%%%%%%%%%%%%%%%%%%%%%%%%%%%%%%%%%%%%%%%%%%%%%%%%%%%%%%%%%%%
\subsection{The definition of Whitehead groups}\label{subsec:The_definition_of_Whitehead_groups}

Define for a group $G$ and a ring $R$ the \emph{$n$-th Whitehead group} $\Wh_n(G;R)$ to be
$H_n^G(EG \to \pt;\bfK_R)$.  The long exact sequence of the map $EG \to \pt$ yields a long
exact sequence
\begin{multline*}
  \cdots \to H_n(BG;\bfK(R))  \to K_n(RG) \to \Wh_n(G;R)
  \\
  \to H_{n-1}(BG;\bfK(R)) \to    K_{n-1}(RG) \to \cdots,
  \label{long_exact_sequence_for_Wh_n(G;R)}
\end{multline*}
where $H_*(-;\bfK(R))$ is the generalized (non-equivariant) homology theory associated to
the non-connective $K$-theory spectrum $\bfK(R)$ of $R$. Suppose that $R$ is regular. Then
$K_n(R) = 0$ for $n \le -1$.  Hence the canonical map $K_n(RG)\to \Wh_n(G;R)$ is
bijective for $n \le -1$, we have the split short exact sequence
$0 \to K_0(R) \to K_0(RG) \to \Wh_0(G;R) \to 0$, and the short exact sequence
$H_1(BG,\bfK(R)) \to K_1(RG) \to \Wh_1(G;R) \to 0$.  If $R$ is regular and the canonical map
$K_0(\IZ) \to K_0(R)$ is bijective, then we get an isomorphism
$\widetilde{K}_0(RG) \xrightarrow{\cong} \Wh_0(G;R)$ and a split short exact 
  sequence $0 \to K_1(R) \oplus G/[G,G] \to K_1(RG) \to \Wh_1(G;R) \to 0$. If
  $R = \IZ$, then $\Wh_1(G;\IZ)$ agrees with  the classical Whitehead group $\Wh(G)$,
  $\widetilde{K}_0(\IZ G) \cong \Wh_0(G;\IZ)$, and $K_n(\IZ G) \cong \Wh_n(G,\IZ)$ for
  $n \le -1$.

Whitehead groups arise naturally when studying $h$-cobordisms, pseudoisotopy, and Waldhausen's A-theory. Their
geometric significance is reviewed, for example, in
Dwyer-Weiss-Williams~\cite[Section~9]{Dwyer-Weiss-Williams(2003)} and
L\"uck-Reich~\cite[Section~1.4.1]{Lueck-Reich(2005)}, where additional references can also
be found.  When $G=\IZ$, it follows from~\eqref{Bass-Heller-Swan} and the fact that
$H_n^{\IZ}(E\IZ;\bfK_R)\cong K_n(R) \oplus K_{n-1}(R)$ that there is an identification
\begin{eqnarray}
H_n^{\IZ}(E\IZ \to \pt;\bfK_R) & \cong & \NK_n(R) \oplus \NK_n(R).
\label{H_nZ(EZ_to_pt;bfK_R)_cong_NK_n(R)_oplus_NK_n(R)}
\end{eqnarray}
  
  %%%%%%%%%%%%%%%%%%%%%%%%%%%%%%%%%%%%%%%%%%%%%%%%%%%%%%%%%%%%%%%%%%
  
  \subsection{Computing Whitehead groups}\label{subsec:Computing:Whitehead_groups}
  
  We will later need the following result to apply the Farrell-Jones Conjecture.
  
\begin{theorem}\label{the:constructing_underline(E)Gamma)}\ 
  \begin{enumerate}
  \item\label{the:constructing_underline(E)Gamma):Fin}
    Suppose that $\Gamma$ satisfies (M) and (NM).
    Let $\calm$ be a
  complete system of representatives of the conjugacy classes of maximal finite subgroups
  $F\subseteq \Gamma$.  Consider the cellular $\Gamma$-pushout
\[
  \xymatrix{\coprod_{F\in\calm} \Gamma \times_{F}EF
    \ar[d]_{\coprod_{F\in\calm}p_F} \ar[r]^-i & E\Gamma \ar[d] \\
    \coprod_{F\in\calm} \Gamma/F \ar[r] & X}
\]
where the map $p_F$ comes from the projection $EF \to \pt$,
and $i$ is an inclusion of $\Gamma$-$CW$-complexes.

Then $X$ is a model for $\underline{E}\Gamma$;

\item\label{the:constructing_underline(E)Gamma):Vcyc}
  Assume that $\Gamma$ satisfies conditions (V) and (NV).  
  Let $\calv$ be a complete system of representatives of the conjugacy
  classes of maximal infinite virtually cyclic subgroups $V\subseteq \Gamma$.

  Consider the cellular $\Gamma$-pushout
\[
  \xymatrix{\coprod_{V\in\calv} \Gamma \times_{V}\eub{V}
    \ar[d]_{\coprod_{V\in\calv} p_V} \ar[r]^-i & \eub{\Gamma} \ar[d] \\
    \coprod_{V\in\calv} \Gamma/V \ar[r] & Y}
\]
where the map $p_V$ comes from the projection $\eub{V} \to \pt$,
and $i$ is an inclusion of $\Gamma$-$CW$-complexes.

Then $Y$ is a model for $\edub{V}$.

\item\label{the:constructing_underline(E)Gamma):Vcyc_II}
  Assume that $\Gamma$ satisfies conditions (V$_{\operatorname{II}}$) and (NV$_{\operatorname{II}}$).  
  Let $\calv_{II}$ be a complete system of representatives of the conjugacy
  classes of maximal infinite virtually cyclic subgroups of type {II}.
  Let $\EGF{V}{\calvcyc_{\operatorname{I}}}$ be the classifying space for the family $\calvcyc_{\operatorname{I}}$
    of $G$, which consists of finite subgroups of infinite  virtually cyclic subgroups of type I.

  Consider the cellular $\Gamma$-pushout
\[
  \xymatrix{\coprod_{V\in\calvII} \Gamma \times_{V}\EGF{V}{\calvcyc_{\operatorname{I}}}
    \ar[d]_{\coprod_{V\in\calv_{II}}p_V} \ar[r]^-i & \EGF{\Gamma}{\calvcyc_{\operatorname{I}}}\ar[d] \\
    \coprod_{V\in\calvII}\Gamma/V \ar[r] & Z}
\]
where the map $p_V$ comes from the projection $\EGF{V}{\calvcyc_I} \to \pt$,
and $i$ is an inclusion of $\Gamma$-$CW$-complexes.

Then $Z$ is a model for $\edub{V}$.
\end{enumerate}
\end{theorem}

\begin{proof} This follows from~\cite[Corollary~2.11]{Lueck-Weiermann(2012)} for
  assertions~\eqref{the:constructing_underline(E)Gamma):Fin}
  and~\eqref{the:constructing_underline(E)Gamma):Vcyc}. The proof for
  assertion~\eqref{the:constructing_underline(E)Gamma):Vcyc_II} is analogous, just
  apply~\cite[Corollary~2.8]{Lueck-Weiermann(2012)}.
 \end{proof}
  
  The next result has already been proved for $R = \IZ$ in~\cite[Theorem~5.1~(d)]{Davis-Lueck(2003)}.

\begin{theorem}\label{the:computing_K-groups} 
  Let $R$ be a regular ring and let $\Gamma$ be a Farrell-Jones group satisfying
  conditions (M) and (NM). Let $\calm$ be a complete system of representatives of the
  conjugacy classes of maximal finite subgroups.  Then the canonical map
  \[
    \bigoplus_{F \in \calm} \Wh_n(F;R)  \xrightarrow{\cong} \Wh_n(\Gamma;R)
  \]
  is bijective for all $n \in \IZ$.
\end{theorem}
\begin{proof}
  Since $\Gamma$ is  a Farrell-Jones group,
  \[H_n^{\Gamma}(\edub{\Gamma};\bfK_R) \xrightarrow{\cong} H_n^{\Gamma}(\pt;\bfK_R)
 \]
 is an isomorphism for all $n \in \IZ$. The relative assembly map
 \[
 H_n^{\Gamma}(\EGF{\Gamma}{\calvcyc_I};\bfK_R) \xrightarrow{\cong} H_n^{\Gamma}(\edub{\Gamma};\bfK_R)
 \]
 is an isomorphism for all $n \in \IZ$ by~\cite[Remark~1.6]{Davis-Quinn-Reich(2011)}.
 Every element $V \in \calvcyc_I$, which is infinite, is an infinite cyclic group by
 Lemma~\ref{lem:consequences_of_the_conditions_(M),(NM)_(F)}.  If $V$ is infinite cyclic,
 we get an isomorphism $H_n^{V}(\eub{V} \to\pt;\bfK_R) \cong \NK_n(R) \oplus \NK_n(R)$
 from the Bass-Heller-Swan decomposition.  Since $R$ is regular, $\NK_n(R)$ vanishes,
 see for instance~\cite[Theorems~3.3.3 and 5.3.30]{Rosenberg(1994)}.  Hence we conclude
 that the assembly map
 \[H_n^V(EV;\bfK_R) \to H_n^V(\pt;\bfK_R)
 \]
 is bijective for all $n \in \IZ$ and $V \in \calvcyc_I \setminus \calfin$.
 We conclude from the Transitivity Principle, see for instance~\cite[Theorem~65 on page 742]{Lueck-Reich(2005)},
 that the map
 \[
 H_n^{\Gamma}(\eub{\Gamma};\bfK_R) \xrightarrow{\cong} H_n^{\Gamma}(\EGF{\Gamma}{\calvcyc_I};\bfK_R)
 \]
 is an isomorphism for all $n \in \IZ$. Hence the map
  induced by the projection $\eub{\Gamma} \to \pt$
  \[H_n^{\Gamma}(\eub{\Gamma};\bfK_R) \xrightarrow{\cong} H_n^{\Gamma}(\pt;\bfK_R)
 \]
 is an isomorphism for all $n \in \IZ$. This implies that the map
 \[
  H_n^{\Gamma}(E\Gamma \to \eub{\Gamma},\bfK_R)\to H_n^{\Gamma}(E\Gamma\to \pt,\bfK_R) =  \Wh_n(G;R)
\]
is bijective for all $n \in \IZ$. Since $\Gamma$ satisfies (M) and (NM), we get from excision and 
Theorem~\ref{the:constructing_underline(E)Gamma)}~\eqref{the:constructing_underline(E)Gamma):Fin}
isomorphisms
\[
\bigoplus_{F \in \calm} H_n^{F}(EF \to \pt;\bfK_R) \xrightarrow{\cong} H_n^{\Gamma}(E\Gamma \to \eub{\Gamma},\bfK_R)
\]
for $n \in \IZ$. Since $H_n^{F}(EF \to \pt;\bfK_R) = \Wh(F;R)$, the proof of 
Theorem~\ref{the:computing_K-groups}  is finished.
\end{proof}

If $\Gamma$ satisfies (M), (NM), (V) and (NV) and the Farrell-Jones Conjecture, one can
also compute the Whitehead group $\Wh_n(\Gamma;R)$ for arbitrary $R$.  Namely, if we
denote by $\calvI$ and $\calvII$ the subset of $\calv$ consisting virtually cyclic
subgroups of type I and of type {II} respectively, then
\begin{align*}
\Wh_n(\Gamma;R) & \cong H_n^\Gamma(E\Gamma \to \pt;\bfK_R) \\
& \cong H_n^\Gamma(E\Gamma \to \edub{\Gamma}  ;\bfK_R) \\
& \cong H_n^\Gamma(E\Gamma \to \eub{\Gamma}  ;\bfK_R) \oplus H_n^\Gamma(\eub{\Gamma} \to \edub{\Gamma}  ;\bfK_R) \\ 
& \cong   \bigoplus_{F \in \calm} H_n^F(EF \to \pt  ;\bfK_R)  \oplus \bigoplus_{V \in \calv} H_n^V(\eub{V} \to \pt  ;\bfK_R) \\
\cong&  \bigoplus_{F \in \calm} \Wh_n(F;R) \oplus \bigoplus_{V \in \calvI} \NK_n(R) \oplus
  \NK_n(R) \oplus \bigoplus_{V \in \calvII} \NK_n(R).
\end{align*}
The first isomorphism is by definition, the second by the Farrell-Jones Conjecture, the
third by~\cite{Bartels(2003b)}, the fourth by
Theorem~\ref{the:constructing_underline(E)Gamma)}~\eqref {the:constructing_underline(E)Gamma):Fin}
and~\eqref{the:constructing_underline(E)Gamma):Vcyc}, and the last by the
Bass-Heller-Swan decomposition if $V \in \calvI$ and  
by~\cite[Corollary 3.27]{Davis-Khan-Ranicki(2011)} if $V \in \calvII$.

%%%%%%%%%%%%%%%%%%%%%%%%%%%%%%%%%%%%%%%%%%%%%%%%%%%%%%%%%%%%%%%%%%%%%%%%%%%%%%%%%%%
%%%%%%%%%%%%%%%%%%%%%%%%%%%%%%%%%%%%%%%%%%%%%%%%%%%%%%%%%%%%%%%%%%%%%%%%%%%%%%%%%%%
%%%%%%%%%%%%%%%%%%%%%%%%%%%%%%%%%%%%%%%%%%%%%%%%%%%%%%%%%%%%%%%%%%%%%%%%%%%%%%%%%%%

\typeout{-------------------------- Section: Computing the L-theory   --------------------------}

\section{Computing the $L$-theory}%
\label{sec:Computing_the_L-theory}

%%%%%%%%%%%%%%%%%%%%%%%%%%%%%%%%%%%%%%%%%%%%%%%%%%%%%%%%%%%%%%%%%%
  
\subsection{Some basics about $K$- and $L$-theory for additive categories with involution}%
\label{subsec:Some_basics_about_K-and_L-theory_for_additive_categories_with_involution}

Although we are only interested in the $L$-groups of group rings, we need some input from
the $L$-theory for additive categories $\cala$ with involution, see
Remark~\ref{rem:decorated_L-theory_is_not_compatible_with_finite_products_of_rings}.  Ranicki
defined decorated $L$-groups $L_n^{\langle j \rangle}(\cala)$ for $n \in \IZ$ and
$j \in \{1,0,-1,-2,\ldots \} \amalg \{-\infty\}$ in~\cite[Section~13
and~17]{Ranicki(1992a)}.  By convention $L_n^{\langle 1 \rangle}(\cala)$ agrees with the
standard $L$-theory $L_n(\cala)$ of $\cala$ and $L_n^{\langle 0 \rangle}(\cala)$ is the
standard $L$-theory $L_n(\Idem(\cala))$ of the idempotent completion
$\Idem(\cala)$. There is a Shaneson splitting and there are Rothenberg sequences,
see~\cite[Theorem~17.2]{Ranicki(1992a)} or~\eqref{Rothenberg_sequences},~\eqref{Shaneson_splitting},
and~\eqref{Shaneson_splitting_infty}.

Given a ring $S$, let $\calf(S)$ be the following small additive category. The set of
objects is $\{[n] \mid n \in \IZ, n \ge 0\}$. A morphism $A \colon [n] \to [m]$ for
$m,n \ge 1$ is given by a $m$-by-$n$ matrix with entries in $S$.  The set of morphisms $[n] \to [m]$ is
defined to be $\{0\}$, if the source or target is $[0]$.  Composition is given by matrix
multiplication. This category is equivalent to the category of finitely generated free
$S$-modules. We define the small additive category $\calp(S)$ to be the idempotent
completion of $\calf(S)$. One easily checks that $\calp(S)$ is equivalent to the additive
category of finitely generated projective $S$-modules.  If $S$ is a ring with involution,
then $\calf(S)$ and $\calp(S)$ become additive categories with involution. One defines
$L_n^{\langle j \rangle}(S) := L_n^{\langle j \rangle}(\calf(S))$. With these conventions
$L_n^{\langle 1 \rangle}(S) = L_n^{h}(S) = L_n(S)$ and
$L_n^{\langle 0 \rangle}(S) = L_n^{p}(S)$.

One reason why it is better to work with additive categories with involutions instead of
rings is the compatibility with direct sums and  direct products. Namely, for a set of additive
categories $\{\cala_i \mid i \in I\}$ for arbitrary $I$, for
$n \in \IZ$, and $j \in \{1,0,-1,-2,\ldots \}$,  the canonical map given by the  projection   yields isomorphisms
\begin{eqnarray}
  K_n(\prod_{i \in I} \cala_i)
  &  \xrightarrow{\cong}  &
  \prod_{i \in I} K_n(\cala_i);
\label{K-theory_and_infinite_products}
\\
  L_n^{\langle j \rangle} (\prod_{i \in I} \cala_i)
  &  \xrightarrow{\cong}  &
 \prod_{i \in I} L_n^{\langle j \rangle}(\cala_i).
\label{L-theory_and_infinite_products}
\end{eqnarray}
This is not true for the decoration $j = -\infty$ in general, unless $I$ is finite or
there exists $j_0 \in \IZ$ such that $K_j(\cala_i) = 0$ for all $i \in I$ and $j \le j_0$,
see~\cite{Carlsson(1995)},~\cite{Carlsson-Pedersen(1995a)},~\cite{Winges(2013)}.

For $n \in \IZ$ and $j \in \{1,0,-1,-2,\ldots \} \amalg \{-\infty\}$,
the canonical maps given by the inclusions  induce isomorphisms
\begin{eqnarray}
\bigoplus _{i \in I} K_n(\cala_i) 
& \xrightarrow{\cong} &
K_n(\bigoplus_{i \in I} \cala_i);
\label{K-theory_and_infinite_sums}
\\
  \bigoplus _{i \in I} L_n^{\langle  j \rangle}(\cala_i)
  & \xrightarrow{\cong}  &
L_n^{\langle j \rangle} (\bigoplus _{i \in I}\cala_i).
\label{L-theory_and_infinite_sums}
\end{eqnarray}
This follows for finite $I$ from~\eqref{K-theory_and_infinite_products} and~\eqref{L-theory_and_infinite_products}
and for general $I$ from the fact that $K$-theory  and the $L$-theory with decoration $\langle j \rangle $ 
commute with colimits over directed systems of additive categories.

%%%%%%%%%%%%%%%%%%%%%%%%%%%%%%%%%%%%%%%%%%%%%%%%%%%%%%%%%%%%%%%%%%
  
\subsection{Some basics about $L$-theory for rings with involution}%
\label{subsec:Some_basics_about_L-theory_for_rings_with_involution}

Let $R$ be a ring satisfying $\widetilde{K}_n(R) = 0$ for $n < 0$ and
$K_0(\Z) \xrightarrow{\cong} K_0(R)$, e.g., a principal ideal domain $R$.  Let $G$ be a
group.  Recall that $\widetilde{K}_n(R)$ is defined to be the cokernel of
$K_n(\IZ) \to K_n(R)$. Consider $S = RG$ equipped with the $w$-twisted involution for a fixed orientation
character $w\colon G \to \{\pm 1\}$. Define $L^{\langle 2 \rangle}(RG,w)$ to be the
$X$-decorated $n$-th $L$-group of $\calf(S)$, where $X$ is the image of the assembly map
$H_1(BG;\bfK(R)) \to K_1(RG)$.  We define $L^{\langle j \rangle}(RG,w)$ for
$j \in \{1,0, -1, \ldots \} \amalg \{-\infty\}$ by $L_n^{\langle j
  \rangle}(\calf(S))$. One can define decorated $L$-groups for arbitrary rings with
involutions.  However, we made the assumption on $R$ essentially in order to guarantee the
following facts.  We have
\begin{eqnarray*}
H_n(BG;\bfK(R)) & = & \{0\} \quad  \text{for} \; n \le -1;
\\
H_0(BG;\bfK(R)) & = &  K_0(R) \cong \IZ; 
\\
H_1(BG;\bfK(R)) &\cong & G/[G,G]  \times K_1(R),
\\
  \Wh_j(G;R) & = & K_j(RG)   \quad  \text{for} \; n \le -1;
\\
\Wh_0(G;R) & = & \widetilde{K}_0(RG)
\end{eqnarray*}
and a  split short exact sequence
\[
  0 \to H_1(BG;\bfK(R)) \to K_1(RG) \to \Wh_1(G;R) \to 0.
\]
There are  Rothenberg sequences for $j \in \{2,1,0,-1,\ldots \}$
\begin{multline}   \label{Rothenberg_sequences}
  \cdots \to L_n^{\langle j+1 \rangle}(RG,w) \to L_n^{\langle j \rangle}(RG,w) \to
  \widehat{H}^{n}(\IZ/2,\Wh_j(G;R))
  \\
  \to L_{n-1}^{\langle j +1 \rangle}(RG,w) \to L_{n-1}^{\langle j\rangle}(RG,w) \to
  \cdots.
\end{multline}
For a $\Z[\Z/2]$-module $A$ with $\Z/2$-action $a \mapsto \overline a$, we define the Tate cohomology
\[
\widehat H^n(\Z/2;A) = \frac{\{a \in A \mid \overline a = (-1)^n a\}}{\{a+ (-1)^n \overline a \mid a \in A\}}
\]
Moreover, the Shaneson splitting  gives  for $j \in \{2,1,0,-1,\ldots\} $ and $n \in \IZ$
\begin{equation}
L_n^{\langle j\rangle}(R[G \times \IZ], \pr^*w) \cong
 L_n^{\langle j\rangle}(RG, w) \oplus L_{n-1}^{\langle j-1 \rangle}(RG,w)
\label{Shaneson_splitting}
\end{equation}
for $\pr^*w  \colon G \times \IZ \xrightarrow{\pr} G \xrightarrow{w} \{\pm 1\}$.
One defines
\[
  L_n^{\langle -\infty \rangle}(RG,w) := \colim_{j \to -\infty} L_n^{\langle j\rangle}(RG,w).
\]
We have
\begin{equation}
L_n^{\langle -\infty \rangle}(R[G \times \IZ], \pr^*w) \cong
 L_n^{\langle -\infty \rangle}(RG, w) \oplus L_{n-1}^{\langle - \infty \rangle}(RG,w).
\label{Shaneson_splitting_infty}
\end{equation}

If $R = \IZ$, then $\Wh_1(G;R)$ is the classical Whitehead group $\Wh(G)$. Moreover,
$L^{\langle j \rangle}(\IZ G,w)$  agrees with the classical decorated $L$-groups
$L_n^s(\IZ G,w)$, $L_n^h(\IZ G,w)$, and $L_n^p(\IZ G,w)$ for $j = 2,1,0$.

\begin{remark}[Decorated $L$-theory is not compatible with finite products of rings]%
\label{rem:decorated_L-theory_is_not_compatible_with_finite_products_of_rings}
Note that decorated $L$-theory is compatible with finite products of rings only for $j = p$,
(or, equivalently, $j = 0$) but in general not for the other decorations. One can see the
problem for example for the decoration $j = h$, (or, equivalently, $j = 1$) from the
Rothenberg sequences, since the canonical map
$\widetilde{K}_0(S_1 \times S_2) \to \widetilde{K}_0(S_1) \times \widetilde{K}_0(S_2)$ for
two rings $S_1$ and $S_2$ is not bijective in general. All of this is due to the facts that
for two rings $S_1$ and $S_2$ the canonical functor
$\calp(S_1) \times \calp(S_2) \to \calp(S_1 \times S_2)$ is an equivalence of additive
categories, where $\calf(S_1) \times \calf(S_2) \to \calf(S_1 \times S_2)$
is \emph{not} an equivalence of additive categories,
since $S_1 \times \{0\}$ is not a free $S_1 \times S_2$-module.
\end{remark}

%%%%%%%%%%%%%%%%%%%%%%%%%%%%%%%%%%%%%%%%%%%%%%%%%%%%%%%%%%%%%%%%%%
  
\subsection{A construction of  Ranicki}%
\label{subsec:A_construction_of_Ranicki}
We need the following result of Ranicki~\cite[Proposition~2.5.1 on page~166]{Ranicki(1981)},
which is stated there only for rings but carries over to additive categories with involutions.

\begin{theorem}\label{the:Ranicki's-result}
  Let $U \colon \cala \to \calb$ be a functor of additive categories with involution.
  Consider $j \in \{1,0,-1, \ldots \}$. 
  Then one can construct a commutative diagram with long exact rows and columns

\begin{small}
  \[\xymatrix@!C=26mm{\vdots \ar[d]
      &
      \vdots \ar[d]
      &
      \vdots \ar[d]
      &
      \vdots \ar[d]
      \\
      \cdots \to 
      L_n^{\langle j+1 \rangle}(\cala) \ar[r]^-{U_*} \ar[d]
      &
      L_n^{\langle j+1 \rangle}(\calb) \ar[r] \ar[d]
      &
      L_n^{\langle j+1 \rangle}(U) \ar[r] \ar[d]
      &
      L_{n-1}^{\langle j +1 \rangle}(\cala)  \xrightarrow{U_*}  \cdots   \ar[d]
      \\
      \cdots \to 
      L_n^{\langle j \rangle}(\cala) \ar[r]^-{U_*} \ar[d]
      &
      L_n^{\langle j \rangle}(\calb) \ar[r] \ar[d]
      &
      L_n^{\langle j \rangle}(U) \ar[r] \ar[d]
      &
      L_{n-1}^{\langle j \rangle}(\cala) \xrightarrow{U_*} \cdots \ar[d]
      \\
      \cdots \to 
      \widehat{H}^n(\IZ/2;\widetilde{K}_j(\cala)) \ar[r]^-{U_*} \ar[d]
      &
      \widehat{H}^n(\IZ/2;\widetilde{K}_j(\calb))  \ar[r] \ar[d]
      &
      \widehat{H}^n(\IZ/2;\widetilde{K}_j(U))  \ar[r] \ar[d]
      &
      \widehat{H}^{n-1}(\IZ/2;\widetilde{K}_j(\cala)) \xrightarrow{U_*} \cdots\ar[d]
      \\
      \cdots \to
      L_{n-1}^{\langle j+1 \rangle}(\cala) \ar[r]^-{U_*} \ar[d]
      &
      L_{n-1}^{\langle j+1 \rangle}(\calb) \ar[r] \ar[d]
      &
      L_{n-1}^{\langle j+1 \rangle}(U) \ar[r] \ar[d]
      &
      L_{n-2}^{\langle j +1 \rangle}(\cala) \xrightarrow{U_*} \cdots\ar[d]
      \\
      \vdots 
      &
      \vdots
      &
      \vdots 
      &
      \vdots 
    }
  \]
\end{small}
\end{theorem}

%%%%%%%%%%%%%%%%%%%%%%%%%%%%%%%%%%%%%%%%%%%%%%%%%%%%%%%%%%%%%%%%%%
  
\subsection{A relative $L$-theory spectrum}%
\label{subsec:A_relative_theory_spectrum}

Given $j \in \{2,1,0,,-1 \ldots\}$, 
define the $(\Groupoids \downarrow \{\pm 1 \})$-spectrum $\bfL_{R,w}^{\langle j+1,j \rangle}$
to be the cofiber of the map of $(\Groupoids \downarrow \{\pm 1 \})$-spectra $\bfL_{R,w}^{\langle j+1 \rangle} \to
\bfL_{R,w}^{\langle j \rangle}$. Then we get for any group with orientation character $(G,w)$ and any morphism
$f \colon X \to Y$ of $G$-$CW$-complexes 
a commutative diagram with long exact rows and columns
\begin{equation}
  \begin{small}
\xymatrix@!C=25mm{\vdots \ar[d]
      &
      \vdots \ar[d]
      &
      \vdots \ar[d]
      &
      \vdots \ar[d]
      \\
      \cdots \to
      H_n^G(X,\bfL_{R,w}^{\langle j+1 \rangle}) \ar[r]^-{f_*} \ar[d]
      &
      H_n^G(Y,\bfL_{R,w}^{\langle j+1 \rangle}) \ar[r] \ar[d]
      &
      H_n^G(f,\bfL_{R,w}^{\langle j+1 \rangle}) \ar[r] \ar[d]
      &
      H_{n-1}^G(X,\bfL_{R,w}^{\langle j+1 \rangle}) \xrightarrow{f_*} \cdots \ar[d]
     \\
      \cdots \to
      H_n^G(X,\bfL_{R,w}^{\langle j \rangle}) \ar[r]^-{f_*} \ar[d]
      &
      H_n^G(Y,\bfL_{R,w}^{\langle j  \rangle}) \ar[r] \ar[d]
      &
      H_n^G(f,\bfL_{R,w}^{\langle j \rangle}) \ar[r] \ar[d]
      &
      H_{n-1}^G(X,\bfL_{R,w}^{\langle j \rangle}) \xrightarrow{f_*} \ar[d] \cdots 
      \\
      \cdots \to
      H_n^G(X,\bfL_{R,w}^{\langle j+1,j \rangle}) \ar[r]^-{f_*} \ar[d]
      &
      H_n^G(Y,\bfL_{R,w}^{\langle j+1,j  \rangle}) \ar[r] \ar[d]
      &
      H_n^G(f,\bfL_{R,w}^{\langle j+1,j \rangle}) \ar[r] \ar[d]
      &
      H_{n-1}^G(X,\bfL_{R,w}^{\langle j+1,,j \rangle}) \xrightarrow{f_*} \cdots \ar[d]
      \\
      \cdots \to
      H_{n-1}^G(X,\bfL_{R,w}^{\langle j+1 \rangle}) \ar[r]^-{f_*} \ar[d]
      &
      H_{n-1}^G(Y,\bfL_{R,w}^{\langle j+1 \rangle}) \ar[r] \ar[d]
      &
      H_{n-1}^G(f,\bfL_{R,w}^{\langle j+1 \rangle}) \ar[r] \ar[d]
      &
      H_{n-2}^G(X,\bfL_{R,w}^{\langle j+1 \rangle}) \xrightarrow{f_*} \cdots \ar[d]
      \\
      \cdots \to
      H_{n-1}^G(X,\bfL_{R,w}^{\langle j \rangle}) \ar[r]^-{f_*} \ar[d]
      &
      H_{n-1}^G(Y,\bfL_{R,w}^{\langle j  \rangle}) \ar[r] \ar[d]
      &
      H_{n-1}^G(f,\bfL_{R,w}^{\langle j \rangle}) \ar[r] \ar[d]
      &
      H_{n-2}^G(X,\bfL_{R,w}^{\langle j \rangle}) \xrightarrow{f_*} \cdots \ar[d]
      \\
      \vdots 
      &
      \vdots
      &
      \vdots 
      &
      \vdots
    }
  \end{small}
  \label{diagram_for_(j_plus_1,j)}
\end{equation}

In particular we get long exact sequences
\begin{multline}
\cdots \to H_n^G(\pt;\bfL_{R,w}^{\langle j+1 \rangle}) = L_n^{\langle j+1 \rangle}(RG,w)
  \to H_n^G(\pt;\bfL_{R,w}^{\langle j \rangle}) = L_n^{\langle j \rangle}(RG,w)
  \\
  \to H_n^G(\pt;\bfL_{R,w}^{\langle j+1, j \rangle}) = \pi_n(\bfL_{R,w}^{\langle j+1, j \rangle}(\wh G))
   \\
   \to H_{n-1}^G(\pt;\bfL_{R,w}^{\langle j+1 \rangle}) = L_{n-1}^{\langle j+1 \rangle}(RG,w)
   \to H_{n-1}^G(\pt;\bfL_{R,w}^{\langle j \rangle}) = L_{n-1}^{\langle j \rangle}(RG,w)
   \to \cdots,
   \label{Rothenberg_for_relative-spectrum}
 \end{multline}
 where $\wh G$ is the one-object groupoid associated to $G$.
 In view of the Rothenberg sequence~\eqref{Rothenberg_sequences}
 this leads to the very reasonable conjecture
 that there is a natural identification\footnote{Christian Kremer gives a proof of
   this equality~\eqref{conjecture_for_L(j_plus_1,j)} in his master thesis~\cite{Kremer(2022)}.}
 \begin{equation}
   H_n^G(\pt;\bfL_{R,w}^{\langle j+1,j \rangle}) = \widehat{H}^n(\IZ/2;\Wh_j(G;R)).
   \label{conjecture_for_L(j_plus_1,j)}
 \end{equation}
 If we would know this claim, this would make the exposition easier and more transparent.
 Actually, this claim will be proven
 in~\cite{Calmes-Dotto-Harpaz-Hebestreit-Land-Moi-Nardin-Nikolaus-Steimle(2024IV)}, where also
 complete constructions of the spectra $\bfL_{R,w}^{\langle j \rangle}$ will be presented and the spectra
 $\bfL_{R,w}^{\langle j+1,j \rangle}$ will be identified with the corresponding Tate spectra.
 (This is private communication with Markus Land.)
 
 Instead of using the unpublished work above, we take  a shortcut based on Theorem~\ref{the:Ranicki's-result}.
 The same attitude is taken in the proof of~\cite[Corollary~5.6]{Connolly-Davis-Khan(2015)}.
 There only rings are considered, which is problematic in view of the failure of
 decorated $L$-groups to be compatible with finite products of rings, see
 Remark~\ref{rem:decorated_L-theory_is_not_compatible_with_finite_products_of_rings}.
 We want to explain here that this problem can be solved by passing to additive categories
 with involution as explained in
 Subsection~\ref{subsec:Some_basics_about_K-and_L-theory_for_additive_categories_with_involution}.

 \begin{lemma}\label{lem:results_about_relative-spectrum}\
   \begin{enumerate}

   \item\label{lem:results_about_relative-spectrum_vanishing}
     The following assertions are equivalent for $j \in \{1,0,-1,-2, \ldots\}$:

     \begin{enumerate}
     \item  The abelian group $\widehat{H}^n(\IZ/2;\Wh_j(G;R))$ vanishes for all $n \in \IZ$;
     \item  The abelian group $H_n^G(\pt;\bfL_{R,w}^{\langle j +1, j\rangle})  \cong \pi_n(\bfL_{R,w}^{\langle j+1,j \rangle}(\wh G))$
       vanishes for all $n \in \IZ$;
      \end{enumerate}
    \item\label{lem:results_about_relative-spectrum:free_actions}
        If $X$ is a free $G$-$CW$-complex, then $H_n^G(X;\bfL_{R,w}^{\langle j+1,j \rangle}) = 0$ holds for all $n \in \IZ$.
      \end{enumerate}
    \end{lemma}
    \begin{proof}~\eqref{lem:results_about_relative-spectrum_vanishing}
      This follows from exact sequences~\eqref{Rothenberg_sequences} and~\eqref{Rothenberg_for_relative-spectrum},
      since both statements are equivalent to the assertion that the map $L_n^{\langle j+1 \rangle}(RG,w) \to L_n^{\langle j \rangle}(RG,w)$
      is bijective for all $n \in \IZ$.
      \\[1mm]~\eqref{lem:results_about_relative-spectrum:free_actions}
      We have $\Wh_j(\{1\},R) = 0 $ for $j \le 1$. Hence $\widehat{H}^n(\IZ/2;\Wh_j(\{1\};R)) = 0$ for all $n \in \IZ$.
      By assertion~\eqref{lem:results_about_relative-spectrum_vanishing} we get
      $H_n^G(G/\{1\};\bfL_{R,w}^{\langle j+1,j \rangle}) \cong \pi_n(\bfL_{R,w}^{\langle j +1,j\rangle}(\{1\})) = 0$
    for all $n \in \IZ$. This implies by the equivariant Atiyah-Hirzebruch spectral sequence~\cite[Theorem 4.7]{Davis-Lueck(1998)}
    that $H_n^G(X;\bfL_{R,w}^{\langle j+1,j \rangle}) = 0$ holds for all $n \in \IZ$, if $X$ is a free $G$-$CW$-complex.
  \end{proof}

 %%%%%%%%%%%%%%%%%%%%%%%%%%%%%%%%%%%%%%%%%%%%%%%%%%%%%%%%%%%%%%%%%% 
  
\subsection{Computing $L$-groups}\label{subsec:Computing_L_groups}

\begin{theorem}\label{the:computing_L-groups}
  Let $R$ be a ring with involution. Let $\Gamma$ be a group coming with a group
  homomorphism $w \colon \Gamma \to \{ \pm 1\}$. 
  Let $\calm$ be a complete system of representatives of the conjugacy classes of maximal
  finite subgroups and let $\calvII$ be a complete system of representatives of the
  conjugacy classes of maximal virtually cyclic subgroups of type {II}.
  Suppose that the following conditions are satisfied:

    \begin{itemize}
  \item The group $\Gamma$ satisfies conditions (M), (NM), and (V$_{\operatorname{II}}$), see
    Definition~\ref{def:conditions_on_Gamma_intro};
  \item The group $\Gamma$ is  a Farrell-Jones group;
  \item There exists $j_0 \in \IZ$  such that $\Wh_j(H;R) = 0$ holds for every finite
    subgroup $H \subseteq \Gamma$ and every $j \le j_0$;
  \item The ring $R$ is regular, $K_n(R) = 0$ for $n < 0$, and $K_0(\Z) \xrightarrow{\cong} K_0(R)$,
    e.g., $R$ is a principal ideal domain;
  \end{itemize}
  
    Consider any $j \in \{2,1,0, -1, \ldots\} \amalg \{-\infty \}$. Then:

  \begin{enumerate}
  \item\label{the:computing_L-groups:edub_pt}
   The map induced by the projection $\edub{\Gamma} \to \pt$ induces   an isomorphism for every $n \in \IZ$
    \[H_n^{\Gamma}(\edub{\Gamma};\bfL_{R,w}^{\langle j \rangle})
    \xrightarrow{\cong}  
    H_n^{\Gamma}(\pt;\bfL_{R,w}^{\langle j \rangle}) = L_n^{\langle j \rangle}(R\Gamma,w);
  \]
 
\item\label{the:computing_L-groups:eub(Gamma)_and_edub(Gamma)}
  For every $n \in \IZ$ we have the short split exact sequence
  \[0 \to H_n^{\Gamma}(\eub{\Gamma} ;\bfL_{R,w}^{\langle -\infty \rangle}) \to 
    H_n^{\Gamma}(\edub{\Gamma} ;\bfL_{R,w}^{\langle -\infty  \rangle}) \to
    H_n^{\Gamma}(\eub{\Gamma} \to \edub{\Gamma} ;\bfL_{R,w}^{\langle -\infty \rangle}) \to
    0;
  \]

\item\label{the:computing_L-groups:eub(Gamma)_relative_terms_as_bigoplus}
We obtain  an isomorphism for any $n \in \IZ$
  \begin{eqnarray*}
   \bigoplus_{F \in \calm} H_n^{F}(EF\to \pt;\bfL_{R,w|_F}^{\langle j \rangle})
   & \xrightarrow{\cong} &
   H_n^{\Gamma}(E\Gamma \to \eub{\Gamma} ;\bfL_{R,w}^{\langle j \rangle});
    \\
    \bigoplus_{V \in \calvII} H_n^{V}(\eub{V} \to \pt;\bfL_{R,w|_V}^{\langle j \rangle})
    & \xrightarrow{\cong} &
                            H_n^{\Gamma}(\eub{\Gamma} \to \edub{\Gamma};\bfL_{R,w}^{\langle j \rangle});
  \end{eqnarray*}
For every $V \in \calvII$ and $n \in \IZ$  the canonical map
  \[H_n^{V}(\eub{V} \to \pt;\bfL_{R,w|_V}^{\langle j \rangle}) \xrightarrow{\cong}
    H_n^{V}(\eub{V} \to \pt;\bfL_{R,w|_V}^{\langle -\infty\rangle})
  \]
  is bijective. 
\end{enumerate}
\end{theorem}
\begin{proof}~\eqref{the:computing_L-groups:edub_pt}   
     Since $\Gamma$ is a Farrell-Jones group,
  the map
   \[H_n^{\Gamma}(\edub{\Gamma};\bfL_{R,w}^{\langle -\infty \rangle})
    \xrightarrow{\cong}  
    H_n^{\Gamma}(\pt;\bfL_{R,w}^{\langle -\infty \rangle}) = L^{\langle - \infty \rangle}_n(R\Gamma,w)
  \]
  is bijective for all $n \in \IZ$.  This takes care of the case $j = -\infty$.

  We can assume without loss of generality that $j_0 \le -1$, otherwise replace $j_0$ by
  $-1$.  Next we prove assertions~\eqref{the:computing_L-groups:edub_pt} for
  $j \in \{j_0,j_0-1, j_0 -2, \ldots\}$.

  By assumption $\Wh_j(H;R)$ vanishes for every  $j \le j_0$ and every finite subgroup
  $H \subseteq \Gamma$.  We conclude from
  Lemma~\ref{lem:consequences_of_the_conditions_(M),(NM)_(F)} that every infinite
  virtually cyclic subgroup of $\Gamma$ is either infinite cyclic or isomorphic to
  $D_{\infty}$.  If we take $\Gamma = \IZ$ or $D_{\infty}$, the assumptions of
  Theorem~\ref{the:computing_K-groups} are satisfied and therefore $\Wh_j(W;R)$ vanishes
  for every  $j \le j_0$ and every  virtually cyclic subgroup $W$ of $\Gamma$.
  We conclude from the Rothenberg sequence~\eqref{Rothenberg_sequences}  that  the
  canonical map $L_n^{\langle j+1 \rangle}(RW,w|_W) \to L_n^{\langle j \rangle}(RW,w|_W)$
  is bijective for every $n \in \IZ$, $j \le j_0$ and every  virtually cyclic subgroup $W$ of
  $\Gamma$.  Hence    the canonical map of spectra
  $\bfL_{R,w}^{\langle j \rangle}(\Gamma/W) \to \bfL_{R,w}^{\langle
    -\infty\rangle}(\Gamma/W)$ is a weak homotopy equivalence for all virtually cyclic
  subgroups $W$ of $\Gamma$ and all $j \le j_0$.  Since all isotropy groups of
  $\edub{\Gamma}$ are virtually cyclic, the canonical map
  \[
    H_n^{\Gamma}(\edub{\Gamma},\bfL_{R,w}^{\langle j \rangle}) \xrightarrow{\cong}
    H_n^{\Gamma}(\edub{\Gamma},\bfL_{R,w}^{\langle -\infty \rangle})
  \]
  is bijective for $n \in \IZ$ and $j \le j_0$. Since $R$ is regular by assumption and
  hence $\Wh_j(\Gamma;R)$ vanishes for $j \le j_0$ by
  Theorem~\ref{the:computing_K-groups}, the canonical map
  \[
    H_n^{\Gamma}(\pt,\bfL_{R,w}^{\langle j \rangle}) = L_n^{\langle j \rangle}(R\Gamma,w)
    \xrightarrow{\cong} H_n^{\Gamma}(\pt,\bfL_{R,w}^{\langle -\infty \rangle}) = L_n^{\langle
      -\infty \rangle}(R\Gamma,w)
  \]
  is bijective for $n \in \IZ$ by the Rothenberg sequence~\eqref{Rothenberg_sequences}.
  We conclude that assertion~\eqref{the:computing_L-groups:edub_pt} holds for
  $j \in \{j_0,j_0-1, j_0 -2, \ldots\}$, since we have already proved it for
  $j = - \infty$.

It remains to show for $j \in \IZ$ with $j \le 1$ that
assertion~\eqref{the:computing_L-groups:edub_pt}
holds for $j+1$, if it  holds for $j$. This is done as follows.

We get from excision  and 
Theorem~\ref{the:constructing_underline(E)Gamma)}~\eqref{the:constructing_underline(E)Gamma):Fin}
the isomorphism
\[
\bigoplus_{F \in \calm} H_n^{\Gamma}(\Gamma \times_F EF \to \Gamma/F,\bfL_{R,w|_F}^{\langle j+1,j \rangle})
\xrightarrow{\cong}
  H_n^{\Gamma}(E\Gamma \to \eub{\Gamma},\bfL_{R,w|_F}^{\langle j+1,j \rangle}).
\]
The groups $H_n^{\Gamma}(E\Gamma,\bfL_{R,w|_F}^{\langle j+1,j \rangle})$
and $H_n^{\Gamma}(\Gamma \times_F EF,\bfL_{R,w|_F}^{\langle j+1,j \rangle})$ vanish
for $n \in \IZ$ by
Lemma~\ref{lem:results_about_relative-spectrum}~\eqref{lem:results_about_relative-spectrum:free_actions},
as  $\Gamma$ acts freely on $E\Gamma$ and on $\Gamma \times_F EF$.
This implies that the canonical maps
\[
H_n^{\Gamma}(\Gamma/F,\bfL_{R,w}^{\langle j+1,j \rangle}) \xrightarrow{\cong}
H_n^{\Gamma}(\Gamma \times_F   EF \to  \Gamma/F,\bfL_{R,w}^{\langle j+1,j \rangle})
\]
and
\[
H_n^{\Gamma}(\eub{\Gamma},\bfL_{R,w}^{\langle j+1,j \rangle})
\xrightarrow{\cong} 
H_n^{\Gamma}(E\Gamma \to \eub{\Gamma},\bfL_{R,w}^{\langle j+1,j \rangle})
\]
are isomorphisms. Hence we get isomorphisms
\[
\bigoplus_{F \in \calm} H_n^{\Gamma}(\Gamma/F,\bfL_{R,w}^{\langle j+1,j \rangle})
\xrightarrow{\cong}
  H_n^{\Gamma}(\eub{\Gamma},\bfL_{R,w}^{\langle j+1,j \rangle}).
\]
This shows that the up to $\Gamma$-homotopy unique map $f \colon \coprod_{F \in \calm} \Gamma/F \to \eub{\Gamma}$
induces for $n \in \IZ$ an isomorphism
\begin{equation}
  H_n^{\Gamma}(f;\bfL_{R,w}^{\langle j+1,j \rangle}) \colon
  H_n^{\Gamma}( \coprod_{F \in \calm} \Gamma/F;\bfL_{R,w}^{\langle j+1,j \rangle}) \xrightarrow{\cong}
 H_n^{\Gamma}(\eub{\Gamma};\bfL_{R,w}^{\langle j+1,j \rangle}).
 \label{f_for_j_plus_1,j_is_iso}
\end{equation}
Let $p \colon \coprod_{F \in \calm} \Gamma/F \to \pt$ be the projection. Next we want to show that
it induces for all $n \in \IZ$ an isomorphism
\begin{equation}
  p_n \colon 
  H_n^{\Gamma}( \coprod_{F \in \calm} \Gamma/F;\bfL_{R,w|_F}^{\langle j+1,j \rangle}) \xrightarrow{\cong}
 H_n^{\Gamma}(\pt;\bfL_{R,w}^{\langle j+1,j \rangle}).
 \label{p_for_j_plus_1,j_is_iso}
\end{equation}

Put $X = \coprod_{F \in \calm} \Gamma/F$. The following commutative diagram with exact rows
\begin{small}
\[
\xymatrix@!C=22mm{\cdots \to 
      H_n^{\Gamma}(X;\bfL^{\langle j+1 \rangle}_{R,w})
      \ar[r]^-{p_*} \ar[d]
      &
      H_n^{\Gamma}(\pt;\bfL^{\langle j+1 \rangle}_{R,w}) \ar[r] \ar[d]
      &
      H_n^{\Gamma}(p;\bfL^{\langle j+1 \rangle}_{R,w})\ar[r] \ar[d]
      &
      H_{n-1}^{\Gamma}(X;\bfL^{\langle j+1 \rangle}_{R,w}) \xrightarrow{p_*} \cdots \ar[d]
      \\
      \cdots \to 
      H_{n}^{\Gamma}(X;\bfL^{\langle j \rangle}_{R,w})
      \ar[r]^-{p_*} 
      &
      H_{n}^{\Gamma}(\pt;\bfL^{\langle j \rangle}_{R,w}) \ar[r] 
      &
      H_{n}^{\Gamma}(p;\bfL^{\langle j \rangle }_{R,w})\ar[r] 
      &
      H_{n-1}^{\Gamma}(X;\bfL^{\langle j \rangle}_{R,w}) \xrightarrow{p_*} \cdots
      }
    \]
  \end{small}%
can be identified with the first two rows in the commutative diagram with exact rows
    and columns of Theorem~\ref{the:Ranicki's-result}, if we take
    $\cala = \bigoplus _{F \in \calm} \calf(RF)$, $\calb = \calf(R\Gamma)$, and
    $U \colon \cala \to \calb$ to be
    $\bigoplus_{F \in \calm} \calf(i_F) \colon \bigoplus_{F\in \calm} \calf(RF) \to
    \calf(R\Gamma)$ for $i_F \colon RF \to R\Gamma$ the ring homomorphism induced by
    the inclusion $F \to \Gamma$.   This
    identification uses~\eqref{L-theory_and_infinite_sums}. The third row in the
    commutative diagram with exact rows and columns of Theorem~\ref{the:Ranicki's-result}
    can be written as
    \begin{multline*}
      \cdot \to \widehat{H}^n(\IZ/2;\bigoplus_{F \in \calm}\Wh_j(F;R))
      \to \widehat{H}^n(\IZ/2; \Wh_j(\Gamma;R))
       \to \widehat{H}^n(\IZ/2;U)
       \\
       \to
        \widehat{H}^{n-1}(\IZ/2;\bigoplus_{F \in \calm}\Wh_j(F;R))
      \to \widehat{H}^{n-1}(\IZ/2;\Wh_j(\Gamma;R))
      \to \cdots.
    \end{multline*}
    This identification uses~\eqref{K-theory_and_infinite_sums}.  Since
    $\bigoplus_{F \in \calm}\Wh_j(F;R) \to \Wh_j(\Gamma;R)$ is an isomorphism for all
    $n \in \IZ$ by Theorem~\ref{the:computing_K-groups}, the first and the third arrow
    appearing in the long exact sequence above is bijective for all $n \in \IZ$. Hence
    $\widehat{H}^n(\IZ/2;U)$ vanishes for all $n \in \IZ$.  We conclude that  the map
    $H_n^{\Gamma}(p;\bfL_{R,w}^{\langle j+1 \rangle}) \xrightarrow{\cong}
    H_n^{\Gamma}(p;\bfL_{R,w}^{\langle j \rangle})$ is an isomorphism for all $n \in \IZ$.  This
    implies that $H_n^{\Gamma}(p;\bfL_{R,w}^{\langle j+1,j \rangle})$ vanishes for all
    $n \in \IZ$.  Hence~\eqref{p_for_j_plus_1,j_is_iso} is an isomorphism.  Note that the
    proof of the bijectivity of~\eqref{p_for_j_plus_1,j_is_iso} would be rather easy, if we
    would know~\eqref{conjecture_for_L(j_plus_1,j)}.

    We conclude from~\eqref{f_for_j_plus_1,j_is_iso} and~\eqref{p_for_j_plus_1,j_is_iso}
    that the map
    \begin{equation}
    H_n^{\Gamma}(\eub{\Gamma};\bfL_{R,w}^{\langle j+1,j \rangle}) \xrightarrow{\cong}
 H_n^{\Gamma}(\pt;\bfL_{R,w}^{\langle j+1,j \rangle})
 \label{eub_to_pt_for_j_plus_1,j_is_iso}
\end{equation}
is bijective for all $n \in \IZ$.

Next we show
   \begin{equation}
     H_n^{V}(\eub{V} \to \pt;\bfL_{R,w|_V}^{\langle j+1,j \rangle}) = 0\quad 
     \text{for}\; n \in \IZ, V \in \calv.
     \label{Vanishing_of_H_n_upper_V(eub(V)_to_pt;bfL_(R,w)_upper_(langle_j_plus_1,j_rangle))}
   \end{equation}
   Recall from Lemma~\ref{lem:consequences_of_the_conditions_(M),(NM)_(F)}
   that any virtually cyclic subgroup of $\Gamma$ is infinite
   cyclic or isomorphic to $D_{\infty}$.  Hence $V$ satisfies the assumptions of
   Theorem~\ref{the:computing_L-groups}, if we take $\Gamma = V$.
   Now~\eqref{Vanishing_of_H_n_upper_V(eub(V)_to_pt;bfL_(R,w)_upper_(langle_j_plus_1,j_rangle))}
   follows from  the isomorphism~\eqref{eub_to_pt_for_j_plus_1,j_is_iso}, which we have already established.

   Next we show that
    \begin{equation}
    H_n^{\Gamma}(\eub{\Gamma};\bfL_{R,w}^{\langle j+1,j \rangle}) \xrightarrow{\cong}
 H_n^{\Gamma}(\edub{\Gamma};\bfL_{R,w}^{\langle j+1,j \rangle})
 \label{eub__edub_for_j_plus_1,j_is_iso}
\end{equation}
is an isomorphism for all $n \in \IZ$. We get from
Theorem~\ref{the:constructing_underline(E)Gamma)}~\eqref{the:constructing_underline(E)Gamma):Vcyc}
an isomorphism
\[
 \bigoplus_{V\in \calv} H_n^{\Gamma}(\Gamma \times_V\eub{V}
  \to \Gamma \times_V \pt;\bfL_{R,w|_V}^{\langle j+1,j \rangle}) 
\xrightarrow{\cong}
 H_n^{\Gamma}(\eub{\Gamma} \to \edub{\Gamma};\bfL_{R,w}^{\langle j+1,j \rangle}).
   \]
   Using the induction structure of the equivariant homology theory
   $\calh^?_*(-;\bfL_{R,w}^{\langle j+1,j \rangle})$, we get isomorphisms
   \[
     H_n^{V}(\eub{V} \to \pt;\bfL_{R,w}^{\langle j+1,j \rangle})
     \xrightarrow{\cong}
     H_n^{\Gamma}(\Gamma \times_V\eub{V} \to \Gamma \times_V \pt;\bfL_{R,w}^{\langle j+1,j \rangle})
   \]
   for every $V \in \calv$. Hence~\eqref{Vanishing_of_H_n_upper_V(eub(V)_to_pt;bfL_(R,w)_upper_(langle_j_plus_1,j_rangle))}
   implies   that~\eqref{eub__edub_for_j_plus_1,j_is_iso} is bijective.

   We conclude from~\eqref{eub_to_pt_for_j_plus_1,j_is_iso}
   and~\eqref{eub__edub_for_j_plus_1,j_is_iso} that the map
   \begin{equation}
     H_n^{\Gamma}(\edub{\Gamma};\bfL_{R,w}^{\langle j+1,j \rangle}) \xrightarrow{\cong}
     H_n^{\Gamma}(\pt;\bfL_{R,w}^{\langle j+1,j \rangle})
     \label{edub_to_pt_for_j_plus_1,j_is_iso}
   \end{equation}
   is bijective for all $n \in \IZ$.

By the induction hypothesis the map
\[H_n^{\Gamma}(\edub{\Gamma};\bfL_{R,w}^{\langle j \rangle})
     \xrightarrow{\cong} 
    H_n^{\Gamma}(\pt;\bfL_{R,w}^{\langle j \rangle})
  \]
  is bijective for all $n \in \IZ$. Now we conclude from~\eqref{edub_to_pt_for_j_plus_1,j_is_iso},
  the Five-Lemma and the long exact
  sequence given by the first two columns in~\eqref{diagram_for_(j_plus_1,j)}  applied to the projection
  $f \colon X = \edub{\Gamma} \to Y= \pt$ that also the map
  \[H_n^{\Gamma}(\edub{\Gamma};\bfL_{R,w}^{\langle j+1 \rangle})
    \xrightarrow{\cong} H_n^{\Gamma}(\pt;\bfL_{R,w}^{\langle j+1 \rangle})
  \]
  is bijective for all $n \in \IZ$.
This finishes the proof of the induction step and hence  of
assertion~\eqref{the:computing_L-groups:edub_pt}.
\\[1mm]~\eqref{the:computing_L-groups:eub(Gamma)_and_edub(Gamma)}
 We conclude from~\cite[Corollary~3.27]{Davis-Khan-Ranicki(2011)} and the
  Bass-Heller-Swan decomposition that $\Wh_j(D_{\infty};R) = 0$ and $\Wh_j(\IZ;R) = 0$ holds
  for $j \le j_0$.  Since any infinite cyclic subgroup of $\Gamma$ is infinite cyclic or
  isomorphic to $D_{\infty}$ by
  Lemma~\ref{lem:consequences_of_the_conditions_(M),(NM)_(F)}, we get $\Wh_j(W;R) = 0$ for
  $j \le j_0 \le -1 $ for any virtually cyclic subgroup $W$ of $\Gamma$.  Since $R$ is
  regular, $K_j(R) = 0$ for $j \le j_0$. This implies $H^W_j(EW;\bfK_R) = 0$ for $j \le j_0$
  by a spectral sequence argument. Hence $K_j(RW) = H^W_j(\pt;\bfK_R) \cong H^W_j(EW \to \pt;\bfK_R) \cong \Wh_j(W;R) = 0$
  holds for $j \le j_0$. Therefore
  we obtain from~\cite[Section~1]{Bartels(2003b)} the desired short split exact sequence.
  \\[1mm]~\eqref{the:computing_L-groups:eub(Gamma)_relative_terms_as_bigoplus}
  We get from excision  and 
Theorem~\ref{the:constructing_underline(E)Gamma)}~\eqref{the:constructing_underline(E)Gamma):Fin}
the first desired isomorphism
\begin{eqnarray*}
   \bigoplus_{F \in \calm} H_n^{F}(EF\to \pt;\bfL_{R,w|_F}^{\langle j \rangle})
   & \xrightarrow{\cong} &
   H_n^{\Gamma}(E\Gamma \to \eub{\Gamma} ;\bfL_{R,w}^{\langle j \rangle}).
\end{eqnarray*}

Next we show
\begin{equation}
  H_n^{V}(\eub{V} \to \pt;\bfL_{R,w|_V}^{\langle j \rangle}) = 0 \quad \text{for}\; n \in \IZ, V \in \calvI.
\label{H_n_upper_V(eub(V)_to_pt;bfL_(R,w)_upper_(langle_j_rangle))}
\end{equation}
This is obvious, if $V$ is finite. 
We conclude from
Lemma~\ref{lem:consequences_of_the_conditions_(M),(NM)_(F)} that every $V \in \calvI$  is infinite cyclic.
Hence it suffices to treat the case where $V$ is infinite cyclic. 

The assembly map
$H_n^{V}(\eub{V};\bfL_{R,w|_V}^{\langle -\infty\rangle}) \to H_n^{V}(\pt;\bfL_{R,w|_V}^{\langle -\infty\rangle})$
is bijective for all $n \in \IZ$. This follows from~\cite[Theorem~13.56]{Lueck(2022book)}  which is in this case essentially
the Shaneson-splitting. Hence
$H_n^{V}(\eub{V} \to \pt;\bfL_{R,w|_V}^{\langle -\infty \rangle})$ vanishes  for $n \in \IZ$.

Next we show that $H_n^{V}(\eub{V} \to \pt;\bfL_{R,w|_V}^{\langle j \rangle}) = 0$ holds for $j \in \{2,1,0,-1, \ldots\}$
and $n \in \IZ$.  For this purpose it suffices to show that
$H_n^{V}(\eub{V} \to \pt;\bfL_{R,w|_V}^{\langle j+1,j \rangle}) = 0$ holds for $j \in \{1,0,-1, \ldots\}$,
since we have already proved the claim for $j = -\infty$.
Since $V$ acts freely on $\eub{V}$, we conclude from
Lemma~\ref{lem:results_about_relative-spectrum}~\eqref{lem:results_about_relative-spectrum:free_actions}
\[
H_n^{V}(\eub{V} \to \pt;\bfL_{R,w|_V}^{\langle j+1,j \rangle})  \cong H_n^{V}(\pt;\bfL_{R,w|_V}^{\langle j+1,j \rangle}).
\]
Because of Lemma~\ref{lem:results_about_relative-spectrum}~\eqref{lem:results_about_relative-spectrum_vanishing}
it suffices to show $\widehat{H}^n(\IZ/2,\Wh_j(V;R)) = 0$.
This follows from the conclusion of
Theorem~\ref{the:computing_K-groups} that $\Wh_j(V;R)$ vanishes. This finishes the proof
of~\eqref{H_n_upper_V(eub(V)_to_pt;bfL_(R,w)_upper_(langle_j_rangle))}.

Next we prove
\begin{equation}
  H_n^{\Gamma}(\eub{\Gamma} \to \EGF{\Gamma}{\calvcyc_I};\bfL_{R,w|_V}^{\langle j \rangle}) = 0 \quad \text{for}\; n \in \IZ.
\label{H_n_upper_V(eub(V)_to_EGF(G)(calvycyc_I);bfL_(R,w)_upper_(langle_j_rangle))}
\end{equation}
For a $\Gamma$-$CW$-complex $Z$ let $\pr_Z \colon \eub{\Gamma} \times Z\to Z$ be the
projection. Since the projection $\eub{\Gamma} \times \EGF{\Gamma}{\calvcyc_I} \to \eub{\Gamma}$ is
a $\Gamma$-homotopy equivalence, it suffices to show
$H_n^{\Gamma}(\pr_{\EGF{\Gamma}{\calvcyc_I}};\bfL_{R,w|_V}^{\langle j \rangle}) = 0$ for all
$n \in \IZ$.  We will show more generally for any $\Gamma$-$CW$-complex $Z$, whose isotropy
groups belong to $\calvcyc_I$, that $H_n^{V}(\pr_Z;\bfL_{R,w|_V}^{\langle j \rangle})$ vanishes
for all $n \in \IZ$. By a colimit argument and induction over the skeletons this claim
can be reduced to the special case $Z = \Gamma/V$ for $V \in \calvcyc_I$. Then
$\eub{\Gamma} \times \Gamma/V$ is $\Gamma$-homeomorphic to $\Gamma \times_V \eub{V}$ and
$\pr_{\Gamma_V}$ is the induction with the inclusion $V \to \Gamma$ applied to
$\eub{V} \to \pt$.  By the induction structure it remains to show
$H_n^V(\eub{V} \to \pt;\bfL_{R,w|_V}^{\langle j \rangle}) = 0$ for all $n \in \IZ$ and all $V \in \calvcyc_I$,
what we have already done,
see~\eqref{H_n_upper_V(eub(V)_to_pt;bfL_(R,w)_upper_(langle_j_rangle))}.  This finishes
the proof
of~\eqref{H_n_upper_V(eub(V)_to_EGF(G)(calvycyc_I);bfL_(R,w)_upper_(langle_j_rangle))}.

  From~\eqref{H_n_upper_V(eub(V)_to_EGF(G)(calvycyc_I);bfL_(R,w)_upper_(langle_j_rangle))}
  we obtain an isomorphism
  \[H_n^{\Gamma}(\eub{\Gamma} \to \edub{\Gamma};\bfL_{R,w}^{\langle j \rangle})
    \xrightarrow{\cong}
    H_n^{\Gamma}(\EGF{\Gamma}{\calvcyc_I}\to \edub{\Gamma};\bfL_{R,w}^{\langle j \rangle}).
  \]
We conclude from Lemma~\ref{lem:consequences_of_the_conditions_(M),(NM)_(F)}
that $\Gamma$ satisfies  (V$_{\operatorname{II}}$) and  (NV$_{\operatorname{II}}$).
We get from excision  and 
Theorem~\ref{the:constructing_underline(E)Gamma)}~\eqref{the:constructing_underline(E)Gamma):Vcyc_II}
isomorphisms
\begin{eqnarray*}
    \bigoplus_{V \in \calvII} H_n^{V}(\eub{V} \to \pt;\bfL_{R,w|_V}^{\langle j \rangle})
    & \xrightarrow{\cong} &
    H_n^{\Gamma}(\EGF{\Gamma}{\calvcyc_I}\to \edub{\Gamma};\bfL_{R,w}^{\langle j \rangle}).
\end{eqnarray*}
Hence we obtain for every $n \in \IZ$ the second desired isomorphism
\begin{eqnarray*}
    \bigoplus_{V \in \calvII} H_n^{V}(\eub{V} \to \pt;\bfL_{R,w|_V}^{\langle j \rangle})
    & \xrightarrow{\cong} &
    H_n^{\Gamma}(\eub{\Gamma}\to \edub{\Gamma};\bfL_{R,w|_V}^{\langle j \rangle}).
\end{eqnarray*}
  For any $V \in \calvII$ the canonical map
  \[H_n^{V}(\eub{V} \to \pt;\bfL_{R,w|_V}^{\langle j \rangle}) \xrightarrow{\cong}
    H_n^{V}(\eub{V} \to \pt;\bfL_{R,w|_V}^{\langle -\infty\rangle})
  \]
  is bijective, since we have already shown that $H_n^{V}(\eub{V} \to \pt;\bfL_{R,w|_V}^{\langle j+1,j \rangle})$
  vanishes for all $n \in \IZ$ and $j \in  \{1,0,-1, \ldots\}$ in~\eqref{eub__edub_for_j_plus_1,j_is_iso}.
  This finishes the proof of Theorem~\ref{the:computing_L-groups}.
\end{proof}

\begin{remark}[Identification with $\UNil$-groups]%
\label{rem:Identification_with_UNil-groups}
  The groups $H_n^{V}(\eub{V} \to \pt;\bfL_{R,w}^{\langle j \rangle})$
    appearing in
    assertion~\eqref{the:computing_L-groups:eub(Gamma)_relative_terms_as_bigoplus} of
    Theorem~\ref{the:computing_L-groups} for $V \in \calvII$, which has to be isomorphic
    to $D_{\infty}$, are independent of the decoration $j$,
     and can be identified with Cappell's $\UNil$-groups
     $\UNil_n(R;R^{\pm 1},R^{\pm 1})$, where the signs $\pm 1$ come from the orientation character
     $w|V$. 
     % \commentw{Add reference or explanation.
     %  In~\cite[Lemma~6.4]{Connolly-Davis-Khan(2015)} a reference
     %  to~\cite[Lemma~4.6]{Connolly-Davis-Khan(2014H1)} is given, which I could not find. Is it
     %  Lemma~4.4? Note that there only $R = \IZ$ is treated.}
      If both signs are $+1$, they will be omitted from the notation. In the case
    $R = \IZ$ the groups $\UNil_n(\IZ;\IZ,\IZ)$ have been computed by Banagl, Connolly, Davis,
    Kozniewski, and Ranicki,
    see~\cite{Connolly-Davis(2004),Banagl-Ranicki(2006),Connolly-Ranicki(2005),Connolly-Kozniewski(1995)},
    \[
      \UNil_n(\IZ;\IZ,\IZ) \cong \begin{cases}
        \{0\} & n \equiv 0 \mod (4);
        \\
        \{0\} & n \equiv 1 \mod (4);
        \\
        (\IZ/2)^{\infty}& n \equiv 2 \mod (4);
        \\
        (\IZ/2 \oplus \IZ/4)^{\infty}& n \equiv 3 \mod (4).
      \end{cases}
    \]
   There is an isomorphism $\UNil_{n}(\IZ;\IZ,\IZ) \xrightarrow{\cong} \UNil_{n+2} (\IZ;\IZ^{-1},\IZ^{-1})$ for all $n$ (see~\cite[page 1118]{Cappell(1974c)}).
  \end{remark}

  \begin{remark}\label{rem:Bartels-splitting}
   We do not know, whether assertion~\eqref{the:computing_L-groups:eub(Gamma)_and_edub(Gamma)}
in Theorem~\ref{the:computing_L-groups} does hold also for other decorations than $-\infty$.
\end{remark}

\begin{remark}\label{rem:FJ_not_true_for_all_decorations}
  One interesting feature of
  Theorem~\ref{the:computing_L-groups}~\eqref{the:computing_L-groups:edub_pt} is that it
  holds for all decorations.  Note that the $L$-theoretic Farrell-Jones Conjecture is
  formulated for the decoration $\langle -\infty \rangle$ only.  Indeed, there are
  counterexamples for the decorations $s$, $h$ and $p$,
  see~\cite{Farrell-Jones-Lueck(2002)}.
\end{remark}

\begin{remark}\label{rem:computation_of_the_underline(Gamma)_part}
  Assume in the sequel that the orientation homomorphism is trivial.
  Then the   computation of $L_n^{\langle j \rangle}(R\Gamma,w)$ boils down by
  Theorem~\ref{the:computing_L-groups}~\eqref{the:computing_L-groups:edub_pt}
  to the computation of the terms
  $H_n^{\Gamma}(\eub{\Gamma};\bfL_{R}^{\langle j \rangle})$ and
  $H_n^{\Gamma}(\eub{\Gamma} \to \edub{\Gamma};\bfL_{R}^{\langle j \rangle})$, modulo a
  possible extension problem unless $j = -\infty$, see
  Theorem~\ref{the:computing_L-groups}~\eqref{the:computing_L-groups:eub(Gamma)_and_edub(Gamma)}
  and Remark~\ref{rem:Bartels-splitting}. The computation of
  $H_n^{\Gamma}(\eub{\Gamma} \to \edub{\Gamma};\bfL_{R}^{\langle j \rangle})$ is
  complete by
  Theorem~\ref{the:computing_L-groups}~\eqref{the:computing_L-groups:eub(Gamma)_relative_terms_as_bigoplus}
  and Remark~\ref {rem:Identification_with_UNil-groups}, if $R = \IZ$. Some information about
  $H_n^{\Gamma}(\eub{\Gamma};\bfL_{R}^{\langle j \rangle})$ is given in
  Theorem~\ref{the:computing_L-groups}~\eqref{the:computing_L-groups:eub(Gamma)_relative_terms_as_bigoplus}.
  One can do a little better than this. Namely, there exists the following long exact
  sequence (see~\cite[Lemma 7.2(ii)]{Davis-Lueck(2013)})
  \begin{multline*}
    \cdots \to H_{n+1}(\bub{\Gamma};\bfL(R)) \to \bigoplus_{F \in \calm}
    \widetilde{L}^{\langle j \rangle}_n(RF) \to
    H_n^{\Gamma}(\eub{\Gamma};\bfL_R^{\langle j \rangle})
    \\
    \to H_{n}(\bub{\Gamma};\bfL^{\langle j \rangle}(R)) \to \bigoplus_{F \in \calm}
    \widetilde{L}^{\langle j \rangle}_n(RF) \to \cdots.
  \end{multline*}
  Here $\bub{\Gamma}$ is $\Gamma \backslash \eub{\Gamma}$ and
  $H_{n}(\bub{\Gamma};\bfL^{\langle j \rangle}(R))$ is its homology with respect to the 
  $L$-theoretic ring spectrum with decoration $\langle j \rangle$ of the ring $R$. The group
  $\widetilde{L}^{\langle j \rangle}_n(RF)$ is  defined to be the kernel of the map
  $L^{\langle j \rangle}_n(RF) \to L^{\langle j \rangle}_n(R)$ coming from the group
  homomorphism $F \to \{1\}$. The composite
  \[
  \bigoplus_{F \in \calm}\widetilde{L}^{\langle j \rangle}_n(RF) \to  H_n^{\Gamma}(\eub{\Gamma};\bfL_R^{\langle j \rangle})
  \to H_n^{\Gamma}(\pt;\bfL_R^{\langle j \rangle}) = L_n^{\langle j \rangle}(R\Gamma)
\]
is the direct sum of the maps induced by the various inclusions $F \to \Gamma$.  The map
$H_n^{\Gamma}(\eub{\Gamma};\bfL_R^{\langle j \rangle}) \to
H_{n}(\bub{\Gamma};\bfL^{\langle j \rangle}(R))$ comes from the induction structure
applied to the group homomorphism $\Gamma \to \{1\}$. Note that this map has a section, if
one inverts the orders of all finite subgroups of $\Gamma$. Essentially the sequence above
reduces the computation of $L_n^{\langle j \rangle}(R\Gamma)$ to that of
$H_{n}(\bub{\Gamma};\bfL^{\langle j \rangle}(R))$, which can be done in special cases,
when one understands the structure of $\bub{\Gamma}$.  (See~\cite[Theorem~10.1]{Davis-Lueck(2013)}.)

  The construction of this
  sequence is analogous to the one appearing in~\cite[Theorem~5.1~(b)]{Davis-Lueck(2003)}
  and left to the reader.
\end{remark}

%%%%%%%%%%%%%%%%%%%%%%%%%%%%%%%%%%%%%%%%%%%%%%%%%%%%%%%%%%%%%%%%%%%%%%%%%%%%%%%%%%%
%%%%%%%%%%%%%%%%%%%%%%%%%%%%%%%%%%%%%%%%%%%%%%%%%%%%%%%%%%%%%%%%%%%%%%%%%%%%%%%%%%%
%%%%%%%%%%%%%%%%%%%%%%%%%%%%%%%%%%%%%%%%%%%%%%%%%%%%%%%%%%%%%%%%%%%%%%%%%%%%%%%%%%%

\typeout{-------------------------- Section:The (periodic) structure group of  a pair   --------------------------}

%%%%%%%%%%%%%%%%%%%%%%%%%%%%%%%%%%%%%%%%%%%%%%%%%%%%%%%%%%%%%%%%%%%%

\section{The (periodic) structure group of  a pair}%
\label{sec:The_(periodic)_structure_set_of_a_pair}.

Let $(A, \partial A)$ be a  $CW$-pair.
Suppose for simplicity that $A$ is connected. We do not assume that $\partial A$ is connected.
Let $\Gamma = \pi_1(A) $ be its fundamental group and $\wt A$ its universal cover.
Let $w\colon  \pi_1A \to \{\pm1\}$ be the orientation character.    Let $\Pi(A)$ and $\Pi(\partial A)$ be the fundamental groupoids.

There is  a long exact sequence of abelian groups, called the \emph{periodic algebraic surgery exact sequence}
\begin{multline}
  \cdots \to H_{n}^{\Gamma}(\wt A,\partial \wt A;\bfL^s_{\Z,w})
  \to
  L^s_{n}(\IZ\Pi(\partial A) \to \IZ\Pi(A),w) \to 
  \cals^{\per,s}_n(A,\partial A)
  \\
  \to
H_{n-1}^{\Gamma}(\wt A,\partial \wt A;\bfL^s_{\Z,w})
  \to
  L^s_{n-1}(\IZ\Pi(\partial A) \to \IZ\Pi(A),w) \to \cdots.
  \label{long_exact_sequence_for_cals_upper_(per,s)(A,partial_A))_periodic}  
\end{multline}

If we replace $\bfL^s_{\Z,w}$ by its $1$-connective cover,
see Subsection~\ref{subsec:Equivariant_homology_theories_from_spectra_over_groupoids}, we obtain a
long exact sequence of abelian groups, called the \emph{algebraic surgery exact sequence}
\begin{multline}
  \cdots \to H_{n}^{\Gamma}(\wt A,\partial \wt A;\bfL^s_{\Z,w}\langle 1 \rangle)
  \to
  L^s_{n}(\IZ\Pi(\partial A) \to \IZ\Pi(A),w) \to 
  \cals^{\per,s}_n(A,\partial A)
  \\
  \to
H_{n-1}^{\Gamma}(\wt A,\partial \wt A;\bfL^s_{\Z,w}\langle 1 \rangle)
  \to
  L^s_{n-1}(\IZ\Pi(\partial A) \to \IZ\Pi(A),w) \to \cdots.
  \label{long_exact_sequence_for_cals_upper_(per,s)(A,partial_A))}   
\end{multline}

These sequences  can be constructed by taking homotopy groups of certain fibrations of spectra, see
for instance~\cite[Definition~14.6 on page~148]{Ranicki(1992)}.

Suppose that $(A, \partial A)$ is  a $(n-1)$-dimensional finite Poincar\'e pair.  Then the  algebraic surgery
sequence~\ref{long_exact_sequence_for_cals_upper_(per,s)(A,partial_A))} can be identified
with the geometric surgery exact sequence due to Sullivan and Wall, see for
instance~\cite[Theorem~18.5 on page~198]{Ranicki(1992)}.
So the algebraic structure groups
$\cals^{s}_n(A,\partial A)$ are relevant for the application of surgery theory for
topological manifolds, whereas the periodic version $\cals^{\per,s}_n(A,\partial A)$ is for us a
very good approximation of $\cals^{s}_n(A,\partial A)$.

The definition of the algebraic periodic structure group
$\cals^{\per,\langle j \rangle }_n(A,\partial A)$ and its non-periodic companion
$\cals^{\langle j \rangle }_n(A,\partial A)$ make sense for any
$j \in \{2,1,0, -1, \ldots\} \amalg \{-\infty\}$.  This includes the exact
sequences~\eqref{long_exact_sequence_for_cals_upper_(per,s)(A,partial_A))_periodic}
and~\eqref{long_exact_sequence_for_cals_upper_(per,s)(A,partial_A))}.

The next result will be crucial for the proof of our main theorems.  A $\Gamma$-homotopy
equivalence of free cocompact $\Gamma$-$CW$-complexes $(F,f) \colon (X,A) \to (Y,B)$ is
called simple, if the Whitehead torsions $\tau(f)$, $\tau(\partial f)$ and
$\tau(f,\partial f)$  vanish in $\Wh(\Gamma)$.  Additivity for the Whitehead torsion implies
that $(F,f)$ is simple, if two of the three elements $\tau(f)$, $\tau(\partial f)$ and $\tau(f,\partial f)$ in $\Wh(\Gamma)$
vanish.

\begin{theorem}\label{the:surgery_classification} Let $\pi$ be a finite index subgroup of a group $\Gamma$.
  Let $(Z, \partial Z)$ be a free cocompact $\Gamma$-$CW$-pair such that
  $Z$ is simply connected and $(Z/\Gamma,\partial Z/\Gamma)$ is a
  simple finite Poincar\'e pair of dimension $d \ge 6$.

\begin{enumerate}
\item\label{the:surgery_classification:existence} 
  Then $(Z/\pi, \partial Z/\pi)$ is a finite simple Poincar\'e pair of dimension $d$. 
  Suppose that the transfer map
  \[
  p^* \colon \cals^{s}_d(Z/\Gamma,\partial Z/\Gamma) \to \cals^{s}_d(Z/\pi,\partial Z/\pi)
  \]
  is injective. Then the following assertions are equivalent:

  \begin{enumerate}
   \item\label{the:surgery_classification:existence:Y} 
     There is a free cocompact $\pi$-manifold $N$ with boundary $\partial N$ and a simple $\pi$-homotopy
     equivalence of pairs $(N,\partial N) \to (Z,\partial Z)$;

    \item\label{the:surgery_classification:existence:Z} 
     There is a free cocompact $\Gamma$-manifold $M$ with boundary $\partial M$ and a simple $\Gamma$-homotopy
     equivalence of pairs $(M,\partial M) \to (Z,\partial Z)$;
   \end{enumerate}  

 \item\label{the:surgery_classification:uniqueness}
   Let $(f_i,\partial f_i) \colon (M_i,\partial M_i) \to (Z,\partial Z)$
   be simple $\Gamma$-homotopy equivalences with free compact $\Gamma$-manifolds $M_i$ with boundary  for $i = 0,1$.
   Suppose that    $\cals^{s}_{d+1}(Z/\Gamma,\partial Z/\Gamma)$ vanishes.

   Then there is a $\Gamma$-homeomorphism
   $(g,\partial g) \colon (M_0,\partial M_0) \xrightarrow{\cong} (M_1,\partial M_1)$ such
   that $(f_1,\partial f_1) \circ (g,\partial g)$ is $\Gamma$-homotopic as map of $\Gamma$-pairs to
   $(f_0,\partial f_0)$.

 \item\label{the:surgery_classification:decoration_h} There are also versions of
   assertions~\eqref{the:surgery_classification:existence}
   and~\eqref{the:surgery_classification:uniqueness} for the decoration $h$.  One has to
   drop simple everywhere, change the decoration from $s$ to $h$ everywhere, and weaken
   the conclusion in assertion~\eqref{the:surgery_classification:uniqueness} to the
   following statement:  There is a $\Gamma$-$h$-cobordism $(W,\partial W)$ with a $\Gamma$-homotopy
   equivalence of pairs
   $(F,\partial F) \colon (W,\partial W) \to (Z \times [0,1], \partial(Z \times [0,1]))$
   from $(f_0,\partial f_0) \colon (M_0,\partial M_0) \to (Z,\partial Z)$ to
   $(f_1,\partial f_1) \colon (M_1,\partial M_1) \to (Z,\partial Z)$.
    \end{enumerate}
\end{theorem}
\begin{proof}~\eqref{the:surgery_classification:existence} The
  implication~\eqref{the:surgery_classification:existence:Z}
  $\implies $~\eqref{the:surgery_classification:existence:Y} is obvious.  The
  implication~\eqref{the:surgery_classification:existence:Y}
  $\implies $~\eqref{the:surgery_classification:existence:Z} follows from Ranicki's theory
  of the total surgery obstruction, which is explained for closed manifolds and Poincar\'e
  complexes in~\cite[Definition~17.1 and Proposition~17.2 on page~190]{Ranicki(1992)} and
  extends to pairs, see~\cite[page~207--208]{Ranicki(1992)}.  More information can be
  found in~\cite{Kuehl-Macko-Mole(2013)}.  The total surgery obstruction assigns to
  $(Z/\pi,\partial Z/\pi)$ an element in $\cals^{s}_d(Z\pi,\partial Z/\pi)$, which
  vanishes, if and only if assertion~\eqref{the:surgery_classification:existence:Y} is
  true.  The total surgery obstruction assigns to $(Z/\Gamma,\partial Z\Gamma)$ an element
  in $\cals^{s}_d(Z/\Gamma,\partial Z/\Gamma)$, which vanishes, if and only if
  assertion~\eqref{the:surgery_classification:existence:Z} is true.  The map
  $p^* \colon \cals^{s}_d(Z\Gamma,\partial Z/\Gamma) \to \cals^{s}_d(Z\pi,\partial Z/\pi)$
  sends the total surgery obstruction of $(Z/\Gamma,\partial Z/\Gamma)$ to that of
  $(Z/\pi,\partial Z/\pi)$.   Since $p^*$ is assumed to be injective, the result follows.
  \\[1mm]~\eqref{the:surgery_classification:uniqueness},\eqref{the:surgery_classification:decoration_h} This
  follows from surgery theory and the identification of the geometric and the algebraic
  structure set, see~\cite[Theorem~18.5 on page~198]{Ranicki(1992)}, which also makes
  sense for Poincar\'e pairs. See also~\cite{Kuehl-Macko-Mole(2013)}.
\end{proof}

For the reader's convenience we spell out what  a $\Gamma$-$h$-cobordism
$(W,\partial W)$ with a $\Gamma$-homotopy equivalence of $\Gamma$-pairs
$(F,\partial F) \colon (W,\partial W) \to (Z \times [0,1], \partial(Z \times [0,1]))$ is.
Namely, we have a decomposition
$\partial W = \partial_0 W \cup \partial_1W \cup \partial_2 W$ into $\Gamma$-submanifolds of
codimension zero such that $\partial \partial_2 W = \partial \partial_0 W \cup \partial \partial_1 W$ and
$\partial_0 W \cap  \partial_1 W = \emptyset$ hold,  $\partial F$ induces $\Gamma$-homotopy equivalences
of pairs
   \begin{eqnarray*}
     \partial_0 F \colon (\partial_0 W,\partial \partial_0W)
     & \xrightarrow{\simeq} &
     (Z,\partial Z) \times \{0\};
    \\
     \partial_1 F \colon (\partial_1 W,\partial \partial_1W)
     & \xrightarrow{\simeq} &
     (Z,\partial Z) \times \{1\};
     \\
     \partial_2 F \colon (\partial_2 W,\partial \partial_2W)
     & \xrightarrow{\simeq}  &
     (\partial Z \times [0,1], \partial Z \times \{0,1\}),
   \end{eqnarray*}
   and there are identifications of $(M_i,\partial M_i)$ with
   $(\partial_i W, \partial \partial_iW)$ compatible with the $\Gamma$-maps to $Z$ for $i = 0,1$. In particular
   $\partial_2 F \colon \partial_2 W \to \partial Z \times [0,1]$ yields an $\Gamma$-$h$-cobordism from
   $\partial f_0 \colon \partial_0M \to \partial Z$ to
   $\partial f_1 \colon \partial_1M \to \partial Z$.

%%%%%%%%%%%%%%%%%%%%%%%%%%%%%%%%%%%%%%%%%%%%%%%%%%%%%%%%%%%%%%%%%%%%%%%%%%%%%%%%%%%
%%%%%%%%%%%%%%%%%%%%%%%%%%%%%%%%%%%%%%%%%%%%%%%%%%%%%%%%%%%%%%%%%%%%%%%%%%%%%%%%%%%
%%%%%%%%%%%%%%%%%%%%%%%%%%%%%%%%%%%%%%%%%%%%%%%%%%%%%%%%%%%%%%%%%%%%%%%%%%%%%%%%%%%

\typeout{-------------------- Section: Vanishing of the periodic structure group   --------------------------}

\section{Computing the periodic structure group}%
\label{sec:Computing_the_periodic_structure_group}

Let $\Gamma$ be the group appearing in the extension~\eqref{group_extension_intro}.

The next theorem follows directly
from~\cite[Theorem~1.12, Theorem~7.12, and Theorem~10.2]{Lueck(2022_Poincare_models)}.
It solves the existence question on the level
of Poincar\'e pairs and thus opens the door to apply surgery theory to replace up to (simple) homotopy  a
Poincar\'e pair by  a manifold with boundary.

Recall that $\calm$ is  a complete system of representatives of the conjugacy classes of maximal
  finite subgroups and $\calvII$ is  a complete system of representatives of the
  conjugacy classes of maximal virtually cyclic subgroups of type {II}.

\begin{theorem}[Poincar\'e models]\label{the:Poincare_models}
  Suppose that the following conditions are satisfied:

      \begin{itemize}

      \item The natural number $d$ satisfies $d \ge 3$;
      
          \item There is a finite $d$-dimensional Poincar\'e complex, which is homotopy
      equivalent to $B\pi$. Fix a generator  $[B\pi]$ of the
       infinite cyclic group $H_d^{\pi}(E\pi;\IZ^{w|_\pi})$;

      \item  For every $F \in \calm$
      the restriction of the homomorphism $w$ of~\eqref{w_colon_Gamma_to_(pm_1)}
      to $F$ is trivial, if $d$ is even, and is non-trivial, if $d$ is odd;

     \item $\Gamma$ satisfies conditions (M), (NM), and (H), see
     Definitions~\ref{def:conditions_on_Gamma_intro} and~\ref{definition:(H)};

    \item There exists a finite $\Gamma$-$CW$-model for $\eub{\Gamma}$;

     \item There exists an oriented  free $d$-dimensional
    slice system $\cals$ in the sense of Definition~\ref{def:free_d-dimensional_slice_system},
    which satisfies condition (S).
    Fix such a choice.
      
    \end{itemize}

    Put $\partial X = \coprod_{F \in \calm} \Gamma \times_F S_F$ and
    $C(\partial X) = \coprod_{F \in \calm} \Gamma \times_F D_F$ for $D_F$ the cone over
    $S_F$.

    Then there exists a finite free $\Gamma$-$CW$-pair $(X,\partial X)$ such that
    $X \cup_{\partial X} C(\partial X)$ is a model for $\eub{\Gamma}$ and
    $(X/\Gamma,\partial X/\Gamma)$ is a finite $d$-dimensional Poincar\'e pair.
    
  \end{theorem}
  
  Recall that the first Stiefel-Whitney class of the finite $d$-dimensional Poincar\'e pair
  $(X/\Gamma,\partial X/\Gamma)$ is automatically  the homomorphism $w \colon \Gamma \to \{\pm 1\}$
    of~\eqref{w_colon_Gamma_to_(pm_1)}.

\begin{theorem}\label{the:computing_the_periodic_structure_group}
  Suppose:

    \begin{itemize}
    
    \item The natural number $d$ satisfies $d \ge 3$;
    
          \item There is a finite $d$-dimensional Poincar\'e complex, which is homotopy
      equivalent to $B\pi$. Fix a generator  $[B\pi]$ of the
       infinite cyclic group $H_d^{\pi}(E\pi;\IZ^{w|_\pi})$;

      \item  For every $F \in \calm$
      the restriction of the homomorphism $w$ of~\eqref{w_colon_Gamma_to_(pm_1)}
      to $F$ is trivial, if $d$ is even, and is non-trivial, if $d$ is odd;

   \item $\Gamma$ satisfies (M), (NM), (H),  and  (V$_{\operatorname{II}}$);

    \item There exists a finite $\Gamma$-$CW$-model for $\eub{\Gamma}$;

  \item There exists an oriented  free $d$-dimensional
    slice system $\cals$ in the sense of Definition~\ref{def:free_d-dimensional_slice_system},
    which satisfies condition (S).
    Fix such a choice;
    
      \item The group $\pi$ is a  Farrell-Jones group.
      \end{itemize}

   Let  $(X,\partial X)$ be a finite free $\Gamma$-$CW$-pair such that
   $\partial X = \coprod_{F \in \calm} \Gamma \times_F S_F$ holds,
    $X \cup_{\partial X} C(\partial X)$ is a model for $\eub{\Gamma}$, and
    $(X/\Gamma,\partial X/\Gamma)$ is a finite $d$-dimensional Poincar\'e pair.
    (It exists by Theorem~\ref{the:Poincare_models}.)

    Consider any decoration
    $j \in \{2,1,0,-1, \ldots \}\amalg \{-\infty\}$. In the sequel
    everything has to be understood with respect to the orientation
    homomorphism $w \colon \Gamma \to \{\pm 1\}$
    of~\eqref{w_colon_Gamma_to_(pm_1)}.

    Then:

   \begin{enumerate}
   \item\label{the:computing_the_periodic_structure_group:general_iso}
   We get for any $n \in \IZ$ an isomorphism
   \[
     H_n^{\Gamma}(\eub{\Gamma} \to \edub{\Gamma};\bfL^{\langle j \rangle}_{\IZ,w}) \xrightarrow{\cong}
     H_n^{\Gamma}(\eub{\Gamma} \to\pt;\bfL^{\langle j \rangle}_{\IZ,w}) \xrightarrow{\cong}
     \cals_n^{\per,\langle j \rangle}(X/\Gamma,\partial X/\Gamma);
   \]
 \item\label{the:computing_the_periodic_structure_group:UNil}
    For every $n \in \IZ$, there is an isomorphism
    \[
    \bigoplus_{V \in \calvII} \UNil_n(\IZ;\IZ^{(-1)^d},\IZ^{(-1)^d})  \xrightarrow{\cong}
     \cals_n^{\per,\langle j \rangle}(X/\Gamma,\partial X/\Gamma);
   \]

  \item\label{the:computing_the_periodic_structure_group:vanishing}
    Fix $n \in \IZ$. Then $H_n^{\Gamma}(\eub{\Gamma} \to\pt;\bfL^{\langle j \rangle}_{\IZ,w}) $ vanishes, if and only if we have 
    $\UNil_n(\IZ;\IZ^{(-1)^d},\IZ^{(-1)^d})  = 0$ for every $V \in \calvII$ or $\calvII$ is empty.
    
  \end{enumerate}
\end{theorem}

  \begin{proof}~\eqref{the:computing_the_periodic_structure_group:general_iso}
    Consider the following sequence of squares $\Phi_1$, $\Phi_2$, and  $\Phi_3$, 
    \[
      \xymatrix{\partial X = \bigcup_{F \in \calm} \Gamma \times_F S_F \ar[r] \ar[d]
        &  X \ar[d]
        \\
        \coprod_{F \in \calm} \Gamma/F \ar[r]  \ar[d]_{\id}
        & \eub{\Gamma}  \ar[d]
        \\
        \coprod_{F \in \calm} \Gamma/F \ar[r] \ar[d]_{\id}
        & \edub{\Gamma}  \ar[d] \ar[d]_{\id}
        \\
        \coprod_{F \in \calm} \Gamma/F \ar[r] 
       &\pt.
     }
   \]
   Let $\Phi = \Phi_3 \circ \Phi_2 \circ \Phi_1$ be the square given by the composition of
   the three squares above.  This implies by inspecting the definition of the algebraic
   structure group of a pair
   \begin{equation}
     \cals_n^{\per,\langle j \rangle}(X/\Gamma,\partial X/\Gamma) = H_n^{\Gamma}(\Phi;\bfL^{\langle j \rangle}_{\IZ,w}).
     \label{cals_n_upper_(per,langle_j_rangle)(Z,partial_Z)}
   \end{equation}
   Since $\Phi_1$ is a $\Gamma$-pushout, we conclude
   $H_n^{\Gamma}(\Phi_1;\bfL^{\langle j \rangle}_{\IZ,w}) = \{0\}$ for all $n \in \IZ$ from
   excision. The ring $\IZ $ is a principal ideal domain and $\Wh_j(H;\IZ) = 0$ holds for
   every finite group $H$ and $j \le -2$, see~\cite{Carter(1980)}. Hence
   Theorem~\ref{the:computing_L-groups}~\eqref{the:computing_L-groups:edub_pt} applies and
   we get $H_n^{\Gamma}(\Phi_3;\bfL^{\langle j \rangle}_{\IZ,w}) = \{0\}$ for all
   $n \in \IZ$.  Therefore we get canonical isomorphisms
   \begin{equation}
     H_n^{\Gamma}(\Phi_2;\bfL^{\langle j \rangle}_{\IZ,w})
     \xrightarrow{\cong} H_n^{\Gamma}(\Phi_3 \circ \Phi_2;\bfL^{\langle j \rangle}_{\IZ,w})
     \xleftarrow{\cong} H_n^{\Gamma}(\Phi;\bfL^{\langle j \rangle}_{\IZ,w}) 
   \label{H_n(phi)_is_H_n(Phi_2)}
 \end{equation}
    for all $n \in \IZ$. There are also  canoncial isomorphisms
   \begin{eqnarray}
     H_n^{\Gamma}(\eub{\Gamma} \to \edub{\Gamma};\bfL^{\langle j \rangle}_{\IZ,w})
     \xrightarrow{\cong}
     H_n^{\Gamma}(\Phi_2;\bfL^{\langle j \rangle}_{\IZ,w});
     \label{phi_2_and_eub_to_edub}
     \\
     H_n^{\Gamma}(\eub{\Gamma} \to \pt;\bfL^{\langle j \rangle}_{\IZ,w})
     \xrightarrow{\cong}
     H_n^{\Gamma}(\Phi_3 \circ \Phi_2;\bfL^{\langle j \rangle}_{\IZ,w}).
     \label{phi_3_circ_phi_2_and_eub_to_pt}
   \end{eqnarray}
   Now put the  isomorphisms~\eqref{cals_n_upper_(per,langle_j_rangle)(Z,partial_Z)},~%
\eqref{H_n(phi)_is_H_n(Phi_2)},~\eqref{phi_2_and_eub_to_edub}, and~\eqref{phi_3_circ_phi_2_and_eub_to_pt}  together.
     \\[1mm]~\eqref{the:computing_the_periodic_structure_group:UNil}
     This follows from assertion~\eqref{the:computing_the_periodic_structure_group:general_iso}, 
    Theorem~\ref{the:computing_L-groups}~\eqref{the:computing_L-groups:eub(Gamma)_relative_terms_as_bigoplus},
    and Remark~\ref{rem:Identification_with_UNil-groups}, since $V \cong D_{\infty}$ holds for every $V \in  \calvII$
    by Lemma~\ref{lem:consequences_of_the_conditions_(M),(NM)_(F)} and $F \cong \IZ/2$ holds for every $F \in \calm$, if $d$ is odd,
    see~\cite[Lemma~3.3~(2)]{Lueck(2022_Poincare_models)}.
     \\[1mm]~\eqref{the:computing_the_periodic_structure_group:vanishing}
     This follows from assertion~\eqref{the:computing_the_periodic_structure_group:UNil}. 
   \end{proof}

 %%%%%%%%%%%%%%%%%%%%%%%%%%%%%%%%%%%%%%%%%%%%%%%%%%%%%%%%%%%%%%%%%%%%%%%%%%%%%%%%%%%
%%%%%%%%%%%%%%%%%%%%%%%%%%%%%%%%%%%%%%%%%%%%%%%%%%%%%%%%%%%%%%%%%%%%%%%%%%%%%%%%%%%
%%%%%%%%%%%%%%%%%%%%%%%%%%%%%%%%%%%%%%%%%%%%%%%%%%%%%%%%%%%%%%%%%%%%%%%%%%%%%%%%%%%

\typeout{-------------------- Section: Existence of manifold models --------------------------------}

\section{Existence of manifold models}%
\label{sec:Existence_of_manifold_models}

Let $\Gamma$ be the group appearing in the extension~\eqref{group_extension_intro}.
Let $\calm$ be a complete system of representatives of the conjugacy
classes of maximal  finite subgroups. 

\begin{theorem}[Existence of manifold models]%
\label{the:existence_of_manifold_model} Suppose:

  \begin{itemize}
    \item\label{the:existence_of_manifold_model:d_ge_6}
      The natural number $d$ satisfies $d \ge 5$;
      
            \item\label{the:existence_of_manifold_model:Bpi}
    There exists a closed manifold of dimension $d$, which is homotopy
    equivalent to $B\pi$. Fix a generator  $[B\pi]$ of the
       infinite cyclic group $H_d^{\pi}(E\pi;\IZ^{w|_\pi})$;

      \item\label{the:existence_of_manifold_model:orientaion}  For every $F \in \calm$
      the restriction of the homomorphism $w$ of~\eqref{w_colon_Gamma_to_(pm_1)}
      to $F$ is trivial, if $d$ is even, and is non-trivial, if $d$ is odd;

    \item\label{the:existence_of_manifold_model:(M),(NM),(H)}
      The group $\Gamma$ satisfies conditions (M), (NM), and (H), see
         Definitions~\ref{def:conditions_on_Gamma_intro} and~\ref{definition:(H)};

  \item\label{the:existence_of_manifold_model:eub(Gamma)}
   There exists a finite $\Gamma$-$CW$-model for $\eub{\Gamma}$;

  \item\label{the:existence_of_manifold_model:slice-system}
    There exists an oriented  free $d$-dimensional
    slice system $\cals$ in the sense of Definition~\ref{def:free_d-dimensional_slice_system},
    which satisfies condition (S).
    Fix such a choice;

    \item\label{the:existence_of_manifold_model:periodic_structure_group}
    The group
   $H_{m}^{\Gamma}(\eub{\Gamma} \to \pt;\bfL^{\langle j \rangle}_{\IZ,w})$ vanishes for every $m\in \{d, d+1\}$.
  
 \end{itemize}

  Let  $(X,\partial X)$  be a finite free $\Gamma$-$CW$-pair such that
   $\partial X = \coprod_{F \in \calm} \Gamma \times_F S_F$ holds,
    $X \cup_{\partial X} C(\partial X)$ is a model for $\eub{\Gamma}$, and
    $(X/\Gamma,\partial X/\Gamma)$ is a finite $d$-dimensional Poincar\'e pair.
    (It exists by Theorem~\ref{the:Poincare_models}.)   Then:
 
   \begin{enumerate}
   \item\label{the:existence_of_manifold_model:structure_groups_d}
  The structure groups $\cals_d^{\langle j \rangle}(X/\Gamma,\partial X/\Gamma)$ and
   $\cals_d^{\langle j \rangle}(X/\pi,\partial X/\pi)$ are infinite cyclic and the map induced by restriction with
   $i \colon \pi \to \Gamma$ induces an injection
   \[
     i^* \colon \cals_d^{\langle j \rangle}(X/\Gamma,\partial X/\Gamma)
     \to \cals_d^{\langle j \rangle}(X/\pi,\partial X/\pi);
   \]
   
 \item\label{the:existence_of_manifold_model:structure_groups_manifold_structure_relative}
   There exists a cocompact free $d$-dimensional $\Gamma$-manifold $N$ with boundary
   $\partial N$ together with a $\Gamma$-homotopy equivalence
   $(f,\partial f) \colon (N,\partial N) \xrightarrow{\simeq} (X,\partial X)$ of
   $\Gamma$-pairs;

 \item\label{the:existence_of_manifold_model:structure_groups_manifold_structure_absolute}
   Let $M$ be $N \cup_{\partial N} C(\partial N)$. Then $M$ is a slice manifold model for
   $\eub{\Gamma}$ with $N$ as a slice complement in the sense of
   Definition~\ref{def:slice-manifold_model}. If $\cals' = \{S'_F \mid F \in \calm\}$ is the
   underlying slice manifold system, then $S_F$ and $S_F'$ are $F$-homotopy equivalent for
   every $F \in \calm$.
  \end{enumerate}
  \end{theorem}
  \begin{proof}~\eqref{the:existence_of_manifold_model:structure_groups_d} In the sequel
    $w \colon \Gamma \to \{\pm 1\}$ is the group
    homomorphism~\eqref{w_colon_Gamma_to_(pm_1)}.  Recall that it agrees with the
    orientation homomorphism $\Gamma \to \{\pm 1\}$ associated to the Poincar\'e pair
    $(X/\Gamma,\partial X/\Gamma)$. Note that we can choose a finite $d$-dimensional
    $\Gamma$-$CW$-complex model for $\eub{\Gamma}$ such that
    $\eub{\Gamma}^{>1} = \coprod_{F \in \calm} \Gamma/F$ holds,
    see~\cite[Theorem~1.12]{Lueck(2022_Poincare_models)}.

    Let $\bfE \colon \Groupoids\to \Spectra$ be a  $\Groupoids$-spectrum.  We
    have introduced in Subsection~\ref{subsec:Equivariant_homology_theories_from_spectra_over_groupoids} the
    $\Groupoids$-spectrum $\bfE\langle 1 \rangle$ obtained from $\bfE$ by passing to the
    $1$-connective covering.  Consider a $\Gamma$-map $f \colon X \to Y$. The equivariant
    Atiyah-Hirzebruch spectral sequence for the $\Gamma$-homology theory associated to
    $\overline{\bfE}$ defined in Subsection~\ref{subsec:Equivariant_homology_theories_from_spectra_over_groupoids}
    has as $E^2$-term $E^2_{p,q}$ the Bredon homology
    $B\!H^{\Gamma}_p(f;\pi_q(\overline{\bfE}))$ and converges to
    $H_{p+q}^{\Gamma}(f;\overline{\bfE})$.  An easy spectral sequence argument yields
    an isomorphism, natural in $f$,
    \begin{equation}
      H^{\Gamma}_{d+1}(f;\overline{\bfE}) \xrightarrow{\cong} B\!H^{\Gamma}_{d+1}(f;\pi_0(\bfE)),
      \label{passage_to_Bredon_homology_general}
    \end{equation}
    provided that $\dim(X) \le d $ and $\dim(Y) \le d + 1$ holds.  If we apply
    equation~\eqref{passage_to_Bredon_homology_general} to $\eub{\Gamma} \to \pt$ for
    $\bfE = \bfL^{\langle j \rangle}_{\IZ,w}$ and write
    $\overline{\bfE} = \overline{\bfL}^{\langle j \rangle}_{\IZ,w}$, we obtain from the long
    exact sequence of the map $\eub{\Gamma} \to \pt$ for Bredon homology an isomorphism
    \begin{equation}
      \partial_{d+1} \colon H^{\Gamma}_{d+1}(\eub{\Gamma} \to \pt;\overline{\bfL}^{\langle j \rangle}_{\IZ,w})
      \xrightarrow{\cong} 
      B\!H^{\Gamma}_{d}(\eub{\Gamma};\pi_0(\bfL^{\langle j \rangle}_{\IZ,w})).
      \label{passage_to_Bredon_homology}
    \end{equation}
    We put
    \[
      \Phi := \raisebox{8mm}{\xymatrix{\partial X\ar[r] \ar[d] & X\ar[d]
          \\
          \pi_0(\partial X) \ar[r] & \eub{\Gamma}}} \quad \quad \quad \text{and} \quad
      \quad \quad \Psi := \raisebox{8mm}{\xymatrix{\partial X\ar[r] \ar[d] & X\ar[d]
          \\
          \pi_0(\partial X) \ar[r] & \pt.}}
    \]
    We have the following long exact sequence
    \begin{multline*}
      \cdots \to H^{\Gamma}_{n+1}(\Phi;\bfE) \to H^{\Gamma}_{n+1}(\Psi;\bfE) \to
      H^{\Gamma}_{n+1}(\eub{\Gamma} \to \pt;\bfE)
      \\
      \to H^{\Gamma}_{n}(\Phi;\bfE) \to H^{\Gamma}_{n}(\Psi;\bfE) \to \cdots.
    \end{multline*}
    Since $\Phi$ is a $\Gamma$-pushout and its upper horizontal arrow is a
    $\Gamma$-cofibration,
    $H^{\Gamma}_{n}(\Phi;\bfL^{\langle j \rangle}_{\IZ,w}\langle 1 \rangle)$ vanishes for
    every $n \in \IZ$. Hence we get for $n \in \IZ$  an isomorphism
    \begin{eqnarray}
      H^{\Gamma}_{n}(\Psi;\bfL^{\langle j \rangle}_{\IZ,w}\langle 1 \rangle)
      \xrightarrow{\cong} H^{\Gamma}_{n}(\eub{\Gamma} \to \pt;\bfL^{\langle j \rangle}_{\IZ,w}\langle 1 \rangle).
      \label{H(Psi)_cong_H_n_upper_Gamma(underline(E)G_to_pt)}
    \end{eqnarray}

    If we apply the long exact
    sequence~\eqref{long_exact_homology_sequence_for_E_to_E_langle_1_rangle_to_overline(E)}
    to $\bfE = \bfL^{\langle j \rangle}_{\IZ,w}$ and $\underline{E} \Gamma \to \pt$ and use
    the assumption that
    $H_{m}^{\Gamma}(\eub{\Gamma} \to \pt;\bfL^{\langle j \rangle}_{\IZ,w})$ vanishes for
    $m = d,d+1$, we get an isomorphism
    \begin{equation}
      \partial_{d+1} \colon H_{d+1}^{\Gamma}(\eub{\Gamma} \to \pt;\overline{\bfL}^{\langle j \rangle}_{\IZ,w})
      \xrightarrow{\cong}
      H_{d}^{\Gamma}(\eub{\Gamma} \to \pt;\bfL^{\langle j \rangle}_{\IZ,w}\langle 1 \rangle).
      \label{underline(E)G_to_pt;overline(L)_to_L_langle_1_rangle}
    \end{equation}
    Combining~\eqref{passage_to_Bredon_homology},~\eqref{H(Psi)_cong_H_n_upper_Gamma(underline(E)G_to_pt)}
    and~\eqref{underline(E)G_to_pt;overline(L)_to_L_langle_1_rangle} yields an isomorphism
    \begin{equation}
      H^{\Gamma}_{d}(\Psi;\bfL^{\langle j \rangle}_{\IZ,w}\langle 1 \rangle)
      \xrightarrow{\cong}
      B\!H^{\Gamma}_{d}(\eub{\Gamma};\pi_0(\bfL^{\langle j \rangle}_{\IZ,w})).
      \label{from_H(Psi)_to_BH(underline(E)G}
    \end{equation}

    Since the image of the inclusion $i \colon \pi \to \Gamma$ has finite index, there are
    natural restrictions functors
    \begin{eqnarray*}
      i^* \colon H_d^{\Gamma}(\Psi,\bfL^{\langle - j \rangle}_{\IZ,w}\langle 1 \rangle)
      & \to &
              H_d^{\pi}(i^* \Psi,\bfL^{\langle - j \rangle}_{\IZ,w}\langle 1 \rangle);
      \\
      i^* \colon H_d^{\Gamma}(\Psi,\bfL^{\langle - j \rangle}_{\IZ,w})
      & \to &
              H_d^{\pi}(i^* \Psi,\bfL^{\langle - j \rangle}_{\IZ,w});
      \\
      i^* \colon H_d^{\Gamma}(\Psi,\overline{\bfL}^{\langle - j \rangle}_{\IZ,w})
      & \to &
              H_d^{\pi}(i^* \Psi,\overline{\bfL}^{\langle - j \rangle}_{\IZ,w}),
    \end{eqnarray*}
    and analogously for Bredon homology
    \[
      i^* \colon B\!H^{\Gamma}_{d}(\Psi;\pi_0(\bfL^{\langle j \rangle}_{\IZ,w})) \to
      B\!H^{\pi}_{d}(i^*\Psi;\pi_0(\bfL^{\langle j \rangle}_{\IZ,w})).
    \]
    These are compatible with the exact
    sequence~\eqref{long_exact_homology_sequence_for_E_to_E_langle_1_rangle_to_overline(E)}
    and the isomorphism~\eqref{passage_to_Bredon_homology_general} for
    $\bfE = \bfL^{\langle - j \rangle}_{\IZ,w}$. Moreover, they are compatible with the
    isomorphisms~%
\eqref{passage_to_Bredon_homology},~\eqref{H(Psi)_cong_H_n_upper_Gamma(underline(E)G_to_pt)}~%
\eqref{underline(E)G_to_pt;overline(L)_to_L_langle_1_rangle}
    and~\eqref{from_H(Psi)_to_BH(underline(E)G} as well.  In particular, we get a
    commutative square whose vertical arrows are isomorphisms
    \begin{equation}
      \xymatrix{H^{\Gamma}_{d}(\Psi;\bfL^{\langle j \rangle}_{\IZ,w}\langle 1 \rangle)
        \ar[d]_{\cong} \ar[r]^-{i^*}
        &
        H^{\pi}_{d}(i^*\Psi;\bfL^{\langle j \rangle}_{\IZ,w}\langle 1 \rangle)
        \ar[d]_{\cong}
        \\
        B\!H^{\Gamma}_{d}(\eub{\Gamma};\pi_0(\bfL^{\langle j \rangle}_{\IZ,w}))
        \ar[r]_-{i^*}
        &
        B\!H^{\pi}_{d}(i^*\eub{\Gamma};\pi_0(\bfL^{\langle j \rangle}_{\IZ,w})).
      }
      \label{from_H(Psi)_to_BH(underline(E)G_with_restrict}
    \end{equation}

    Since $\eub{\Gamma}$ is a finite $\Gamma$-$CW$- complex such that
    $\dim(\eub{\Gamma}) \le d$ and all it cells in dimension $d$ have trivial isotropy
    group, $B\!H^{\Gamma}_{d}(\eub{\Gamma};\pi_0(\bfL^{\langle j \rangle}_{\IZ,w}))$ is an
    abelian subgroup of a finite direct sum of copies of
    $L_0^{\langle j \rangle}(\IZ) \cong L_0(\IZ) \cong \IZ$.  We conclude that there is a
    natural number $r$ satisfying
    \begin{equation}
      B\!H^{\Gamma}_{d}(\eub{\Gamma};\pi_0(\bfL^{\langle j \rangle}_{\IZ,w}))
      \cong \IZ^r.
      \label{B!H_upper_(Gamma)_n(underline(E)Gamma;pi_0(bfL_upper_langle_j_rangle)_Z)_cong_Z_upper_r}
    \end{equation}

    Consider the following commutative diagram

\[
  \xymatrix@!C=18em{ B\!H^{\Gamma}_{d}(\eub{\Gamma};\pi_0(\bfL^{\langle j \rangle}_{\IZ,w}))
    \ar[d]^{\cong} \ar[r]_-{i^*} & B\!H^{\pi}_{d}(i^*\eub{\Gamma};\pi_0(\bfL^{\langle j
      \rangle}_{\IZ,w})) \ar[d]_{\cong}
    \\
    B\!H^{\Gamma}_{d}(\coprod_{F \in \calm} \Gamma/F \to \eub{\Gamma};\pi_0(\bfL^{\langle
      j \rangle}_{\IZ,w})) \ar[r]_-{i^*} & B\!H^{\pi}_{d}(\coprod_{F \in \calm} i^* \Gamma/F
    \to i^*\eub{\Gamma};\pi_0(\bfL^{\langle j \rangle}_{\IZ,w}))
    \\
    B\!H^{\Gamma}_{d}(\coprod_{F \in \calm} \Gamma\times_F EF \to
    E\Gamma;\pi_0(\bfL^{\langle j \rangle}_{\IZ,w})) \ar[r]_-{i^*} \ar[u]_{\cong} &
    B\!H^{\pi}_{d}(\coprod_{F \in \calm} i^* (\Gamma \times_F EF) \to
    i^*E\Gamma;\pi_0(\bfL^{\langle j \rangle}_{\IZ,w})).\ar[u]_{\cong}  }
\]
We have the long exact sequence
\begin{multline*}
  \cdots \to B\!H^{\Gamma}_{d}(\coprod_{F \in \calm} \Gamma\times_F EF;\pi_0(\bfL^{\langle
    j \rangle}_{\IZ,w})) \to B\!H^{\Gamma}_{d}(E\Gamma;\pi_0(\bfL^{\langle j
    \rangle}_{\IZ,w})) \to
  \\
  \to B\!H^{\Gamma}_{d}(\coprod_{F \in \calm} \Gamma\times_F EF \to
  E\Gamma;\pi_0(\bfL^{\langle j \rangle}_{\IZ,w})) \to
  \\
  \to B\!H^{\Gamma}_{d-1}(\coprod_{F \in \calm} \Gamma\times_F EF;\pi_0(\bfL^{\langle j
    \rangle}_{\IZ,w})) \to B\!H^{\Gamma}_{d-1}(E\Gamma;\pi_0(\bfL^{\langle j
    \rangle}_{\IZ,w})) \to \cdots.
\end{multline*}
Since $\Gamma$ acts freely on $\Gamma \times_F EF$ and $E\Gamma$ and
$L^{\langle j \rangle}_0(\IZ)$ is independent of the decoration $j$, we get
identifications
\begin{eqnarray*}
  B\!H^{\Gamma}_{k}(\coprod_{F \in \calm} \Gamma\times_F EF;\pi_0(\bfL^{\langle j \rangle}_{\IZ,w}))
  &\cong &
           \bigoplus_{F \in \calm} H_k^{F}(EF;L_0(\IZ)^{w|_F})
  \\
  B\!H^{\Gamma}_{k}(E\Gamma;\pi_0(\bfL^{\langle j \rangle}_{\IZ,w}))
  &\cong &
           H_{k}^{\Gamma}(E\Gamma;L_0(\IZ)^w).
\end{eqnarray*}
Since $F$ is a finite group, we get for $k \ge 1$
\[
  H_k^{F}(EF;L_0(\IZ)^{w|_F})_{(0)} = 0,
\]
where for any abelian group $A$ we denote by $A_{(0)} := \IQ \otimes_{\IZ,w} A$ its
rationalization.  Hence we obtain an isomorphism
\[
  B\!H^{\Gamma}_{d}(\eub{\Gamma};\pi_0(\bfL^{\langle j \rangle}_{\IZ,w}))_{(0)} \cong
  H_{d}^{\Gamma}(E\Gamma;L_0(\IZ)^w)_{(0)}.
\]
Since we have the exact sequence $1 \to \pi \xrightarrow{i} \Gamma \to G \to 1$, we
conclude from the Leray-Serre spectral sequence that restriction with $i$ yields an
isomorphism
\[
  i^*_{(0)} \colon H_{d}^{\Gamma}(E\Gamma;L_0(\IZ)^w)_{(0)} \xrightarrow{\cong}
  H_{d}^{\pi}(E\pi;L_0(\IZ)^{w|_{\pi}})^G_{(0)}.
\]
By Poincar\'e duality we obtain an isomorphism
\[
  H_{d}^{\pi}(E\pi;L_0(\IZ)^{w|_{\pi}}) \cong   H^0_{\pi}(E\pi;L_0(\IZ)) \cong H^0(B\pi;L_0(\IZ)) \cong L_0(\IZ),
  \]
Moreover, the $G$-action on $H_{d}^{\pi}(E\pi;L_0(\IZ)^w)_{(0)}$ is
trivial by a direct inspection, cf.~\cite[Proof of Lemma~6.15]{Lueck(2022_Poincare_models)}.
Hence we get an isomorphism
\begin{equation}
  B\!H^{\Gamma}_{d}(\eub{\Gamma};\pi_0(\bfL^{\langle j \rangle}_{\IZ,w}))_{(0)}
  \xrightarrow{\cong} L_0(\IZ)_{(0)} \cong \IQ.
  \label{BH_upper_(Gamma)_n(underline(E)Gamma;pi_0(bfL_upper_langle_j_rangle)_Z)_(0)}
\end{equation}
Combining~\eqref{B!H_upper_(Gamma)_n(underline(E)Gamma;pi_0(bfL_upper_langle_j_rangle)_Z)_cong_Z_upper_r}
and~\eqref{BH_upper_(Gamma)_n(underline(E)Gamma;pi_0(bfL_upper_langle_j_rangle)_Z)_(0)}
yields
\begin{equation}
  B\!H^{\Gamma}_{d}(\eub{\Gamma};\pi_0(\bfL^{\langle j \rangle}_{\IZ,w}))
  \cong \IZ.
  \label{B!H_upper_(Gamma)_n(underline(E)Gamma;pi_0(bfL_upper_langle_j_rangle)_Z)_cong_Z}
\end{equation}
The following diagram commutes
\[\xymatrix{H_{d}^{\Gamma}(E\Gamma;L_0(\IZ)^w)_{(0)}  \ar[r]^{i^*} \ar[d] & H_{d}^{\pi}(E\pi;L_0(\IZ)^{w|_{\pi}})
    \ar[d]^{\cong}
    \\
    B\!H^{\Gamma}_{d}(\eub{\Gamma};\pi_0(\bfL^{\langle j \rangle}_{\IZ,w})) \ar[r]^{i^*} &
    B\!H^{\pi}_{d}(i^*\eub{\Gamma};\pi_0(\bfL^{\langle j \rangle}_{\IZ,w})). }
\]
The upper horizontal arrow, the left vertical arrow, and the right vertical arrow are bijective after rationalizing.
Hence the lower horizontal arrow is bijective after rationalizing.  Since its source and
target consists of infinite cyclic groups, the map
\[i^* \colon B\!H^{\Gamma}_{d}(\eub{\Gamma};\pi_0(\bfL^{\langle j \rangle}_{\IZ,w})) \to
  B\!H^{\pi}_{d}(i^*\eub{\Gamma};\pi_0(\bfL^{\langle j \rangle}_{\IZ,w}))
\]
is injective.  We conclude from~\eqref{from_H(Psi)_to_BH(underline(E)G_with_restrict} that
the map
\[
  i^* \colon H^{\Gamma}_{d}(\Psi;\bfL^{\langle j \rangle}_{\IZ,w}\langle 1 \rangle) \to
  H^{\pi}_{d}(i^*\Psi;\bfL^{\langle j \rangle}_{\IZ,w}\langle 1 \rangle)
\]
is injective. This map can be identified with
$i^* \colon \cals_d^{\langle j \rangle}(X/\Gamma,\partial X/\Gamma) \to \cals_d^{\langle j
  \rangle}(X/\pi,\partial X/\pi)$,
c.f.~\eqref{cals_n_upper_(per,langle_j_rangle)(Z,partial_Z)}.  This finishes the proof of
assertion\eqref{the:existence_of_manifold_model:structure_groups_d}.
\\[1mm]~\eqref{the:existence_of_manifold_model:structure_groups_manifold_structure_relative}
This follows from assertion~\eqref{the:existence_of_manifold_model:structure_groups_d} and
Theorem~\ref{the:surgery_classification}. In dimension $d = 5$ we use that fact that on
the boundary, which is $4$-dimensional, all $\Gamma$-components are induced from the finite
group $F$, which is good in the sense of Freedman.
\\[1mm]~\eqref{the:existence_of_manifold_model:structure_groups_manifold_structure_absolute}
By
assertion~\eqref{the:existence_of_manifold_model:structure_groups_manifold_structure_relative}
we can choose a free cocompact $d$-dimensional $\Gamma$-manifold $N$ with boundary
$\partial N$ together with a $\Gamma$-homotopy equivalence
$(f,\partial f) \colon (N,\partial N) \xrightarrow{\simeq} (X,\partial X)$ of
$\Gamma$-pairs.  Because the Poincar\'e Conjecture is known to be true, there exists a
$d$-dimensional free slice manifold system $\cals' = \{S'_F \mid F \in \calm\}$ with
$\partial N = \coprod_{F \in \calm} \Gamma \times_F S'_F$.  Let $D_F'$ be the cone of
$S_F'$.  Define the $\Gamma$-spaces $M$ and $X \cup_{\partial X} C(\partial X)$ by the
$\Gamma$-pushouts
\[
  \xymatrix{\partial N = \coprod_{F \in \calm} \Gamma \times_F S'_F \ar[r] \ar[d] & N
    \ar[d]
    \\
    C(\partial N) = \coprod_{F \in \calm} \Gamma \times_F D'_F \ar[r] & M}
\]
and
\[
  \xymatrix{\partial X = \coprod_{F \in \calm} \Gamma \times_F S_F \ar[r] \ar[d] & X
    \ar[d]
    \\
    C(\partial X) = \coprod_{F \in \calm} \Gamma \times_F D_F \ar[r] & X \cup_{\partial X}
    C(\partial X).}
\]
Obviously the $\Gamma$-homotopy equivalence
$\partial f \colon \coprod_{F \in \calm} \Gamma \times_F S'_F = \partial N \to \Gamma
\times_F S_F = \partial X$ extends to a $\Gamma$-homotopy equivalence
$\widehat{\partial f} \colon \coprod_{F \in \calm} \Gamma \times_F D'_F \to \Gamma
\times_F D_F$.  Let $F \colon M \to X \cup_{\partial X} C(\partial X) $ be the
$\Gamma$-map given by the $\Gamma$-pushout of the three $\Gamma$-homotopy equivalences
$\widehat{\partial f}$, $\partial f$ and $f$. Then $F$ itself is a $\Gamma$-homotopy
equivalence.  The $\Gamma$-space $X \cup_{\partial X} C(\partial X) $ is a
$\Gamma$-$CW$-model for $\eub{\Gamma}$ by assumption. This finishes the proof of
Theorem~\ref{the:existence_of_manifold_model}.
\end{proof}

   %%%%%%%%%%%%%%%%%%%%%%%%%%%%%%%%%%%%%%%%%%%%%%%%%%%%%%%%%%%%%%%%%%%%%%%%%%%%%%%%%%%
%%%%%%%%%%%%%%%%%%%%%%%%%%%%%%%%%%%%%%%%%%%%%%%%%%%%%%%%%%%%%%%%%%%%%%%%%%%%%%%%%%%
%%%%%%%%%%%%%%%%%%%%%%%%%%%%%%%%%%%%%%%%%%%%%%%%%%%%%%%%%%%%%%%%%%%%%%%%%%%%%%%%%%%

\typeout{-------------------- Section: Uniqueness of manifold models --------------------------------}

\section{Uniqueness of manifold models}%
\label{sec:Uniqueness_of_manifold_models}

Let $\Gamma$ be the group appearing in the extension~\eqref{group_extension_intro}.
Let $\calm$ be a complete system of representatives of the conjugacy
classes of maximal  finite subgroups.

\begin{theorem}[Uniqueness of manifold models]%
\label{the:uniqueness_of_manifold_models} Suppose:

  \begin{itemize}
    
   \item The natural number $d$ satisfies $d \ge 5$;

   \item $\Gamma$  satisfies  condition (V$_{\operatorname{II}})$,
      see Definition~\ref{def:conditions_on_Gamma_intro};

    \item   The group $\pi $ is a Farrell-Jones group, see
      Subsection~\ref{subsec:The_Full_Farrell_Jones_Conjecture};

    \item The group $\UNil_{d+1}(\IZ;\IZ^{(-1)^d},\IZ^{(-1)^d})$ vanishes or $\Gamma$
      contains no subgroup isomorphic to $D_{\infty}$.
      \end{itemize}

  Let $M$ and $M'$ be two slice manifold models for $\eub{\Gamma}$ with respect to the
  $d$-dimensional free slice manifold systems $\cals = \{S_F \mid F\in \calm\}$ and
  $\cals' = \{S'_F \mid F\in \calm\}$ in the sense of
  Definition~\ref{def:slice-manifold_model}. Let $N$ and $N'$ be slice complements, i.e.,
  cocompact proper proper free $d$-dimensional $\Gamma$-manifolds with boundary such that there
  are $\Gamma$-pushouts
\[\xymatrix{\partial N = \coprod_{F \in \calm} \Gamma \times_F S_F \ar[r]\ar[d]
    &
    N \ar[d]
    \\
    C(\partial N)
    \ar[r]
    &
    M
  }
   \quad \quad \raisebox{-7mm}{\textup{and}} \quad \quad 
  \xymatrix{\partial N' = \coprod_{F \in \calm} \Gamma \times_F S'_F \ar[r]\ar[d]
    &
    N \ar[d]
    \\
    C(\partial N')
    \ar[r]
    &
    M'
  }
\]
where we abbreviate $C(\partial N) := \coprod_{F \in \calm} \Gamma \times_F D_F$
and $C(\partial N') := \coprod_{F \in \calm} \Gamma \times_F D'_F$.

    Then:
 
    \begin{enumerate}

    \item\label{the:Uniqueness_of_manifold_models:periodic_structure_group_(d_plus_1)_vanishes}
      The group
      $H_{d+1}^{\Gamma}(\eub{\Gamma} \to \pt;\bfL^{\langle j \rangle}_{\IZ,w})$ vanishes for
      every  $j \in \{2,1,0,-1, \ldots \} \amalg \{-\infty\}$;
      
    \item\label{the:Uniqueness_of_manifold_models:structure_group_(d_plus_1)_vanishes}
      The structure group $\cals_{d+1}^{\langle j \rangle}(N'/\Gamma,\partial N'/\Gamma)$
      vanishes for $j \in \{2,1,0,-1,\ldots\} \amalg \{-\infty\}$;

      \item\label{the:Uniqueness_of_manifold_models:homeo_and_h-cob}
     We have
     \begin{enumerate}

       \item\label{the:Uniqueness_of_manifold_models:homeo_and_h-cob:slices_Gamma-h-cobordant}
       For every $F \in F$ there exists an $F$-equivariant $h$-cobordism between $S_F$ and $S_F'$;

     \item\label{the:Uniqueness_of_manifold_models:homeo_and_h-cob:homeomorphism}
     There exists a $\Gamma$-homeomorphism $f \colon M \to M'$;

     \end{enumerate}
     
    \item\label{the:Uniqueness_of_manifold_models:uniqueness_of_slice_complements}
      \begin{enumerate}
        \item\label{the:Uniqueness_of_manifold_models:uniqueness_of_slice_complements:homeo_implies_simple}
        Every $\Gamma$-homeomorphism of cocompact proper free $\Gamma$-manifolds with boundary
    $(N,\partial N) \to (N',\partial N')$ is a simple $\Gamma$-homotopy equivalence  of $\Gamma$-$CW$-pairs;

  \item\label{the:Uniqueness_of_manifold_models:uniqueness_of_slice_complements:simple_implies_homeo}
    Every simple $\Gamma$-homotopy equivalence 
     $(N,\partial N) \to (N',\partial N')$  of $\Gamma$-$CW$-pairs
     is $\Gamma$-homotopic to a $\Gamma$-homeomorphism
     $(N,\partial N) \to (N',\partial N')$  of cocompact proper free $\Gamma$-manifolds with boundary;

   \end{enumerate}

   \item\label{the:Uniqueness_of_manifold_models:uniqueness_of_slice_complements_special}
     The following assertions are equivalent, if we additionally assume that
     for all $F \in \calm$ the $2$-Sylow subgroup of $F$ is cyclic;

       \begin{enumerate}
       \item\label{the:Uniqueness_of_manifold_models:uniqueness_of_slice_complements_special:homeo_pairs}
         There exists a $\Gamma$-homeomorphism of cocompact proper free $\Gamma$-manifolds with
         boundary
         $(h,\partial h) \colon (N,\partial N) \xrightarrow{\cong} (N',\partial N')$ such
         that $\partial h$ induces for each $F \in \calm$ a $F$-homeomorphism
         $\partial h_F \colon S_F \to S_F'$ satisfying
         $\partial h = \coprod_{F \in \calm} \id_{\Gamma} \times_F~\partial h_F$;
    
        \item\label{the:Uniqueness_of_manifold_models:uniqueness_of_slice_complements_special_simple_pairs}
         There exists a simple $\Gamma$-homotopy equivalence 
         $(f,\partial f) \colon (N,\partial N) \xrightarrow{\simeq_s} (N',\partial N')$  of $\Gamma$-$CW$-complexes
         such that such that $\partial f$  induces for each $F \in \calm$ a simple $F$-homotopy equivalence
         $\partial f_F \colon S_F \to S_F'$ satisfying
         $\partial f = \coprod_{F \in \calm} \id_{\Gamma} \times_F~\partial f_F$;

   \item\label{the:Uniqueness_of_manifold_models:uniqueness_of_slice_complements_special:homeo_S_F}
     For every $F \in F$ there exists a $F$-homeomorphism of cocompact proper free $F$-manifolds 
     $S_F \xrightarrow{\cong} S'_F$;

      \item\label{the:Uniqueness_of_manifold_models:uniqueness_of_slice_complements_special:simple_S_F}
     For every $F \in F$ there exists a simple $F$-homotopy equivalence of finite free $\Gamma$-$CW$-complexes
     $S_F \xrightarrow{\simeq_s}  S'_F$.

   \end{enumerate}

    \end{enumerate}
\end{theorem}
\begin{proof}~\eqref{the:Uniqueness_of_manifold_models:periodic_structure_group_(d_plus_1)_vanishes}
  The existence of a slice model $M'$ implies the following facts.  We conclude
  from~\cite[Lemma~1.9] {Lueck(2022_Poincare_models)}, which directly extends to the case,
  where $w$ is non-trivial, that $\Gamma$ satisfies conditions (M), (NM), and (H), see
Definitions~\ref{def:conditions_on_Gamma_intro} and~\ref{definition:(H)}, and that there is a finite
  $d$-dimensional $\Gamma$-$CW$-complex model for $\eub{\Gamma}$ such that
  $\eub{\Gamma}^{>1} = \coprod_{F \in \calm} \Gamma/F$ holds. Moreover, there is a closed
  manifold model for $B\pi$ and the orientation character $w \colon \Gamma \to \{\pm 1\}$
  has for every $F \in \calm$ the property that $w|_F$ is trivial, if $d$ is even, and is
  non-trivial, if $d$ is odd, see~\cite[Subsection~6.2]{Lueck(2022_Poincare_models)}.  Now
  apply Theorem~\ref{the:computing_the_periodic_structure_group}~%
  \eqref{the:computing_the_periodic_structure_group:vanishing}.
  \\[1mm]~\eqref{the:Uniqueness_of_manifold_models:structure_group_(d_plus_1)_vanishes} We
  proceed as in the proof of Theorem~\ref{the:existence_of_manifold_model}~%
  \eqref{the:existence_of_manifold_model:structure_groups_d} An easy spectral sequence
  argument shows that we get from
  assertion~\eqref{the:Uniqueness_of_manifold_models:periodic_structure_group_(d_plus_1)_vanishes}
  an isomorphism
\begin{equation}
 H^{\Gamma}_{d+2}(\eub{\Gamma} \to \pt;\overline{\bfL}^{\langle j \rangle}_{\IZ,w}) \ =  \{0\}.
 \label{vanishing_of_H_(d_plus_2)(eub(Gamma)_to_pt;overline(L)}
\end{equation}
We put
\[
  \Phi
 := 
   \raisebox{8mm}{\xymatrix{\partial N'\ar[r] \ar[d]
    & N'\ar[d]
    \\
    \pi_0(\partial N')  \ar[r]
    & \eub{\Gamma}}}
  \quad \quad \quad \text{and} \quad \quad  \quad 
  \Psi
   := 
   \raisebox{8mm}{\xymatrix{\partial N'\ar[r] \ar[d]
    & N'\ar[d]
    \\
    \pi_0(\partial N')  \ar[r]
    & \pt.}}
\]
We have the following long exact sequence
\begin{multline*}
  \cdots \to H^{\Gamma}_{n+1}(\Phi;\bfL^{\langle j \rangle}_{\IZ,w}\langle 1 \rangle)
  \to H^{\Gamma}_{n+1}(\Psi;\bfL^{\langle j \rangle}_{\IZ,w}\langle 1 \rangle)
  \to  H^{\Gamma}_{n+1}(\eub{\Gamma} \to \pt;\bfL^{\langle j \rangle}_{\IZ,w}\langle 1 \rangle)
  \\
  \to H^{\Gamma}_{n}(\Phi;\bfL^{\langle j \rangle}_{\IZ,w}\langle 1 \rangle)
  \to H^{\Gamma}_{n}(\Psi;\bfL^{\langle j \rangle}_{\IZ,w}\langle 1 \rangle)
  \to  \cdots.
\end{multline*}
Since $\Phi$ is a $\Gamma$-pushout and its upper horizontal arrow is a $\Gamma$-cofibration,
$H^{\Gamma}_{n}(\Phi;\bfL^{\langle j \rangle}_{\IZ,w}\langle 1 \rangle)$ vanishes for every
$n \in \IZ$. Hence we get for  $n \in \IZ$ an isomorphism
\begin{eqnarray}
H^{\Gamma}_{n}(\Psi;\bfL^{\langle j \rangle}_{\IZ,w}\langle 1 \rangle)
  \xrightarrow{\cong} H^{\Gamma}_{n}(\eub{\Gamma} \to \pt;\bfL^{\langle j \rangle}_{\IZ,w}\langle 1 \rangle).
 \label{H(Psi)_cong_H_n_upper_Gamma(underline(E)G_to_pt)_d_plus_1}
\end{eqnarray}

If we apply the long exact
sequence~\eqref{long_exact_homology_sequence_for_E_to_E_langle_1_rangle_to_overline(E)} to 
$\bfE = \bfL^{\langle j \rangle}_{\IZ,w}$ and $\underline{E} \Gamma \to \pt$ and use  that
$H_{d+1}^{\Gamma}(\eub{\Gamma} \to \pt;\bfL^{\langle j \rangle}_{\IZ,w})$ vanishes
by assertion~\eqref{the:Uniqueness_of_manifold_models:periodic_structure_group_(d_plus_1)_vanishes},
we get  an epimorphism
\begin{equation}
\partial_{d+2} \colon H_{d+2}^{\Gamma}(\eub{\Gamma} \to \pt;\overline{\bfL}^{\langle j \rangle}_{\IZ,w})
\xrightarrow{\cong}
H_{d+1}^{\Gamma}(\eub{\Gamma} \to \pt;\bfL^{\langle j \rangle}_{\IZ,w}\langle 1 \rangle).
\label{underline(E)G_to_pt;overline(L)_to_L_langle_1_rangle_plus_1}
\end{equation}

Combining~\eqref{vanishing_of_H_(d_plus_2)(eub(Gamma)_to_pt;overline(L)},~%
\eqref{H(Psi)_cong_H_n_upper_Gamma(underline(E)G_to_pt)_d_plus_1}
and~\eqref{underline(E)G_to_pt;overline(L)_to_L_langle_1_rangle_plus_1} yields
\begin{equation}
  H^{\Gamma}_{d+1}(\Psi;\bfL^{\langle j \rangle}_{\IZ,w}\langle 1 \rangle)
  = \{0\}.
  \label{from_H(Psi)_to_BH(underline(E)G_vanishing}
\end{equation}
Note that we get an identification
$H^{\Gamma}_{d+1}(\Psi;\bfL^{\langle j \rangle}_{\IZ,w}\langle 1 \rangle)
= \cals_{d+1}^{\langle j \rangle}(N'/\Gamma,\partial N'/\Gamma)$
from the definition of the algebraic structure groups.
Hence $\cals_{d+1}^{\langle j \rangle}(N'/\Gamma,\partial N'/\Gamma)$
vanishes because of~\eqref{from_H(Psi)_to_BH(underline(E)G_vanishing}.
\\[1mm]~\eqref{the:Uniqueness_of_manifold_models:homeo_and_h-cob}
The main arguments for the proof of assertion~\eqref{the:Uniqueness_of_manifold_models:homeo_and_h-cob}
have already  been presented  in~\cite[Proof of Lemma~4.3]{Connolly-Davis-Khan(2015)}. For the
reader's convenience we give more details here.

We conclude
from~\cite[Theorem~8.9 and Theorem~10.2]{Lueck(2022_Poincare_models)}
taking~\cite[Lemma~3.4]{Lueck(2022_Poincare_models)}  into account
that there is  a $\Gamma$-homotopy equivalence of $\Gamma$-pairs
$(f,\partial f) \colon (N,\partial N) \xrightarrow{\simeq_{\Gamma}} (N',\partial N')$
such that  $\partial f$ induces $F$-homotopy equivalences
$\partial f_F \colon S'_F \to S'_F$  for every $F \in \calm$.
We get $\cals_{d+1}^{\langle h\rangle}(N'/\Gamma,\partial N'/\Gamma) = 0$ from
assertion~\eqref{the:Uniqueness_of_manifold_models:structure_group_(d_plus_1)_vanishes}.
We conclude from Theorem~\ref{the:surgery_classification}~\eqref{the:surgery_classification:decoration_h}
that there is an $h$-cobordism $(W,\partial W)$ with a $\Gamma$-homotopy
equivalence of pairs
$(F,\partial F) \colon (W,\partial W) \to (N' \times [0,1], \partial(N' \times [0,1]))$
from $(f,\partial f) \colon (N,\partial N) \to (N',\partial N')$ to
$\id_{(N',\partial N')} \colon(N',\partial N')\to (N',\partial N')$.
In dimension $d = 5$ we use that fact that on the boundary,
which is $4$-dimensional, all $\Gamma$-components are induced from  finite groups $F$
which are all  good in the sense of Freedman. In particular we see
that~\eqref{the:Uniqueness_of_manifold_models:homeo_and_h-cob:slices_Gamma-h-cobordant} holds.
We conclude from Theorem~\ref{the:computing_K-groups}, the $s$-Cobordism Theorem for pairs,
and basic properties of Whitehead torsion as for instance homotopy invariance and the sum formula
that we can choose for $F\in  \calm$ an  $F$-$h$-cobordism $V_F$ between $S_F$ and $S'_F$
such that $W = \bigcup_{F \in \calm} \Gamma \times_F V_F$ is a $\Gamma$-h-cobordism
between $\partial N$ and $\partial N'$  and a simple $\Gamma$-$h$-cobordism of pairs between
$N \cup_{\partial N} W$ and $N'$ relative $\partial N'$. The latter implies
that there is a $\Gamma$-homeomorphism of pairs $(N \cup_{\partial N} W,\partial N') \xrightarrow{\cong} (N',\partial N')$
which is the identity on $\partial N'$. It induces a $\Gamma$-homeomorphism
$N \cup_{\partial N} W \cup _{\partial N' }C(\partial N') \xrightarrow{\cong} M' = N'\cup _{\partial N'} C(\partial N')$.
Hence it suffices to show that
$N \cup_{\partial N} W \cup_{\partial N'} C(\partial N')$ and
$M = N \cup_{\partial N} C(\partial N)$ are $\Gamma$-homeomorphic.

Let $W^-$ be a $\Gamma$-$h$-cobordism between $\partial N'$ and $\partial N$ satisfying
$\tau(\partial N' \to W^-) = - \tau(\partial N \to W)$. Then $W \cup_{\partial N'} W^-$ is
a trivial $\Gamma$-$h$-cobordism over $\partial N$ because of
\begin{multline*}
\tau(\partial N \to W \cup_{\partial N'} W^-) = \tau(\partial N \to W) + \tau(\partial N' \to W^-)
\\
= \tau(\partial N \to W) - \tau(\partial N \to W) =0.
\end{multline*}
Analogously one shows that $W^- \cup_{\partial N} W$ is a trivial $F$-h-cobordism over
$\partial N'$.  Now one constructs by an Eilenberg swindle a $\Gamma$-homeomorphism
relative to  $\partial N$
\[
W \cup_{\partial N'} \partial N' \times [0,\infty) \xrightarrow{\cong} \partial N\times  [0,\infty).
\]
 It extends by passing to the one-point-compatification of the various path components
  to a $\Gamma$-homeomorphism relative $\partial N$.
  \[
      W \cup_{\partial N'} C(\partial N') \xrightarrow{\cong} C(\partial N).
  \]
 It together with $\id_N$ yields the desired $\Gamma$-homeomorphism
 $N \cup_{\partial} W \cup_{\partial N'} C(\partial N') \xrightarrow{\cong} N
 \cup_{\partial N} C(\partial N)$. 
\\[1mm]~\eqref{the:Uniqueness_of_manifold_models:uniqueness_of_slice_complements}
Since $\Gamma$ acts freely on  $(N,\partial N)$ and $(N',\partial N')$, any $\Gamma$-homeomorphism
$(N,\partial N) \xrightarrow{\cong} (N',\partial N')$ of pairs is a simple $\Gamma$-homotopy equivalence
of $\Gamma$-$CW$-pairs by the topological invariance of Whitehead torsion,
see~\cite{Chapman(1973),Chapman(1974)}. 

Consider  a simple homotopy equivalence
of $\Gamma$-pairs $(f,\partial f) \colon (N,\partial N) \xrightarrow{\cong} (N',\partial N')$,
We get $\cals_{d+1}^{\langle s\rangle}(N'/\Gamma,\partial N'/\Gamma) = 0$ from
assertion\eqref{the:Uniqueness_of_manifold_models:periodic_structure_group_(d_plus_1)_vanishes}.
Now we conclude from Theorem~\ref{the:surgery_classification}~\eqref{the:surgery_classification:uniqueness}that $(f,\partial f)$
is $\Gamma$-homotopic to a $\Gamma$-homeomorphism of pairs.
\\[1mm]~\eqref{the:Uniqueness_of_manifold_models:uniqueness_of_slice_complements_special}
We conclude~\eqref{the:Uniqueness_of_manifold_models:uniqueness_of_slice_complements_special:homeo_pairs}~%
$\Longleftrightarrow$~\eqref{the:Uniqueness_of_manifold_models:uniqueness_of_slice_complements_special_simple_pairs}
from assertion~\eqref{the:Uniqueness_of_manifold_models:uniqueness_of_slice_complements}.
Obviously~\eqref{the:Uniqueness_of_manifold_models:uniqueness_of_slice_complements_special_simple_pairs}~%
$\implies$~\eqref{the:Uniqueness_of_manifold_models:uniqueness_of_slice_complements_special:simple_S_F}
and~\eqref{the:Uniqueness_of_manifold_models:uniqueness_of_slice_complements_special:homeo_pairs}~%
$~\implies$\eqref{the:Uniqueness_of_manifold_models:uniqueness_of_slice_complements_special:homeo_S_F} hold.
We get~\eqref{the:Uniqueness_of_manifold_models:uniqueness_of_slice_complements_special:homeo_S_F}~%
$\implies$~\eqref{the:Uniqueness_of_manifold_models:uniqueness_of_slice_complements_special:simple_S_F}
from the topological invariance of Whitehead torsion,
see~\cite{Chapman(1973),Chapman(1974)}. Hence it remains to prove the 
implication~\eqref{the:Uniqueness_of_manifold_models:uniqueness_of_slice_complements_special:simple_S_F}~$\implies$~%
\eqref{the:Uniqueness_of_manifold_models:uniqueness_of_slice_complements_special_simple_pairs} what we do  next.

Any $F$-homotopy equivalence $S_F \xrightarrow{\sim_s} S'_F$ is simple
by~\cite[Lemma~3.3~(5)]{Lueck(2022_Poincare_models)}, since there exists one
$F$-homotopy equivalence $S_F \xrightarrow{\sim_s} S'_F$ by assumption.
Now apply~\cite[Theorem~9.1 and Theorem~10.3]{Lueck(2022_Poincare_models)}
taking~\cite[Lemma~3.4]{Lueck(2022_Poincare_models)}  into account. The condition (S)
appearing in~\cite[Definition~7.9]{Lueck(2022_Poincare_models)}  can be arranged to hold
by~\cite[Lemma~7.10]{Lueck(2022_Poincare_models)}. 
This finishes the proof of Theorem~\ref{the:uniqueness_of_manifold_models}.
\end{proof}

\begin{remark}\label{rem:slice_system}
  Note that from Theorem~\ref{the:uniqueness_of_manifold_models} we get a slice model
  system $\cals = \{S_F \mid F \in \calm\}$ such that each element $S_F$ is unique up to
  $F$-$h$-cobordism.  It is unclear how we can determine $\cals$ just from $\Gamma$,
  provided that all the assumptions appearing in
  Theorem~\ref{the:uniqueness_of_manifold_models} are satisfied.  Note that we have at
  least a recipe to determine $\{\kappa_F\ \mid F \in \calm\}$ from $\Gamma$ and hence the
  $F$-homotopy type of each $S_F$, see Definition~\ref{definition:(S)}.

  If $F$ is finite cyclic of odd order, the simple structure set of $S_F/F$ has completely
  been determined in terms of Reidemeister torsion and $\rho$-invariants by
  Wall~\cite[Theorem~14E.7 on page~224]{Wall(1999)}.
  
  But there is a geometric case, where the passage from $\Gamma$ to the slice manifold
  system is explicit.  Let $X$ be a closed negatively curved manifold and suppose that
  $\Gamma$ is a discrete cocompact subgroup of $\Isom(\widetilde X)$, that $\Gamma$ is a
  finite extension of the deck transformations $\pi$, and that $\Gamma$ satisfies
  condition (F).  Then the action of $\Gamma$ extends to the sphere $S^{d-1}_\infty$ and
  any non-trivial finite subgroup $F$ fixes a point in $\widetilde X$ and acts freely on
  the boundary sphere.  The $F$-h-cobordism between the boundary of a $F$-tubular
  neighborhood of the fixed point and the sphere at infinity determines the slice manifold
  structure.
\end{remark}

%%%%%%%%%%%%%%%%%%%%%%%%%%%%%%%%%%%%%%%%%%%%%%%%%%%%%%%%%%%%%%%%%%%%%%%%%%%%%%%%%%%
%%%%%%%%%%%%%%%%%%%%%%%%%%%%%%%%%%%%%%%%%%%%%%%%%%%%%%%%%%%%%%%%%%%%%%%%%%%%%%%%%%%
%%%%%%%%%%%%%%%%%%%%%%%%%%%%%%%%%%%%%%%%%%%%%%%%%%%%%%%%%%%%%%%%%%%%%%%%%%%%%%%%%%%

%%%%%%%%%%%%%%%%%%%%%%%%%%%%%%%%%%%%%%%%%%%%%%%%%%%%%%%%%%%%%%%%%%%%%%%%%%%%%%%%%%%
%%%%%%%%%%%%%%%%%%%%%%%%%%%%%%%%%%%%%%%%%%%%%%%%%%%%%%%%%%%%%%%%%%%%%%%%%%%%%%%%%%%
%%%%%%%%%%%%%%%%%%%%%%%%%%%%%%%%%%%%%%%%%%%%%%%%%%%%%%%%%%%%%%%%%%%%%%%%%%%%%%%%%%%

  \typeout{-------------------------- Section: Comparisons --------------------------}

  \section{Comparisons}\label{sec:Comparisons}
  
\subsection{Cocompact and d-dimensional models}%
\label{subsec:Cocompact_and_d-dimensional_models}

We discuss the relationship between cocompact manifold models and $d$-dimensional manifold
models for $\eub{\Gamma}$.

Assume that $\Gamma$ is a finite extension of the fundamental group of a closed aspherical
manifold $X$ of dimension $d$.  Note that $H_d(X;\Z/2) = \Z/2$ and $H_e(X;\Z/2)=0$ for
$e>d$. We claim that a cocompact manifold model for $\eub{\Gamma}$ is a $d$-dimensional
manifold model for $\eub{\Gamma}$ and conversely.  If $M$ is a cocompact manifold model for
$\eub{\Gamma}$, then $M/\pi$ is a closed manifold which has the homotopy type of $X$;
hence its homology shows that it has dimension $d$.  Conversely, if $M$ is a
$d$-dimensional manifold model for $\eub{\Gamma}$, then $M/\pi$ is a $d$-manifold having
the homology of $X$ so must be closed.  Hence $M/\Gamma$ is compact.

\subsection{(M),(NM), and (F)}\label{subsec:(M),(NM),and_(F)}
  
Clearly conditions (M) and (NM) imply condition (F), but the converse is not true in
general.  If $\Gamma$ is virtually torsionfree and a cocompact manifold model for
$\eub{\Gamma}$ exists, then conditions (M)+(NM) are equivalent to condition (F)
by Lemma~\ref{lem:consequences_of_the_conditions_(M),(NM)_(F)}~\eqref{lem:consequences_of_pseudo-free}. This
provides an intriguing possibility for a negative answer to the Manifold Model Question:
Construct a finite extension $\Gamma$ of a fundamental group of a closed aspherical
manifold which satisfies condition (F), but not (M)+(NM).

\subsection{Assumptions}\label{subsec:Assumptions}

The assumptions of our uniqueness Theorem~\ref{the:uniqueness_intro} and the assumptions
of the uniqueness theorem in~\cite{Connolly-Davis-Khan(2015)} are different and it is
important to compare them so that we can use the results of both.  The main difference is
our assumption of a slice manifold model and their assumption of condition (C\tp i), which
says that there exists a proper cocompact $\Gamma$-manifold $X$ such that
$(X\setminus X^{>1})/\Gamma$ has the $\Gamma$-homotopy type of a finite
$\Gamma$-$CW$-complex and $X$ is $\Gamma$-homotopy equivalent to $\eub{\Gamma}$.

  The following further conditions appear in~\cite{Connolly-Davis-Khan(2015)}.
  The condition (C\tp ii) says that each infinite dihedral
  subgroup of $\Gamma$ lies in a unique maximal infinite dihedral group, Condition (C\tp iii)
  says that $\Gamma$ satisfies the $K$-and $L$-theoretic Farrell-Jones Conjecture with
  coefficients in the ring $\IZ$. Note that condition (C\tp iii) is automatically satisfied,
  if $\pi$ is a Farrell-Jones group.
  The condition (C) says that there is a contractible
  Riemannian manifold of  non-positive sectional curvature with effective cocompact
  proper $\Gamma$-action by isometries. In~\cite[Remark~1.2]{Connolly-Davis-Khan(2015)} it
  is proved that the conditions (F) and  (C) together imply conditions (C\tp i), (C\tp \tp ii),
  and (C\tp iii), actually such $\Gamma$ is a Farrell-Jones group.

  It is obvious that the existence of a  slice manifold model for $\eub{\Gamma}$  implies condition (C\tp i). 
  In~\cite[Lemma~4.2]{Connolly-Davis-Khan(2015)} it is shown
  that there exists a slice manifold model for $\eub{\Gamma}$, provided that
  conditions  (F), (C\tp i)  and (C\tp ii) and (C\tp iii) are satisfied. The conditions
  (F) and (C\tp i) follow from conditions (M) and (NM)
  by Lemma~\ref{lem:consequences_of_the_conditions_(M),(NM)_(F)}.
  Hence   condition (C\tp i) is equivalent to the existence of slice manifold model for $\eub{\Gamma}$,
  provided that conditions (M) and (NM) hold and $\pi$ is a Farrell-Jones group.
  Recall that Theorem~\ref{the:existence_intro} gives conditions for the
  the existence of a slice manifold model for $\eub{\Gamma}$ and that most of the conditions appearing
  there are necessary, see Remark~\ref{rem:some_conditions_are_necessary}.
  Obviously the conditions appearing
  in Theorem~\ref{the:existence_intro} are more accessible  than the condition (C'i).

  \subsection{All models are slice manifold models}%
\label{sec:All_models_are_slice_manifold_models}

  Now suppose that there  exists a slice manifold model $X$ for $\eub{\Gamma}$. Consider any
  cocompact proper $\Gamma$-manifold $M$ such that $M$ is $\Gamma$-homotopy equivalent to
  $\eub{\Gamma}$. Recall that $X$ is a $\Gamma$-$CW$-model for $\eub{\Gamma}$, which is
  pseudo-free, i.e., $X^{>1}$ is zero-dimensional.  Hence also $M$ is pseudo-free by
  Lemma~\ref{lem:consequences_of_the_conditions_(M),(NM)_(F)}~\eqref{lem:consequences_of_pseudo-free}.   We
  conclude from~\cite[Lemma~3.2]{Connolly-Davis-Khan(2015)} that there is an isovariant
  $\Gamma$-homotopy equivalence $f \colon M \to
  X$. Now~\cite[Proposition~4.2]{Connolly-Davis-Khan(2015)} implies that $M$ is a slice manifold
  model for $\eub{\Gamma}$.

\subsection{NRQ and MMQ}\label{subsec:NRQ_and_MMQ}
In the introduction we argued, that in fairly general circumstances, an affirmative answer
to the Manifold Model Question gives an affirmative answer to the Nielsen Realization
Question.  In this subsection we argue, that under quite technical conditions, an
affirmative answer to the Nielsen Realization Question implies an affirmative answer to
the Manifold Model Question.

\begin{proposition}
  Suppose a finite group $G$ acts on a closed aspherical manifold $X$ with fundamental
  group $\pi$.  Let
 \[
 1 \to \pi \to \Gamma \to G \to 1
 \]
 be the associated exact sequence of groups.  If $\Gamma$ satisfies condition (F), namely
 that every non-trivial finite subgroup of $\Gamma$ has finite normalizer, then there is a
 pseudo-free cocompact $\Gamma$-manifold $M$ which has the $\Gamma$-homotopy type of a
 model for $\eub{\Gamma}$.
\end{proposition}

\begin{proof}
  Let $M = \widetilde X$ which is a manifold by with a $\Gamma$-action by hypothesis.  It is cocompact
  since $M/\Gamma$ is a quotient of the compact space $X$.  Proposition 2.3
  of~\cite{Connolly-Davis-Khan(2015)} asserts that the action of $\Gamma$ on $M$
  is pseudo-free,  that the fixed-point set of any finite non-trivial subgroup is a
  point, and the fixed-point set of any infinite subgroup is empty.  Finally, Proposition
  2.5 of~\cite{Connolly-Davis-Khan(2015)} asserts that $M$ has the $\Gamma$-homotopy type
  of a $\Gamma$-CW-complex.
\end{proof}

The conclusion stops short of answering the Manifold Model Question, since it does not
assert that $M$ is a  $\Gamma$-CW-complex.  However, we suspect that using some
algebraic $K$-theory together with the proof of Proposition 2.5
of~\cite{Connolly-Davis-Khan(2015)}, that one could prove that $M$ is a compact manifold
model for $\eub{\Gamma}$ if $G$ has order 2 or 3, in which case the lower Whitehead groups $\Wh_i(G)$ vanish for $i \leq 1$.

\subsection{Hyperbolic manifolds satisfies NRQ and MMQ}\label{subsec:hyperbolic_NRQ_and_MMQ}

We point out that in a special case, geometry gives answers to the Nielsen Realization
Question and the Manifold Model Question.  A \emph{hyperbolic manifold} is a manifold with
constant sectional curvature equal to -1.  The only complete simply connected hyperbolic
$n$-manifold is hyperbolic $n$-space $\mathbb H^n$.

\begin{theorem}
  Let $X$ be a closed hyperbolic $n$-manifold with fundamental group $\pi$.  Any group
  monomorphism $\phi: G \to \Out(\pi)$ can be realized by a map $G \to \Isom(X)$.
\end{theorem}

\begin{proof}
  When $n = 2$, this is a consequence of Kerckhoff's solution of the Nielsen Realization
  Problem~\cite{Kerckhoff(1983)}, so we assume $n > 2$.

  Since $\pi$ contains no rank 2 free abelian subgroup, the center of $\pi$ is trivial.
  Thus, as mentioned in the introduction, group cohomology shows that $\phi$ can be
  realized by a group extension
\[
1 \to \pi \to \Gamma \to G \to 1,
\]
unique up to isomorphism.

We now need to use the Milnor-Schwarz Theorem, the work of Tukia, and the Mostow Rigidity
Theorem.  We will follow the terminology of the book~\cite{Drutu+Kapovich(2018)}.  Recall
that there is an equivalence relation on metric spaces called quasi-isometry,
see~\cite[Definition~8.10]{Drutu+Kapovich(2018)}.  A finite generating set $S$ of a group $G$
defines a metric on the Cayley graph of $(G,S)$, hence on $G$, its set of vertices.  The
isometric inclusion of $G$ into the Cayley graph is a quasi-isometry.  Different finite
generating sets for a group $G$ give quasi-isometric metrics on $G$
(see~\cite[Exercise~7.82]{Drutu+Kapovich(2018)}).

An action of a group $G$ on a metric space $X$ is \emph{geometric} if it is properly
discontinuous, isometric, and cobounded. The Milnor-Schwarz Theorem~\cite[Theorem~8.37]{Drutu+Kapovich(2018)}
states that if a group $G$ acts geometrically on a proper
geodesic metric space $X$, then $G$ is finitely generated and for any $x \in X$, the orbit
map $G \to X,\;  g \mapsto gx$ is a quasi-isometry.

Let $X$ be a closed hyperbolic manifold with fundamental group $\pi$.  Applying the
Milnor-Schwarz Theorem to the inclusion of an orbit gives that $\pi$ is quasi-isometric to
$\widetilde X = \mathbb H^n$.  Since $\pi \subset \Gamma$ is finite index, $\Gamma$ has a
finite generating set $S$. An application of the Milnor-Schwarz Theorem to the
$\pi$-action on the Cayley graph of $(\Gamma, S)$ gives that $\pi$ and $\Gamma$ are
quasi-isometric (see~\cite[Corollary 8.47 (1)]{Drutu+Kapovich(2018)}).

Thus $\Gamma$ is quasi-isometric to $\mathbb H^n$.  A theorem of Tukia (see~\cite[Theorem
23.1]{Drutu+Kapovich(2018)}) asserts that $\Gamma$ acts geometrically on $\mathbb H^n$.
We now have two embeddings of $\pi$ as a discrete cocompact subset of
$\Isom(\mathbb H^n)$, one as the restriction of the action of $\Gamma$ to $\pi$ and the
other as deck transformations associated to the cover $\widetilde X \to X$.  The Mostow
Rigidity Theorem (see~\cite[Theorem 24.15]{Drutu+Kapovich(2018)}) implies that these two
embeddings are conjugate by an element $\alpha$ of $\Isom(\mathbb H^n)$ since $n > 2$.
Thus conjugating the geometric action of $\Gamma$ by $\alpha$ gives an isometric extension
of the $\pi$-action on $\widetilde X$ by deck transformations.  Passing to the $\pi$-orbit
space gives the isometric $G$-action on $X$.
\end{proof}

\begin{corollary}
  If $X$ is a closed hyperbolic manifold with fundamental group $\pi$,
  than the answers to both the Nielsen Realization Question for any
  group monomorphism $\phi : G \to \Out(\pi)$ and the Manifold Model
  Question for a finite normal extension $\pi \subset \Gamma$ are yes.
\end{corollary}

\begin{proof}
  The Nielsen Realization Question is an immediate consequence of the
  previous theorem.  So is the Manifold Model Question, since the
  fixed set of a subgroup of $\Isom(\mathbb H^n)$ is the intersection
  of hyperbolic subspaces, hence a hyperbolic subspace, hence
  contractible.  Hence $\mathbb H^n$ is a cocompact manifold model for
  $\eub{\Gamma}$.
\end{proof}

% %%%%%%%%%%%%%%%%%%%%%%%%%%%%%%%%%%%%%%%%%%%%%%%%%%%%%%%%%%%%%%%%%%%% 
% %%%%%%%%%%%%%%%%%%%%%%%%%%%% Reference  %%%%%%%%%%%%%%%%%%%%%%%%%%%%%%%%
% %%%%%%%%%%%%%%%%%%%%%%%%%%%%%%%%%%%%%%%%%%%%%%%%%%%%%%%%%%%%%%%%%%%%

% \addcontentsline{toc<<}{section}{References} 
% \bibliographystyle{abbrv}
% \bibliography{dbpub,dbpre}

% \version{10.01.2024 (Wolfgang)}

\end{document}